\documentclass[a4paper,reqno,oneside,10pt]{amsart}
\usepackage[utf8]{inputenc}
\usepackage{amsfonts,amsmath,amsthm,eucal,url,cite}
\usepackage{enumitem}
\usepackage{tikz-cd}
\usepackage{hyperref}
\usepackage{multicol}
\usepackage{slashed}

\allowdisplaybreaks[4]

\newcommand{\dist}{\mathrm{dist}}
\newcommand{\diver}{\mathrm{div \,}}
\newcommand{\dd}{\mathrm{d}}
\renewcommand{\Tilde}{\widetilde}
\renewcommand{\Hat}{\widehat}

\newcommand{\RR}{\mathbb{R}}
\newcommand{\CC}{\mathbb{C}}
\newcommand{\NN}{\mathbb{N}}

\newcommand{\U}{\mathcal{U}}

\newcommand{\dom}{\mathrm{dom}}

\def\nn{\mathbf n}
\def\Scal{\mathrm{Scal}}
\def\Ric{\mathrm{Ric}}
\def\R{\mathrm{R}}
\def\Spin{\mathrm{Spin}}
\def\ddt{\frac{\partial}{\partial t}}
\def\Span{\mathrm{Span}}
\def\SS{\mathbf S}
\def\ii{i \,}
\def\CCl{\CC \mathrm{l}}

\newtheorem{theorem}{Theorem}[section]
\newtheorem{prop}[theorem]{Proposition}
\newtheorem{cor}[theorem]{Corollary}
\newtheorem{lemma}[theorem]{Lemma}
\newtheorem{definition}[theorem]{Definition}

\theoremstyle{definition}

\theoremstyle{remark}
\newtheorem{rem}[theorem]{\bf Remark}

\numberwithin{equation}{section} 

\newenvironment{nouppercase}{\renewcommand{\uppercasenonmath}[1]{}}{}

\begin{document}

\title[MIT Bag operator on spin manifolds]{A generalized MIT Bag operator on spin manifolds in the non-relativistic limit}
\author[B. Flamencourt]{Brice Flamencourt}
\address[B. Flamencourt]{Université Paris-Saclay, CNRS, Laboratoire de mathématiques d'Orsay, 91405, Orsay, France.}
\email{brice.flamencourt@universite-paris-saclay.fr}


\begin{nouppercase}
\maketitle
\end{nouppercase}

\begin{abstract}
We consider Dirac-like operators with piecewise constant mass terms on spin manifolds, and we study the behaviour of their spectra when the mass parameters become large. In several asymptotic regimes, effective operators appear: the extrinsic Dirac operator and a generalized MIT Bag Dirac operator. This extends some results previously known for the Euclidean spaces to the case of general spin geometry. 
\end{abstract}

\setcounter{tocdepth}{1}
\tableofcontents

\section{Introduction}

The MIT Bag model was developed by the physicists to describe the behaviour of quarks fields inside hadrons. Mathematically, the hadron is seen as a compact region $\mathcal K$ with smooth boundary of the overall space, where the quarks are supposed to be confined. This could be quantified by saying that the quantum flux through the border of $\mathcal K$ is null, a condition which is satisfied if we add the so-called MIT Bag condition on the boundary of $\mathcal K$ (see \cite{Joh} for the details). Moreover, the quarks fields inside the hadron are Dirac fields, which means they are governed by the Dirac equation.

A Dirac field in the case of the space of dimension 3 is a $\CC^4$-valued function $\psi$ also depending on time, and the Dirac equation takes the form
\begin{equation} \label{Diracequation}
H_m \psi := \left( - \ii \sum\limits_{k = 1}^3 \alpha_k \partial_k + m \beta \right) \psi = \ii \frac{\partial}{\partial t} \psi
\end{equation}
where $\alpha_1, \alpha_2, \alpha_3, \beta \in M_4 (\CC)$ are four Hermitian matrices satisfying the conditions $\alpha_k \alpha_l + \alpha_l \alpha_k = 2 \delta_i^j \mathrm{I}_4$, $\beta^2 = \mathrm{I}_4$ and $\alpha_k$ anti-commutes with $\beta$ for all $k,l \in \lbrace 1,2,3 \rbrace$. In view of this equation, the Dirac operator $H_m$ can be interpreted as an Hamiltonian, and the description of its spectrum is a natural question. Thus, in the context of the MIT Bag model, we are interested in the operator resulting from the combination of $H_m$ restricted to the region $\mathcal K$ together with the MIT Bag boundary condition, namely
\begin{equation} \label{MITop3}
H^\mathcal K_m \psi := H_m \psi, \quad \dom (H^\mathcal K_m) = \lbrace \psi \in H^1(\mathcal K, \CC^4), \, - \ii \beta (\alpha \cdot \nn) \, \psi_{\vert \partial \mathcal K} = \psi_{\vert \partial \mathcal K} \rbrace,
\end{equation}
where $\nn$ is the normal outer vector field over $\partial \mathcal K$. The spectrum of this operator has been investigated in \cite{ALTR16}, where the non-relativistic limit was considered, i.e. the asymptotic regime where the mass goes to infinity. From a physical point of view, this last fact means that the speed of light becomes large, since this constant is hidden in the mass term in \eqref{Diracequation}. It was shown that if we denote by $(\mu_j)_{j \ge 1}$ the non-decreasing  sequence of positive eigenvalues of $H^\mathcal K_m$, one has the asymptotic
\begin{equation}
\mu_j \underset{m \rightarrow - \infty}{=} \Tilde \mu_j^\frac 1 2 + \mathcal O (m^{- \frac 12})
\end{equation}
where $(\Tilde \mu_j)$ is the non-decreasing sequence of eigenvalues of an effective operator acting on the boundary of $\mathcal K$.

In the same framework, the MIT Bag Dirac operator was interpreted as the limit of a Dirac-type operator with a potential corresponding to two masses $m$ and $M$ in the regions $\mathcal K$ and $\mathcal K^c$ respectively \cite{ALTR18}. More precisely, if we define the operator
\begin{equation}
H_{m,M} := H_m + (M - m) \mathbf 1_{\mathcal K^c}, \quad \dom (H{m,M}) := H^1(\RR^3, \CC^4),
\end{equation}
then the eigenvalues of $H_{m,M}$ converge to the corresponding ones of $H^\mathcal K_m$ when $M \rightarrow + \infty$.

In the recent article \cite{MOP}, the case of Euclidean spaces was studied in order to enlarge the precedent results. The expression of the operator in dimension 3 given by \eqref{MITop3} was generalized in dimension $n$ by considering $n+1$ Hermitian matrices $\alpha_1, \ldots, \alpha_{n+1} \in M_N(\CC)$ ($N := 2^{\lfloor \frac{n+1}{2} \rfloor}$) satisfying the Clifford conditions $\alpha_k \alpha_l + \alpha_l \alpha_k = 2 \delta_k^l \mathrm{I}_{n+1}$ and by setting
\begin{equation} \label{operatorDm}
D_m \psi := \left( - \ii \sum\limits_{k=1}^{n+1} \alpha_k \partial_k + m \alpha_{n+1} \right) \psi, \dom(D_m) = H^1(\RR^n, \CC^N).
\end{equation}
This last operator is not the intrinsic Dirac operator in $\RR^n$ but it can be interpreted like in \eqref{Diracequation} as the Hamiltonian appearing in the Dirac equation of a Lorentzian space of dimension $n+1$. From these considerations, the MIT Bag Dirac operator $A_m$ can be defined by
\begin{equation} \label{MITBagEuclidean}
A_m := D_m, \dom(A_m) := \lbrace \psi \in H^1(\mathcal K, \CC^ 4), \, - \ii \alpha_{n+1} \sum\limits_{k=1}^n \nn_k \alpha_k \, \psi_{\vert \partial \mathcal K} = \psi_{\vert \partial \mathcal K} \rbrace.
\end{equation}
With this definition, the result on the convergence of the eigenvalues of $A_m$ still holds, and the effective operator on the boundary can be explicited. Namely, the eigenvalues of $A_m^2$ converge to the eigenvalues of the square of the intrinsic Dirac operator on $\partial \mathcal K$. Moreover, if $n \notin 4 \mathbb Z$, the spectra of the operators are symmetric with respect to the origin, and we recover the result stated in dimension 3.

As for the Minkowski space, the operator $A_m$ can be viewed as the limit of an operator with two masses \cite[Theorem 1.2]{MOP}. This operator is defined in the same way as before:
\begin{equation} \label{BEuclidean}
B_{m,M} := D_m + (M - m) \mathbf 1_{\mathcal K} \alpha_{n+1}, \quad \dom(B_{m,M}) := H^1(\RR^n, \CC^N),
\end{equation}
and the eigenvalues of $B^2_{m,M}$ converge to the eigenvalues of $A^2_m$ when $M \rightarrow + \infty$. In addition, a combination of the two previous asymptotic behaviours is also true \cite[Theorem 1.3]{MOP}: in the asymptotic regime $m \rightarrow - \infty$ and $M \rightarrow + \infty$ with $\frac m M \rightarrow 0$, one has that the eigenvalues of $B^2_{m,M}$ converge to the corresponding ones of the intrinsic Dirac operator on the boundary $\partial \mathcal K$.

In the precedent discussion, the spaces considered where always flat, but the Dirac operator can be defined in a more general setting, for example over a manifold admitting a $\Spin$-structure. Consequently, our aim in the present text is to extend the results of \cite{MOP} to this more general framework. In order to do so, the first step is to understand the geometrical meaning of the operator considered in the MIT Bag model, because we recall that the Dirac operator considered in \cite{MOP} is not the intrinsic Dirac operator of the Euclidean space. Indeed, the operator $D_m$ is the so-called Dirac-Witten operator on $\RR^n$ seen as an hypersurface of $\RR^{n+1}$, plus a mass term which is actually the Clifford multiplication by the vector $\ii m x_{n+1}$ in $\RR^{n+1}$.

Nevertheless, even if the expression \eqref{MITBagEuclidean} is a direct generalization of equation \eqref{MITop3}, the Dirac-Witten operator is not the operator we obtain from the physical model \cite{Joh}. Indeed, in \eqref{Diracequation} we used the alpha matrices, but the Dirac equation is more likely written using the gamma matrices defined by
\[
\gamma^0 := \beta, \quad \gamma^k := - \ii \gamma^0 \alpha_k \: k = 1,2,3.
\]
If one rewrites \eqref{Diracequation} with the $\gamma$ matrices, one obtains
\begin{equation}
H_m \psi = \left(\sum\limits_{k = 1}^3 \gamma^0 \gamma^k \partial_k + m \gamma^0 \right) \psi,
\end{equation}
and this last operator is, up to a change of sign, the extrinsic Dirac operator on the hypersurface $\RR^3$ plus the mass term. Moreover, the boundary condition defined in \cite{ALTR16} by $- \ii \beta (\alpha \cdot \nn) \psi = \psi$ reads $\ii (\gamma \cdot \nn) \psi = \psi$ and this last boundary condition is the MIT Bag boundary condition as introduced in \cite{Joh}.

All together, we have two natural ways of setting the problem in the case of a complete spin manifold $\mathcal N$. In both cases, we have to see $\mathcal N$ as an hypersurface of the Riemannian product $\mathcal C := \mathcal N \times \RR$, and we denote by $\nu$ the outer normal vector field over $\mathcal N$. In addition, the region $\mathcal K$ is now a compact submanifold of $\mathcal N$ with boundary. The theory of $\Spin$-structures restricted to hypersurfaces gives that $\mathcal C$ and $\partial \mathcal K$ are also spin manifolds. Consequently, we can define $\Sigma \mathcal C$ the spinor bundle over $\mathcal C$, and the extrinsic Dirac operator $\mathcal D^\mathcal N$, which acts on spinors of $\mathcal C$ restricted to $\mathcal N$.

From the previous discussion, the obvious generalization of the MIT Bag Dirac operator in the Euclidean spaces \eqref{MITBagEuclidean} is defined as the Dirac-Witten operator on $\mathcal N$ plus a mass term and we add the boundary condition $\ii \nu \cdot \nn \cdot \Psi = \Psi$ on $\partial \mathcal K$. This last condition is not the MIT Bag boundary condition, but the condition associated to a chirality operator, and it is consistent with the condition imposed in \eqref{MITBagEuclidean}. Namely, we have
\begin{equation}
A_m := \nu \cdot \mathcal D^\mathcal N + \ii m \, \nu \cdot, \, \dom (A_m) = \left\lbrace \Psi \in H^1 (\Sigma \mathcal C_{\vert \mathcal K}), \ii \nu \cdot \nn \cdot \Psi = \Psi \textrm{ on $\partial \mathcal K$}\right\rbrace.
\end{equation}

Furthermore, the cylinder $\mathcal C$ can be endowed with a Lorentzian metric such that $\nu$ is a time-like vector, and in this case, solving the Dirac equation in $\mathcal C$ in the same way as for dimension 3 let us with the study of the extrinsic Dirac operator on $\mathcal N$ plus the mass term. The boundary condition imposed in this case is the original MIT Bag boundary condition $\ii \nn \cdot \Psi = \Psi$.

Actually, the two operators we defined this way are unitary equivalent since the manifold $\mathcal N$ is totally geodesic in $\mathcal C$. This last result explains how the operator studied in \cite{MOP} is obtained from the physical model, and the two definitions we gave above are equivalent.

In the same way as before, the two-masses operator is obtained by adding a potential corresponding to two masses in $\mathcal K$ and $\mathcal K^c$ in the expression of the operator $A_m$. Since in our framework the manifold $\mathcal N$ is only complete, $B_{m,M}$ is defined as the closure of the operator
\begin{equation}
\Tilde B_{m,M} := \nu \cdot \mathcal D^\mathcal N + \ii (m \mathbf 1_{\mathcal K} + M \mathbf 1_{\mathcal K^c}) \nu \cdot,
\end{equation}
whose domain is the set of smooth sections with compact support in $\Sigma \mathcal C_{\vert \mathcal N}$. This definition is consistent with \eqref{BEuclidean} because it was shown in \cite{MOP} that the two-masses operator is essentially self-adjoint on the smooth functions with compact support.

The operators $A_m$ and $B_{m,M}$ are self-adjoint and we are interested in the behaviour of the spectrum of $A_m$ when $m \rightarrow - \infty$ and the spectrum of $B_{m,M}$ in the asymptotic regime $M \rightarrow + \infty$ and $\min(-m,M) \rightarrow + \infty$. These limits are the ones studied in \cite{MOP}, and the three main theorems we state below are the counterparts of \cite[Theorems 1.1, 1.2, 1.3]{MOP}.

From now on, we use for $j \in \NN$ and a lower semibounded operator $T$ the notation $E_j(T)$, which stands for the $j$-th eigenvalue of $T$ when counting with multiplicity in the non-decreasing order.

First of all, one has the convergence of the eigenvalues of $A_m^2$ to the eigenvalues of the square of the Dirac operator on $\partial \mathcal K$:

\begin{theorem} \label{main1}
For any $j \in \NN$, one has $E_j(A_m^2) \underset{m \rightarrow - \infty} \longrightarrow E_j \left( (\slashed D^{\partial \mathcal K})^2 \right)$.
\end{theorem}

The two operators $A_m^2$ and $B_{m,M}^2$ are surprisingly related in the asymptotic regime $M \rightarrow + \infty$:

\begin{theorem} \label{main2}
For any $j \in \NN$, there is $M_0 \in \RR$ such that for all $M \ge M_0$, $B^2_{m,M}$ has at least $j$ eigenvalues, and one has $E_j(B_{m,M}^2) \underset{M \rightarrow + \infty} \longrightarrow E_j( A_m^2)$.
\end{theorem}

In addition, one has a combination of these two results:

\begin{theorem} \label{main3}
For any $j \in \NN$, there is $\tau_0 \in \RR$ such that for all $M \ge \tau_0$ and $m \le - \tau$, the operator $B^2_{m,M}$ has at least $j$ eigenvalues, and one has $E_j(B_{m,M}^2) \underset{\min (M,-m) \rightarrow + \infty} \longrightarrow E_j \left( (\slashed D^{\partial \mathcal K})^2 \right)$.
\end{theorem}

Theorem~\ref{main3} is an improvement of \cite[Theorem 1.3]{MOP} since we drop the assumption $\frac m M \rightarrow 0$.

\begin{rem}
We can also look at the operator $A_m^2$ when $m \to + \infty$ and the operator $B_{m,M}^2$ when $m,M \to + \infty$ (or $m,M \to - \infty$). We can prove that in these two cases, the spectrum escapes to infinity (see Remarks~\ref{asym1} and~\ref{asym2} below).
\end{rem}

\subsection*{Organization of the paper} The proofs of the three theorems are really closed to the ones written in \cite{MOP} once we have stated the correct geometrical context. The global strategy is thus to compute sesquilinear forms for the operators $A_m^2$ and $B_{m,M}^2$ in order to find lower bound and upper bound for the limits of the eigenvalues by use of the Min-Max principle.

In section~\ref{Notations} we first recall some fundamental results in spectral theory on the correspondence between self-adjoint operator and sesquilinear forms on Hilbert space. The Min-Max principle, which is the key point of our proof, is stated, and we also give a quick review on the monotone convergence theorem in the case of sesquilinear forms. This last theorem is helpful to find the lower bounds for the limits of the eigenvalues, since it gives a description of the asymptotic domain of the operators. After these preliminaries on operators theory, we introduce the basic tools needed to understand the geometrical context. Indeed, the theory on restriction of spin structure to oriented hypersurfaces of spin manifold plays a significant role in the understanding of the generalized MIT Bag operator.

The section~\ref{SectionOperators} is devoted to the construction of the operators. We develop here the discussion about the two equivalent ways of defining $A_m$. We also define the operator $B_{m,M}$ and we show that it is self-adjoint as a direct consequence of the completeness of $\mathcal N$. The self-adjointness of $A_m$ is more difficult to prove, and we need to compute the sesquilinear form for $A_m^2$ in order to understand its graph norm and its domain. The computations for the forms of square operators is done in section~\ref{SectionForms} and the main tool used to this aim is the Schr\"odinger-Lichnerowicz formula, which gives the expression of the square of the Dirac operator on a spin manifold. Once we get the sesquilinear forms, the graph norm of $A_m$ is shown to be equivalent to the $H^1$ norm on its domain, and we can use the analysis done in \cite{GN} to conclude on self-adjointness.

An important idea to prove the main results is that we can restrict the analysis to a tubular neighbourhood of the boundary of $\mathcal K$. Thanks to this restriction of domain, we only have to understand the operators on a generalized cylinder $\partial \mathcal K \times (-\delta, \delta)$ with $\delta > 0$. However, there is an additional difficulty since we cannot compare the covariant derivatives on the different slices of the cylinder as it is done in \cite{MOP}. Thus, we prove some comparison lemmas in section~\ref{SectionTubular}, where we express the operators in tubular coordinates.

The aim of this restriction is to be able to separate the variables in the generalized cylinder previously introduced. Thus, some one-dimensional operators will appear later in the analysis, and we devote section~\ref{OneDim} to the spectral analysis of these operators, even if a large part of this work has already been done in \cite[Section 3]{MOP}.

In section~\ref{Aform} we prove Theorem~\ref{main1}. The geometrical context is well-defined, and it remains to follow the lines of \cite[Section 4]{MOP}. The proof is done by restricting the analysis to the tubular neighbourhood of $\partial \mathcal K$ intersected with the interior of $\mathcal K$ thanks to the Min-Max principle. Next, an upper bound can be find for the limit by choosing good test functions which are tensorial products between eigenspinors of a model operator on $\partial \mathcal K$ and the first eigenfunction of a one-dimensional operator. The proof of the lower bound relies on the monotone convergence theorem after operating a transformation on the operator in tubular coordinates.

The result stated in Theorem~\ref{main2} is proved in section~\ref{BforM}. We find an appropriate extension operator which sends eigenspinors of $A_m^2$ into $\dom(B_{m,M})$, and this gives the upper bound. The lower bound is once again a consequence of the monotone convergence theorem together with the Min-Max principle.

Finally, we investigate Theorem~\ref{main3} in section~\ref{BformM}. We use a combination of the two precedent proofs. After restricting the problem to the tubular neighbourhood of $\partial \mathcal K$, the upper bound is found in the same way as for Theorem~\ref{main1} by choosing good test functions in the Min-Max principle, and the lower bound is a consequence of the monotone convergence theorem.

\subsection*{Acknowledgements} The author thanks his advisors Andrei Moroianu and Konstantin Pankrashkin for the constant support during the work and their helpful remarks for the improvement of the paper.

\section{Notations and preliminaries.} \label{Notations}

\subsection{About spectral theory.} Let $\mathbf H$ be an infinite-dimensional Hilbert space endowed with the inner product $(\cdot, \cdot)_\mathbf H$. For a self-adjoint and lower semibounded operator $T$ on $\mathbf H$, we denote by $\dom \, T$ its domain, and for any $j \in \NN$,  $E_j(T)$ is the $j$th eigenvalue of $T$, counted with multiplicity in the non-decreasing order. We also note $\sigma(T)$, $\sigma_{ess} (T)$ and $\sigma_d (T)$ the spectrum, the essential spectrum and the discrete spectrum of $T$ respectively.

We denote the adjoint of an operator $T$ by $T^*$ and its closure by $\overline T$.

For a sesquilinear form $t$ in $\mathbf H$, we denote its domain by $\mathcal Q (t)$. There is a one-to-one correspondence between densely defined, closed, symmetric, lower semibounded forms and lower semibounded self-adjoint operators (see \cite[VI, Theorem 2.1]{Kat} for details). For a lower semi-bounded self-adjoint operator $T$, we will denote by $\mathcal Q(T)$ the domain of the associated form. If $T$ and $T'$ are two such operators, and $t,t'$ are the associated forms, we write $T \le T'$ if $\mathcal Q (T') \subset \mathcal Q (T)$ and $t(u, u) \le t'(u, u)$ for all $u \in \mathcal Q (T')$.

For $j \in \NN$, we define the $j$th Rayleigh quotient of the form $t$ by
\begin{equation}
\Lambda_j(t) := \inf_{\substack{V \subset \mathcal Q (t) \\ \dim V = j}} \sup_{\substack{u \in V \setminus \{ 0 \} }} \frac{t(u,u)}{\Vert u \Vert^2_\mathcal H}.
\end{equation}

We recall that if $t$ and $t'$ are two semibounded from below bilinear forms, we write $t \le t'$ if $\mathcal Q(t') \subset \mathcal Q(t)$ and $t(u,u) \le t'(u,u)$ for all $u \in \mathcal Q(t)$.

Let $t$ be a closed symmetric lower semibounded form, and $T$ its associated operator. The well-known Min-Max principle gives a link between the Rayleigh quotients of $t$ and the eigenvalues of $T$. More precisely, we have the following theorem:

\begin{theorem}[Min-Max principle]
Let $\Sigma := \inf \, \sigma_{ess} T$. We are in one of the following cases:
\begin{itemize}
\item[(a)] $\Lambda_j(t) < \Sigma$ for all $j$, $\underset{m \rightarrow + \infty}{\lim} \Lambda_m(t) = \Sigma$ and $E_j(T)=\Lambda_j(t)$ for all $j$.
\item[(b)] $\sigma_{ess} T < + \infty$ and there is $N < + \infty$ such that the interval $(- \infty , \Sigma)$ contains exactly $N$ eigenvalues of $T$ counted with multiplicity and for all $j \le N$, one has $\Lambda_j(t) = E_j(T)$ and $\Lambda_m(t)=\Sigma$ for all $m> N$.
\end{itemize}
\end{theorem}

The proofs of the spectral part of this text will use monotone convergence of operators. The result stated below is a reformulation of \cite[Theorem 4.2]{BHSW}.

\begin{theorem} \label{monotone}
Let $(T_n)_{n \in \NN}$ be a sequence of lower semibounded self-adjoint operators in closed subspaces $(\mathbf H_n)_{n \in \NN}$ of $\mathbf H$, and let $(t_n)_{n \in \NN}$ be the sequence of associated forms. Assume there exists $\gamma \in \RR$ such that $t_n \ge \gamma$ for all $n$ and suppose moreover that the sequence $(t_n)$ (or equivalently $(T_n)$) is non-decreasing.
Then, the form $t_\infty$ defined by
\begin{equation}
\mathcal Q (t_\infty) = \left\lbrace h \in \bigcap\limits_{n \in \NN} \mathcal Q (t_n), \, \lim_{n \rightarrow \infty} t_n(h,h) < \infty \right\rbrace
\end{equation}
and $t_\infty (h,h) = \lim_{n \rightarrow \infty} t_n(h,h)$ for all $h \in \mathcal Q (t_\infty)$ is closed, symmetric, and $t_\infty \ge \gamma$.

Moreover, if $\mathbf H_\infty := \overline{\mathcal Q(t_\infty)}$, one can define the self-adjoint operator $T_\infty$ on $\mathbf H_\infty$ associated to $t_\infty$, and the sequence $(T_n)$ strongly converges to $T_\infty$ in the generalized resolvent sense, i.e. for all $\lambda < \gamma$, one has
\begin{equation}
((T_n - \lambda)^{-1} \oplus 0_{\mathbf H_n^\perp}) h \underset{n \rightarrow \infty}{\longrightarrow} ((T_\infty - \lambda)^{-1} \oplus 0_{\mathbf H_\infty^\perp}) h, \; \forall h \in \mathbf H.
\end{equation} 
\end{theorem}

Since we are interested in the behaviour of the spectrum, we state that in the framework of Theorem~\ref{monotone}, one has actually the convergence of the eigenvalues of $T_n$ to the corresponding eigenvalues of $T_\infty$. To show this, we first recall \cite[Theorem 2.1]{Wei2}:

\begin{theorem} \label{Weidmann2.1}
Let $(T_n)$ be a sequence of self-adjoint operators which are bounded below with $T_n \le T_{n+1}$, strongly converging  to $T$ in the generalized resolvent sense. Assume that the essential spectrum of $T_n$  is contained in $[0, + \infty)$ for all $n \in \NN$. Suppose that $T$ has $j_0$ eigenvalues below zero ($j_0$ might be infinity). It follows that
\begin{align*}
E_j(T_n) \underset{n \rightarrow + \infty}{\longrightarrow} E_j(T) \textrm{ for all $j \le j_0$} \\
\lim_{n \rightarrow + \infty} E_j(T_n)  \ge \eta \textrm{ for all $j > j_0$}.
\end{align*}
Moreover,
\begin{equation*}
\Vert \mathbf 1_{(- \infty, \lambda)} (T_n) - \mathbf 1_{(- \infty, \lambda)} (T) \Vert  \underset{n \rightarrow + \infty}{\longrightarrow} 0 \textrm{ for all $\lambda < 0$}.
\end{equation*}
\end{theorem}

From Theorem~\ref{monotone} and Theorem~\ref{Weidmann2.1} we deduce the following corollary:

\begin{cor} \label{monotonecor}
Let $(T_n)_{n \in \NN}$ and $T_\infty$ be like in Theorem~\ref{monotone}. Assume moreover that $\sigma_{ess} (T_{n_0}) \subset [\eta, + \infty)$ for some $n_0 \in \NN$ and that $T_\infty$ has $j_0$ eigenvalues below $\eta$ ($j_0$ might be infinity). Thus, one has
\begin{equation}
E_j(T_n) \underset{n \rightarrow + \infty}{\longrightarrow} E_j(T) \textrm{ for all $j \le j_0$}
\end{equation}
and
\begin{equation}
\Vert \mathbf 1_{(-\infty, \lambda)} (T_n) - \mathbf 1_{(-\infty, \lambda)} (T_\infty) \Vert \underset{n \rightarrow + \infty}{\longrightarrow} 0, \quad \forall \lambda < \eta.
\end{equation}
\end{cor}
\begin{proof}
We consider for $n \ge n_0$ large enough the bounded self-adjoint operators in $\mathbf H$
\begin{align*}
B_n &:= \frac{1}{\eta - \gamma} - ((T_n - \gamma)^{-1} \oplus 0_{\mathbf H_n^\perp}) \\
B_\infty &:= \frac{1}{\eta - \gamma} - ((T_\infty - \gamma)^{-1} \oplus 0_{\mathbf H_\infty^\perp}).
\end{align*}
From \cite[Proposition 2.2]{BHSW}, it comes that for all $n \ge n_0$, one has $B_n \le B_{n+1} \le B_\infty$. In addition, $\sigma_{ess} (B_n) \subset [0, \frac{1}{\eta - \gamma}]$, $\sigma_{ess} (B_\infty) \subset [0, \frac{1}{\eta - \gamma}]$, and $(B_n)$ converges strongly to $B_\infty$. Thus, Theorem~\ref{Weidmann2.1} gives that for all $j \in \NN$ such that $E_j(B_\infty) < 0$ one has
\begin{equation}
E_j(B_n) \underset{n \rightarrow + \infty}{\longrightarrow} E_j(B_\infty)
\end{equation}
and that for all $t < 0$, there holds
\begin{equation}
\Vert \mathbf 1_{(-\infty, t)} (B_n) - \mathbf 1_{(-\infty, t)} (B_\infty) \Vert \underset{n \rightarrow \infty}{\longrightarrow} 0.
\end{equation}
For $\lambda > \gamma$, we define the strictly increasing function $f(\lambda) := \frac{1}{\eta - \gamma} - \frac{1}{\lambda - \gamma}$. One has $B_n  = f(T_n)$ and $B_\infty = f(T_\infty)$ and we deduce that for all $j \le j_0$
\begin{equation*}
E_j(T_n) \underset{n \rightarrow + \infty}{\longrightarrow} E_j(T) \textrm{ for all $j \le j_0$}
\end{equation*}
and from
\[
\mathbf 1_{(-\infty, f(\lambda))} (B_n) = \mathbf 1_{(-\infty, \lambda)} (T_n), \quad \mathbf 1_{(-\infty, f(\lambda))} (B_\infty) = \mathbf 1_{(-\infty, \lambda)} (T_\infty),
\]
we deduce that for all $\lambda < \eta$
\begin{equation*}
\Vert \mathbf 1_{(-\infty, \lambda)} (T_n) - \mathbf 1_{(-\infty, \lambda)} (T_\infty) \Vert \underset{n \rightarrow \infty}{\longrightarrow} 0. \qedhere
\end{equation*}
\end{proof}

\subsection{Clifford algebra} \label{SectionClifford} We recall here the basic facts about Clifford algebra, and we refer to \cite{BH3M} for the details.

For any $d \in \NN$, the real Clifford algebra $\mathrm{Cl}_d$ is the quotient of the tensorial algebra over $\RR^d$ by the two-sided ideal generated by the elements $x \otimes x + \Vert x \Vert^2 1$. The complex Clifford algebra is defined by $\CCl_d := \mathrm{Cl}_d \otimes_\RR \CC$. The spin group is the subgroup of $\mathrm{Cl}_d$ given by
\[
\Spin_d := \lbrace x_1 \cdot \ldots \cdot x_{2k} \in \mathrm{Cl}_d, \, k \in \NN \textrm{ and } x_j \in \RR^d, \, \Vert x_j \Vert = 1 \textrm{ for all $1 \le j \le 2 k$} \rbrace.
\]

We define the complex volume form as the element of $\CCl_d$
\begin{equation}
\omega^\CC_d := i^{\lfloor \frac{d+1}{2} \rfloor} e_1 \cdot \ldots \cdot e_d
\end{equation}
where $(e_1, \ldots, e_d)$ is any positively-oriented orthonormal frame of $\RR^d$, canonically identified with a basis of $\CC^d$.

If $d$ is even, $\CCl_d$ admits an unique irreducible complex representation $(\rho_d, \Sigma_d)$ where  $\Sigma_d$ is a complex vector space of dimension $2^\frac{d}{2}$. When restricted to the $\Spin$ group, this Clifford module decomposes into $\Sigma_d = \Sigma_d^+ \oplus \Sigma_d^-$ and the representation splits in two irreducible inequivalent representations $(\rho_d^\pm, \Sigma^\pm_d)$. These submodules are characterized by the action of the complex volume form, namely $\omega_d^\CC$ acts as $\pm \textrm{Id}$ on $\Sigma_d^\pm$.

When $d$ is odd, $\CCl_d$ admits two irreducible inequivalent representations over complex vector spaces of dimension $2^\frac{d-1}{2}$. They are characterized by the action of the complex volume form which acts as $\pm \textrm{Id}$. We denote by $(\rho_d, \Sigma_d)$ the representation on which $\omega_d^\CC$ acts as the identity.

\subsection{Notations for manifolds and bundles} \label{notman} In all this text, the manifolds will be considered smooth and paracompact.

Let $\mathcal (\mathcal M, g)$ be a Riemannian manifold of dimension $d+1$, with boundary $\partial \mathcal M$ (possibly empty). If $\mathcal M$ is oriented, we denote by $v_\mathcal M$ the volume form on $\mathcal M$ compatible with the metric. In all this article, the integrations will be done with respect to the measure associated to the elementary Riemannian volume, which coincides with the integration with respect to the volume form $v_\mathcal M$ in the oriented case.

We denote by $\nabla^\mathcal M$ the Levi-Civita connection of $(\mathcal M, g)$ and by $\R^\mathcal M$, $\Ric^\mathcal M$, $\Scal^\mathcal M$ the Riemann curvature tensor, the Ricci tensor, and  the scalar curvature of $\mathcal M$ respectively.

If $E$ is a vector bundle over $\mathcal M$, we denote respectively by $\Gamma(E)$, $\Gamma_c(E)$ and $\Gamma_{cc}(E)$ the smooth sections of $E$, the smooth sections of $E$ with compact support in $\mathcal M$, and the smooth sections of $E$ with compact support in $\mathcal M \setminus \partial \mathcal M$. If moreover $E$ is a Hermitian bundle, we note $L^2(E)$ the space of square integrable sections of $E$. If it is necessary, we will write $L^2(E,v_\mathcal M)$ to specify the measure used for the integration.

We now assume that $\mathcal M$ is oriented. The manifold $\mathcal M$ admit a spin structure if there exists a map $\chi$ and a principal bundle $P_{\Spin_{d+1}} \mathcal M$ of fibre $\Spin$ over $\mathcal M$ such that we have the commutative diagram:
\begin{equation}
\begin{tikzcd}
\Spin_{d+1} \arrow[r, "s \mapsto u s"] \arrow[dd] & P_{\Spin_{d+1}} \mathcal M \arrow[dd, "\chi"] \arrow[rd] \\
& & \mathcal M \\
\mathrm{SO}_{d+1} \arrow[r, "g \mapsto \chi (u) g"] & P_{\mathrm{SO}_{d+1}} \mathcal M \arrow[ru]
\end{tikzcd}
\end{equation}

Then, in the case where $\mathcal M$ is spin, we define the associated complex spinor bundle by $\Sigma \mathcal M := P_{\Spin_{d+1}} \mathcal M \times_{\rho_{d+1}} \Sigma_{d+1}$ where we recall that $(\rho_{d+1}, \Sigma_{d+1})$ is an irreducible representation of the Clifford algebra $\CCl_{d+1}$ as defined in section~\ref{SectionClifford}. 

There is a natural action of the Clifford bundle $\CC \mathcal M := P_{\mathrm{SO}_{d+1}} \times_r \CCl_{d+1}$ (where $r$ is the action of $\mathrm{SO}_{d+1}$ on $\RR^d$ extended to a representation on $\CCl_d$) defined by:
\begin{equation}
[\chi (u), v] ([u, \psi]) := [u, \rho_{d+1} (v) \psi]
\end{equation}
for all $u \in P_{\Spin_{d+1}} \mathcal M$, $v \in \CC l_{d+1}$ and $\psi \in \Sigma_{d+1}$. This action is called the Clifford product and will be denoted by "$\cdot$".

One has a canonical Hermitian product $\langle \cdot,\cdot \rangle$ on $\Sigma \mathcal M$ for which the Clifford product by a unit vector is unitary. Moreover, one obtains a metric connection on $\Sigma \mathcal M$ by lifting the Levi-Civita connection on the orthonormal frame bundle of $\mathcal M$ through the map $\chi$. The covariant derivative obtained this way will still be denoted by $\nabla^\mathcal M$.

We define the intrinsic Dirac operator $\slashed D^\mathcal M$ on $\mathcal M$, by its pointwise expression
\begin{equation}
\slashed D^{\mathcal M} \Psi = \sum\limits_{k = 1}^{d+1} e_k \cdot \nabla^{\mathcal M}_{e_k} \Psi, \; \dom(\slashed D^{\mathcal M}) = \Gamma_c(\Sigma \mathcal M),
\end{equation}
where $(e_1, \ldots,e_{d+1})$ is an orthonormal frame. This definition does not depend on the choice of the frame.

Finally, we remind the Schrödinger-Lichnerowicz formula, which will be a fundamental tool to compute sesquilinear forms of operators. A proof can be found in \cite[Theorem 1.3.8]{Gin}.

\begin{theorem} [Schrödinger-Lichnerowicz formula] \label{S-Lformula}
The Dirac operator $\slashed D^\mathcal M$ satisfies the formula
\begin{equation}
(\slashed D^\mathcal M)^2 = \left(\nabla^\mathcal M \right)^* \nabla^\mathcal M + \frac{\Scal^\mathcal M}{4},
\end{equation}
where $(\nabla^\mathcal M)^*: \Gamma(T^* \mathcal M \otimes \Sigma \mathcal M) \rightarrow \Gamma(\Sigma \mathcal M)$ is the formal adjoint of $\nabla^\mathcal M$.
\end{theorem}

\subsection{Restriction of the spinor bundle to hypersurfaces} We take $(\mathcal M, g)$ as in the previous section.

Let $\mathcal H$ be a smooth oriented hypersurface of $\mathcal M$. Let $\nu$ be the outer unit normal vector field on $\mathcal H$, that is, the only vector field such that if $(e_1,\ldots, e_d)$ is an oriented frame of $\mathcal H$, then $(e_1,\ldots, e_d, \nu)$ is an oriented frame of $\mathcal M$. We define the Weingarten operator of $\mathcal H$ as the endomorphism of $T \mathcal H$ given by
\begin{equation}
W_\mathcal H (X) := -\nabla^{\mathcal M}_X \nu,
\end{equation}
and $H_\mathcal H: \mathcal M \rightarrow \RR$ will be the pointwise trace of this operator.

The hypersurface $\mathcal H$ inherits a spin structure from the one of $\mathcal M$, and we can define the spinor bundle $\Sigma \mathcal H$ (for the details, see \cite[Section 2.4]{BH3M}). This last bundle is endowed with the natural Hermitian product on spinors, still denoted by $\langle \cdot, \cdot \rangle$. The covariant derivative on $\Sigma \mathcal H$ induced by the Levi-Civita connection will be denoted by $\nabla^\mathcal H$. We will also write $\nabla^\mathcal H$ for the covariant derivative on $\Sigma \mathcal H \oplus \Sigma \mathcal H$ (where $\oplus$ stands for the Whitney product), and for all $X \in T \mathcal H$, the Clifford product by $X$ on $\Sigma \mathcal H \oplus \Sigma \mathcal H$ is given by
\begin{equation}
X \cdot (\Psi_1, \Psi_2) := (X\cdot\Psi_1, -X\cdot\Psi_2), \quad \forall (\Psi_1, \Psi_2) \in \Sigma \mathcal H \oplus \Sigma \mathcal H.
\end{equation}

There is a link between the restricted spinor bundle $\Sigma \mathcal M_{\vert \mathcal H}$ and $\Sigma \mathcal H$, given by the following proposition (see \cite[Proposition 1.4.1]{Gin}):

\begin{prop} \label{hypersurface}
Let $\mathcal M$ and $\mathcal H$ be as above. There exists an isomorphism $\zeta$ from $\Sigma \mathcal M_{\vert \mathcal H}$ into $\Sigma \mathcal H$ if $d$ is even and into $\Sigma \mathcal H \oplus \Sigma \mathcal H$ otherwise, and $\zeta$ can be chosen so that it satisfies the following properties:
\begin{enumerate}
\item Let $x \in \mathcal H$, then for all $X \in \Gamma(T_x \mathcal H)$ and all $\Psi \in (\Sigma \mathcal M)_{\vert\lbrace x \rbrace}$, the Clifford product on $\mathcal H$ satisfies
\begin{equation}
X \cdot \zeta(\Psi) = \zeta (X\cdot\nu(x)\cdot\Psi),
\end{equation}
\item The isomorphism $\zeta$ is unitary,
\item For all $\Psi \in \Gamma(\Sigma \mathcal M_{\vert \mathcal H})$ and $X \in T\mathcal H$, 
\begin{equation}
\zeta(\nabla^{\mathcal M}_X \Psi) = \nabla^\mathcal H_X \zeta(\Psi) + \frac{1}{2} W_\mathcal H X\cdot\zeta(\Psi).
\end{equation}
\item For $\Psi \in \Sigma \mathcal M_{\vert \mathcal H}$,
\begin{equation}
\zeta (\ii \nu \cdot \Psi) = \left\lbrace
\begin{array}{cc}
{\left(\begin{matrix} 0 & I_{d} \\ I_{d} & 0 \end{matrix}\right) \zeta(\Psi)} & \textrm{if $d$ is odd} \\
\omega^\CC_{d} \cdot \zeta (\Psi) & \textrm{if $d$ is even}
\end{array}
\right.,
\end{equation}
where the complex volume form $\omega^\CC_d$ was defined in section~\ref{SectionClifford}.
\end{enumerate}
\end{prop}

We can define a covariant derivative $\overline \nabla^\mathcal M$ on $\Sigma \mathcal M_{\vert \mathcal H}$ such that $\overline \nabla^\mathcal M \Psi$ is the restriction of $\nabla^\mathcal M \Psi$ to $\Gamma(T^* \mathcal H \otimes E)$. This notation will be useful as we will often consider the norm of the restricted covariant derivative on hypersurfaces.

The link between $\Sigma \mathcal M_{\vert \mathcal H}$ and $\Sigma \mathcal H$ gives rise to a natural operator called the extrinsic Dirac operator. This is actually the Dirac operator of $\mathcal H$ which acts on the spinor bundle $\Sigma \mathcal M_{\vert \mathcal H}$.
This extrinsic Dirac operator on $\mathcal H$ is the operator acting on $\Gamma_c(\Sigma \mathcal M)$ defined by
\begin{equation} \label{diracext}
\mathcal D^\mathcal H := \zeta^* \slashed D^\mathcal H \zeta \textrm{ if $d$ is odd}, \quad \mathcal D^\mathcal H := \zeta^* \mathbb (\slashed D^\mathcal H \oplus -\slashed D^\mathcal H) \zeta \textrm{ if $d$ is even}.
\end{equation}
where $\zeta$ is the isomorphism given by Proposition~\ref{hypersurface}.
It can be explicitly computed, and its expression at $x \in \mathcal H$ for $\Psi \in \Sigma \mathcal M$ is
\begin{equation} \label{extrinsicexpression}
\mathcal D^\mathcal H \Psi (x) =  \frac{H_\mathcal H (x)}{2} \Psi(x) - \nu(x) \cdot \sum_{k=1}^d e_k \cdot \nabla^{\mathcal M}_{e_k} \Psi(x)
\end{equation}
where $(e_1,\ldots,e_d)$ is an orthonormal frame of $T_x \mathcal H$ \cite[Proposition 1.4.1]{Gin}, \cite{HMZ}.

\subsection{Sobolev spaces on manifolds} Let $(\mathcal M,g)$ be a Riemannian manifold of dimension $d+1$ with boundary $\partial \mathcal M$. We denote by $\nu_\mathcal M$ the normal unit vector field over $\partial \mathcal M$.

Let $(E, \nabla^E, \langle \cdot, \cdot \rangle_E)$ be an Hermitian bundle of dimension $q$ over $\mathcal M$, and assume moreover that $\mathcal  M$ is compact. The construction of the Sobolev spaces on $E$ is done for example in \cite[Definition 3.5]{GN}, but we recall the idea to be self-contained.

In what follows, we will denote by $\exp^\mathcal M$ the Riemannian exponential map on $\mathcal M$ and by $B^\mathcal M_x(r)$ the ball of radius $r > 0$ and of center 0 in $T_x \mathcal M$ where $x \in \mathcal M$. These notations will be used for the boundary $\partial \mathcal M$ with an obvious modification. By the compactness of $\mathcal M$, there is $r_t > 0$ such that:
\begin{itemize}
\item the map
\begin{equation}
F : \partial \mathcal M \times [0, 2 r_t) \ni (x,t) \rightarrow \exp^\mathcal M_x (t \nu_\mathcal M (x))
\end{equation}
is a diffeomorphism on its image;
\item for all $x \in \mathcal M \setminus F(\partial \mathcal M \times [0, 2 r_t))$, $\exp^\mathcal M$ is injective on the open ball of radius $r_t$ of $T_x \mathcal M$;
\item for all $x \in \partial \mathcal M$, $\exp^{\partial \mathcal M}$ is injective on the open ball of radius $r_t$ of $T_x \partial \mathcal M$.
\end{itemize}
Let $(U_j)_{j \in J}$ be a finite covering of $\mathcal M$ such that $U_j = \exp^\mathcal M_x (B^\mathcal M_x (r_t))$ with $x \in \mathcal M \setminus F(\partial \mathcal M \times [0, 2 r_t))$ (Gaussian coordinates) or $U_j = F(B^{\partial \mathcal M}_x(r_t) \times [0,2 r_t))$ with $x \in \partial \mathcal M$ (normal coordinates). The maps given by these charts are denoted by $(f_j)_{j \in J}$. 
We trivialize $E$ over $U_j$ with Gaussian coordinates by identifying $E_x$ with $\CC^q$ and by making parallel transport along the radial geodesics. Over the set $U_j$ with normal coordinates, we trivialize $E$ by identifying $E_x$ with $\CC^q$ and by making parallel transport first along the radial geodesics in $\partial \mathcal M$ and then along the geodesics normal to $\partial \mathcal M$. The trivializations obtained are denoted by $\xi_j$.

Let $(h_j)_{j \in J}$ be a partition of unity adapted to the covering $(U_j)_{j \in J}$. For $s \in \RR$ we define the $H^s$ norm by
\begin{equation} \label{Hsnorm}
\Vert \Psi \Vert^2_{H^s(E)} := \sum\limits_{j \in J} \Vert (\xi_j)_* (h_j \Psi) \circ f_j^{-1} \Vert^2_{H^s(\mathbf R^{d+1}_j,\CC^q)},
\end{equation}
where $\mathbf R^{d+1}_j := \RR^{d+1}$ when $U_j \cap \partial \mathcal M = \emptyset$ and $\mathbf R^{d+1}_j := \RR^d \times \RR^+$ otherwise.

\begin{definition} \label{Hs}
Let $s \in \RR$. The Sobolev space $H^s(E)$ is the completion of the space $\Gamma_c(E)$ for the $H^s$ norm.
\end{definition}

\begin{rem}
The Sobolev spaces defined in this way are a generalization of the $H^s$ spaces in $\RR^{d+1}$, and for $k \in \NN$, the $H^s$ norm is equivalent to the norm defined by the square root of $\sum\limits_{j = 0}^k \Vert (\nabla^E)^j \cdot \Vert^2$ (see \cite[Theorem 5.7]{GS}, or \cite[Remark 3.6]{GN}).
\end{rem}

A direct consequence of Definition~\ref{Hs} is that the intrinsic Dirac operator on a compact manifold without boundary is essentially self-adjoint and the domain of its closure is the Sobolev space $H^1$.

\begin{prop} \label{closure}
If $(\mathcal M, g)$ is a compact Riemannian spin manifold without boundary, $\slashed D^\mathcal M$ is essentially self-adjoint, and the domain of its closure is $H^1(\Sigma \mathcal M)$.
\end{prop}
\begin{proof}
The Dirac operator is symmetric, and then it is closable. By compactness, there exists $C > 0$ such that $\vert \Scal^\mathcal M \vert \leq C$. Moreover, by the Schrödinger-Lichnerowicz formula (Theorem~\ref{S-Lformula}), the graph norm of $\slashed D^\mathcal M$ is equivalent to 
\[(1 + C) \Vert \cdot \Vert^2_{L^2(\mathcal M)} + \Vert \slashed D^\mathcal M \cdot \Vert^2_{L^2(\mathcal M)} = \left(1+ C + \frac{\Scal^\mathcal M}{4} \right) \Vert \cdot \Vert^2_{L^2(\mathcal M)} + \Vert \nabla^\mathcal M \cdot \Vert^2_{L^2(\mathcal M)}
\]
and this last norm is equivalent to the $H^1(\Sigma \mathcal M)$-norm because of the boundedness of $\Scal^\mathcal M$. Then, the domain of the closure of $\slashed D^\mathcal M$ is the completion of $\Gamma_c(\Sigma \mathcal M)$ for the graph norm, which is exactly $H^1(\Sigma \mathcal M)$.

The manifold $(\mathcal M, g)$ is complete because it is compact, and then the Dirac operator is essentially self-adjoint in $L^2(\Sigma \mathcal M)$ \cite[Proposition 1.3.5]{Gin}, which concludes the proof.
\end{proof}

By the definition of the Sobolev spaces, one can observe that it is possible to extend the results valid for Euclidean spaces. We state a trace theorem which is a modification of \cite[Theorem 3.7]{GN}, where we add a bound for the $L^2$-norm of the trace.

\begin{theorem} \label{trace}
Let $(\mathcal M, g)$ be a compact Riemannian manifold with boundary $\partial \mathcal M$. Let $(E,\nabla^E, \langle \cdot, \cdot \rangle_E)$ be an Hermitian vector bundle with base $\mathcal M$.

Then, the pointwise restriction operator $\gamma_\mathcal M : \Gamma_c (E) \rightarrow \Gamma_c (E_{\vert \partial \mathcal M})$ extends to a bounded operator from $H^1(E)$ onto $H^{\frac{1}{2}}(E_{\vert \partial \mathcal M})$, and there is a bounded right inverse to $\gamma_\mathcal M : H^1(E) \rightarrow H^\frac{1}{2}(E_{\vert \partial \mathcal M})$ denoted by $\epsilon_\mathcal M$, which maps $\Gamma_c(E_{\vert \partial \mathcal M})$ into $\Gamma_ c(E)$.
Moreover, there exists $K > 0$ such that for any $\varepsilon \in (0,1)$,
\[
\Vert \gamma_\mathcal M \Psi \Vert^2_{L^2(\partial \mathcal M)} \leq K \left( \varepsilon^{\frac{1}{2}} \Vert \nabla^E \Psi \Vert^2_{L^2(\mathcal M)} + \varepsilon^{-\frac{1}{2}} \Vert \Psi \Vert^2_{L^2(\mathcal M)} \right), \,\Psi \in H^1(E).
\]
\end{theorem}

\begin{proof}
The proof of the first part of the theorem is done in \cite[Theorem 3.7]{GN}. We prove the last estimate.

With the notations of \eqref{Hsnorm}, we denote by $J_N$ the set of all $j \in J$ such that $U_j \cap \partial \mathcal M \neq \emptyset$, and there is a constant $C > 0$ and a constant $\Tilde K > 0$ given by \cite[Theorem 1.5.1.10]{Gri} such that for any $\varepsilon \in (0,1)$ and for all $\Psi \in H^1(E)$
\begin{align*}
\Vert \gamma_\mathcal M \Psi \Vert^2_{L^2(\partial \mathcal M)} \le& C \sum\limits_{j \in J_N} \Vert (\xi_j)_* (h_j \Psi) \circ f_j^{-1} \Vert^2_{L^2(\RR^d \times \{ 0 \},\CC^q)} \\
\le& C \Tilde K \sum\limits_{j \in J} \lbrack \varepsilon^\frac 12 \Vert (\xi_j)_* (h_j \Psi) \circ f_j^{-1} \Vert^2_{H^1(\mathbf R^{d+1}_j,\CC^q)} \\
&+ \varepsilon^{-\frac 12} \Vert (\xi_j)_* (h_j \Psi) \circ f_j^{-1} \Vert^2_{L^2(\mathbf R^{d+1}_j,\CC^q)} \rbrack \\
=& C \Tilde K \left( \varepsilon^{\frac{1}{2}} \Vert \nabla^E \Psi \Vert^2_{L^2(\mathcal M)} + \varepsilon^{-\frac{1}{2}} \Vert \Psi \Vert^2_{L^2(\mathcal M)} \right). \qedhere
\end{align*}
\end{proof}

The Rellich-Kondrachov theorem still holds for the Sobolev spaces on compact manifolds. Consequently, the operators with domain include in the first Sobolev space on a vector bundle with compact basis have compact resolvent. We refer to \cite[Proposition 3.13]{TS2} for the following theorem.

\begin{theorem}[Rellich-Kondrachov-type theorem] \label{RellKond}
Let $E$ be an Hermitian vector bundle over a compact manifold $\mathcal M$. Then, the inclusion $H^1(E) \subset L^2(E)$ is compact.
\end{theorem}

We end this section with a direct consequence of Proposition~\ref{hypersurface}. We assume that $(\mathcal M, g)$ is a compact Riemannian spin manifold of dimension $d+1$ and we take an oriented hypersurface $\mathcal H$ of $\mathcal M$. We use the notation of section~\ref{notman}.

\begin{cor} \label{cor2}
The isomorphism $\zeta$ given by Proposition~\ref{hypersurface} is an isomorphism between $H^1(\Sigma \mathcal M_{\vert \mathcal H})$ and $H^1(\Sigma \mathcal H)$ if $d$ is even or $H^1(\Sigma \mathcal H \oplus \Sigma \mathcal H)$ if $d$ is odd.
\end{cor}
\begin{proof}
We define $\Vert W_\mathcal H \Vert_\infty := \underset{x \in \mathcal H}{\sup} \, \underset{X \in T_x \mathcal H \setminus \{ 0 \} }{\sup}\frac{\vert g(W X, X) \vert}{g(X,X)} < \infty$.
Let $\Psi \in \Gamma_c(\Sigma \mathcal M_{\vert \mathcal H})$ and $(e_1,\ldots,e_{d})$ a local orthonormal frame of $\mathcal H$ at a point $x \in \mathcal H$. At this point, one has, using Proposition~\ref{hypersurface}, (3),
\begin{align*}
\vert \nabla^\mathcal H \zeta\Psi \vert^2 &= \sum\limits_{k=1}^{d}  \vert \zeta \left( \nabla^\mathcal M_{e_k} \Psi \right) - \frac{1}{2} W_\mathcal H e_k \cdot \zeta( \Psi ) \vert^2_{L^2(\mathcal H)} \\
&\leq 2 \vert \zeta(\overline \nabla^\mathcal M \Psi) \vert^2_{L^2(\mathcal H)} + \frac{1}{2} \sum\limits_{k=1}^{d} \vert W_\mathcal H e_k \cdot \nu \cdot \Psi \vert^2_{L^2(\mathcal H)} \\
&\leq 2 \vert \overline \nabla^\mathcal M \Psi \vert^2_{L^2(\mathcal H)} + \frac{d}{2} \left\Vert W_\mathcal H \right\Vert_\infty^2 \left\Vert \Psi \right\vert^2_{L^2(\mathcal H)}
\end{align*}
and then, by integration we obtain
\begin{align*}
\Vert \zeta \Psi \Vert^2_{H^1(\mathcal H)} &= \Vert \zeta \Psi \Vert^2_{L^2(\mathcal H)} + \Vert \nabla^\mathcal H \zeta\Psi \Vert^2_{L^2(\mathcal H)} \\
&\leq \Vert \Psi \Vert^2_{L^2(\mathcal H)} + 2 \Vert \overline \nabla^\mathcal M \Psi \Vert^2_{L^2(\mathcal H)} + \frac{d}{2} \left\Vert W_\mathcal H \right\Vert_\infty^2 \left\Vert \Psi \right\Vert^2_{L^2(\mathcal H)} \\
&\leq C_1 \Vert \Psi \Vert^2_{H^1(\mathcal H)},
\end{align*}
where $C_1 > 0$. The same argument shows that there exists $C_2 > 0$ such that for all $\Psi \in \zeta(\Gamma_c(\Sigma \mathcal M_{\vert \mathcal H}))$, one has $\Vert \zeta^{-1} \Psi \Vert^2_{H^1(\mathcal H)} \leq C_2 \Vert \Psi \Vert^2_{H^1(\mathcal H)}$.
\end{proof}

\section{Definition of the operators} \label{SectionOperators}

\subsection{The generalized MIT Bag Dirac operator} In this section, we would like to give a generalization of the MIT Bag Dirac operator in the context of spin manifolds. Our construction will be done by considering the Riemannian product of a manifold $\mathcal N$ with $\RR$ and interpreting the operator as the extrinsic Dirac operator on the hypersuface $\mathcal N \times \lbrace 0 \rbrace$, modified by a Clifford multiplication by the normal vector field. Since the hypersurface $\mathcal N$ is totally geodesic, this operator is the so-called Dirac-Witten operator.

We first introduce the context of our study. Let $n \in \NN$ and let $(\mathcal N, g)$ be a $n$-dimensional smooth Riemannian manifold which is spin and complete.

Let $(\mathcal C, g_\mathcal C) := (\mathcal N, g) \times (\RR, \dd t^2)$ be the Riemannian product of $\mathcal N$ and $\RR$. We identify $\mathcal N$ with $\mathcal N \times \lbrace 0 \rbrace$. Let $p_1$ be the projection on $\mathcal N$ in $\mathcal C$. We endow $\mathcal C$ with a spin structure as follows: we denote by $P$ the pull-back of the bundle $P_{\Spin_n} \mathcal N$ by the projection $p_1$ on $\mathcal N$, and then the extension of $P$ to $\Spin_{n+1}$ is a spin structure on $\mathcal C$ (see \cite[Section 5]{BGM} for example).

We denote by $\nu$ the outer unit normal vector field on $\mathcal N \times \lbrace 0 \rbrace$ in $\mathcal C$, i.e. the vector field $(0, \ddt)$. By construction, the Weingarten tensor of $\mathcal N$ vanishes, so the mean curvature $H_\mathcal N$ is zero.

We denote by $\iota$ be the isomorphism given by in Proposition~\ref{hypersurface}, in the particular case where $\mathcal M := \mathcal C$ and $\mathcal H := \mathcal N$. It is important to remark that the spin structure originaly defined on $\mathcal N$ and the spin structure inherited by $\mathcal N$ from the one of $\mathcal C$ according to Proposition~\ref{hypersurface} are the same.

Let $\mathcal K$ be a submanifold of $\mathcal N$ with dimension $n$, and assume that $\mathcal K$ is compact with non-empty boundary $\partial \mathcal K$. From these assumptions, we know that $\partial \mathcal K$ is oriented. Thus, we denote by
\[
\mu : \Sigma \mathcal N \rightarrow \begin{cases} \Sigma (\partial \mathcal K) & \textrm{if $n$ is odd} \\ \Sigma (\partial \mathcal K) \oplus \Sigma (\partial \mathcal K) & \textrm{if $n$ is even} \end{cases}
\]
the isomorphism given by Proposition~\ref{hypersurface} and by $\nn$ the unit outer normal vector field over $\partial \mathcal K$ viewed as a submanifold of $\mathcal N$.

It is well-known that the Dirac operator on a Riemannian complete manifold without boundary is essentially self-adjoint \cite[Proposition 1.3.5]{Gin}. Thus, we will still denote by $\mathcal D^\mathcal N$, $\slashed D^\mathcal N$, $\mathcal D^{\partial \mathcal K}$ and $\slashed D^{\partial \mathcal K}$ the closures of these operators.

In all this text, we will simply write $W$ for $W_{\partial \mathcal K}$ and $H$ for $H_{\partial \mathcal K}$.

Let $m \in \RR$. To any $\Psi \in \Gamma(\Sigma \mathcal C_{\vert \mathcal N})$, we associate an element $\hat\Psi_m$ of $\Gamma(\Sigma \mathcal C)$ defined for $(x,t) \in \mathcal C$ by $\hat\Psi_m(x,t) = e^{i m t} \Tilde\Psi(x,t)$ where $\Tilde\Psi(x,t)$ is obtained by parallel transport of $\Psi(x)$ along the curves $s \mapsto (x,s)$.

Let $(e_1,\ldots,e_n)$ be a local orthonormal frame at $x \in \mathcal N$. Thus, we compute
\begin{align*}
(\slashed D^\mathcal C \hat\Psi_m) (x) &= \left(\sum\limits_{j=1}^n e_j \cdot \nabla^{\mathcal C}_{e_j} \hat\Psi_m + \ii m \nu \cdot \hat\Psi_m\right)(x,0) \\
&= \left(-\sum\limits_{j=1}^n \nu \cdot \nu \cdot e_j \cdot \nabla^{\mathcal C}_{e_j} \Psi \right)(x) + \ii m \nu \cdot \Psi(x) \\
&= \nu \cdot \left(\mathcal{D}^\mathcal N + im \right) \Psi (x),
\end{align*}
where the extrinsic Dirac operator $\mathcal D^\mathcal N$ is the operator given by the expression \eqref{extrinsicexpression}.
The operator we obtain in the last line is the operator we want to study, as it can be interpreted as a Dirac operator with a mass.

We remark that the above construction can be done by restricting the domain of the operator to $\mathcal K$, and then we introduce the generalized MIT Bag operator
\begin{equation} \label{tildeA}
\Tilde{A}_m := \nu \cdot \left(\mathcal{D}^\mathcal N + im \right), \; \dom(\Tilde{A}_m) := \left\lbrace \Psi \in \Gamma_c (\Sigma \mathcal C_{\vert \mathcal K}), \ii \nu\cdot \nn \cdot \Psi = \Psi \textrm{ on } \partial \mathcal K \right\rbrace.
\end{equation}

\begin{rem}
One can observe that in the case of euclidean spaces, the expression \eqref{tildeA}  coincides with \cite[Equation (1)]{MOP}, which is already a generalization of the MIT Bag Dirac operator in dimension 3 (see \cite[Equation 1.1]{ALTR18}). Indeed, the only difference comes from the convention on the Clifford multiplication, because we have the identity $X \cdot X = - \vert X \vert^2$.
\end{rem}

\begin{rem} \label{operatorA}
It is easily seen that the operator $\Tilde A_m$ is symmetric since $\nu$ anti-commute with $\mathcal D^\mathcal N$ (see \cite[Proposition 1]{HMZ} for the general case, or simply remark that $\nu$ is parallel in our framework). Since symmetric operators are closable, we denote by $A_m$ its closure.
\end{rem}

Actually, the boundary condition we put in the domain of the operator is not the Lorentzian MIT Bag boundary condition as stated by the physicists \cite{Joh} because of the Clifford multiplication by $\nu$. However, this is consistent with the boundary conditions imposed in \cite{ALTR16}, \cite{ALTR18} and \cite{MOP}. To understand this, we can give another interpretation of the operator $\Tilde{A}_m$ which seems more physical, and appears to give an unitary equivalent operator.

Until the end of this section, we will deal with Clifford algebra and spin structures in the Lorentzian case. We refer to \cite[section 2]{BGM} for a detailed presentation.

One can endow $\mathcal C$ with the Lorentzian metric $g - \dd t^2$. There is a $\Spin_0$-structure over $\mathcal C$ given by the pull-back of the $\Spin$-structure on $\mathcal N$ and extending the fiber. One can construct the associated spinor bundle $\Sigma_L \mathcal C$, whose Clifford multiplication will be denoted by " $\cdot_L$". Moreover, we write $\nabla^L$ for the covariant derivative on $\Sigma_L \mathcal C$, and we denote by $\langle \cdot, \cdot \rangle_L$ the Hermitian product on this spinor bundle. We recall that this inner product is not necessarly definite. In this framework, the Dirac operator with a mass on $\Sigma_L \mathcal C$ admits the pointwise expression
\begin{equation}
\slashed D^\mathcal C_L \Psi := \ii \left(- \nu \cdot_L \nabla^L_\nu \Psi + \sum\limits_{j = 1}^n e_j \cdot_L \nabla^L_{e_j} \Psi \right) - m \Psi
\end{equation}
where $(e_1, \ldots, e_n)$ is any orthonormal frame on $\mathcal N$ (see \cite[section 2]{BGM}). Consequently, the Dirac equation $\slashed D^\mathcal C_L \Psi = 0$ is equivalent to
\begin{equation}
\ii \nabla^L_\nu \Psi = \ii \sum\limits_{j = 1}^n \nu \cdot_L e_j \cdot_L \nabla^L_{e_j} \Psi - m \, \nu \cdot_L \Psi.
\end{equation}
Now, if we take $\Psi(x,t) = e^{i \omega t} \phi(x)$ for all $(x,t) \in \mathcal C$, where $\phi$ is parallel along the time lines, we arrive at
\begin{equation} \label{intermediaire1}
\omega \phi = - \ii \sum\limits_{j = 1}^n \nu \cdot_L e_j \cdot_L \nabla^L_{e_j} \phi + m \, \nu \cdot_L \phi.
\end{equation}

We have the counterpart of Proposition~\ref{hypersurface} for the Lorentzian case. Namely, the spinor bundle $\Sigma_L \mathcal C$ can be identified to one or two copies of $\Sigma \mathcal N$ as in the Riemannian case. 

\begin{prop}
There is an isomorphism $\iota_L$ from $\Sigma_L \mathcal C_{\vert \mathcal N}$ into $\Sigma \mathcal N$ if $n$ is even and into $\Sigma \mathcal N \oplus \Sigma \mathcal N$ if $n$ is odd such that:
\begin{itemize}
\item $\iota_L (- \ii X \cdot_L \nu \cdot_L \Psi) = X \cdot \iota_L \Psi$ for all $X \in T \mathcal N$ and $\Psi \in \Sigma_L \mathcal C$,
\item $\iota_L \nu \cdot_L = \omega_n^\CC \cdot \iota_L$ when $n$ is even, and $\left( \begin{matrix} 0 & \mathrm{Id} \\ \mathrm{Id} & 0 \end{matrix} \right)$ when $n$ is odd.
\item $\langle \iota_L \Psi, \iota_L \Phi \rangle = \langle \Psi, \nu \cdot_L \Phi \rangle_L$ for all $\Phi, \Psi \in \Sigma_L \mathcal C_{\vert \mathcal N}$,
\item $\iota_L \nabla^L_X \Psi = \nabla^\mathcal N_X \iota_L \Psi$ for $X \in T \mathcal N$ and $\Psi \in \Sigma_L \mathcal C_{\vert \mathcal N}$.
\end{itemize}
\end{prop}
\begin{proof}
We recall that the notations for Clifford algebras were introduced in section~\ref{SectionClifford}.

Now, we consider the space $\RR^{n,1}$ endowed with the Lorentzian quadratic form of signature $(n,1)$. Let $(e_1, \ldots, e_{n+1})$ be the canonical basis of $\RR^{n,1}$, so that $e_{n+1}$ is timelike. The Clifford algebra over this Lorentzian space is denoted by $\CCl_{n,1}$. We turn the representation $(\rho_{n+1}, \Sigma_{n+1})$ into a complex representation of $\CCl_{n,1}$ $(\rho_{n,1}, \Sigma_{n+1})$ by setting
\[
\rho_{n,1} (e_i) := \rho_{n+1} (e_i) \textrm{ for $1 \le i \le n$, and } \rho_{n,1} (e_{n+1}) := \ii \rho_{n+1} (e_{n+1}).
\]
We remark that when $n$ is even $i^\frac{n}{2} \rho_{n,1} (e_1 \cdot \ldots \cdot e_{n+1})$ acts as the identity.

Following \cite[section 2]{BGM}, the Hermitian product $\langle \cdot, \cdot \rangle_L$ on $\Sigma_{n+1}$ for the Lorentzian structure is defined for all $\psi, \phi \in \Sigma_{n+1}$ by
\[
\langle \psi, \phi \rangle_L := \langle \psi, \rho_{n,1}(e_{n+1}) \phi \rangle
\]
where $\langle \cdot, \cdot \rangle$ is the natural $\Spin_{n+1}$-invariant Hermitian product on $\Sigma_{n+1}$.

One can define a representation $\rho$ of $\CCl_n$ over the space $\Sigma_{n+1}$ by
\begin{equation*}
\rho(x) = - \ii \rho_{n,1}(x \cdot e_{n+1}) \quad \textrm{for all $x \in \RR^n$}.
\end{equation*}
For $n$ even, this representation is equivalent to $(\rho_n, \Sigma_n)$, so we have an isomorphism $U : \Sigma_{n+1} \rightarrow \Sigma_n$ such that $\rho_n U = U \rho$. Moreover, since $i^\frac{n}{2} \rho_{n,1} (e_1 \cdot \ldots \cdot e_{n+1})$ acts as the identity on $\Sigma_{n+1}$, an easy computation gives $U \rho_{n,1}(e_{n+1}) U^{-1} = \rho_n (\omega_n^\CC)$.

We still denote by $\langle \cdot, \cdot \rangle$ the Hermitian product on $\Sigma_n$ and we remark that $U$ can be chosen unitary for this inner product. Thus, for all $\psi, \phi \in \Sigma_{n+1}$ one has
\[
\langle U \psi, U \phi \rangle = \langle \psi, \phi \rangle = \langle \psi, \rho_{n,1}(e_{n+1})^2 \phi \rangle = \langle \psi, \rho_{n,1}(e_{n+1}) \phi \rangle_L.
\]

For $n$ odd, the restriction of $\rho$ to $\Sigma_{n+1}^+$ is equivalent to $(\rho_n, \Sigma_n)$, so we have an isomorphism $U_0: \Sigma_{n+1}^+ \rightarrow \Sigma_n$ such that $\rho_n U_0 = U_0 \rho$. In addition, $\rho_{n,1} (e_{n+1})$ is an isomorphism from $\Sigma_{n+1}^\pm$ into $\Sigma_{n+1}^\mp$, so we set
\[
U : \Sigma_{n+1} = \Sigma_{n+1}^+ \oplus \Sigma_{n+1}^- \rightarrow \Sigma_n \oplus \Sigma_n, \: U := (U_0 \oplus U_0) (\mathrm{Id} \oplus \rho_{n,1} (e_{n+1})).
\]
Easy computations give $U \rho(x) U^{-1} = \rho_n(x) \oplus - \rho_n(x)$ for all $x \in \RR^n \subset \RR^{n+1}$ and $U \rho_{n,1} (x) U^{-1} (\psi_1, \psi_2) = (\psi_2, \psi_1)$ for all $(\psi_1, \psi_2) \in \Sigma_n \oplus \Sigma_n$.

The Hermitian product on $\Sigma_n$ extends to $\Sigma_n \oplus \Sigma_n$ and this extension is still denoted by $\langle \cdot, \cdot \rangle$. The isomorphism $U$ can be chosen unitary for this inner product, and one has for all $\psi, \phi \in \Sigma_{n+1}$
\[
\langle U \psi, U \phi \rangle = \langle \psi, \phi \rangle = \langle \psi, \rho_{n,1}^2 \phi \rangle = \langle \psi, \rho_{n,1}  \phi \rangle_L.
\]

Now, all these properties transport to manifolds by identifying $e_{n+1}$ with $\nu$ since the $\Spin_0$ structure over $\mathcal C$ is defined by pull-back of the $\Spin$ structure over $\mathcal N$.

The last point follows from the explicit formula of the covariant derivative on spinor \cite[formula 2.5]{BGM} and the fact that $\mathcal N$ is totally geodesic in $\mathcal C$.
\end{proof}

We infer that $\Sigma \mathcal C_{\vert \mathcal N}$ and $\Sigma_L \mathcal C_{\vert \mathcal N}$ are both isomorphic to $\Sigma \mathcal N$ if $n$ is even and to $\Sigma \mathcal N \oplus \Sigma \mathcal N$ if $n$ is odd, so we can identify them via the isomorphism $\iota^{-1} \iota_L$.

\begin{cor} \label{identification1}
The isomorphism $\iota^{-1} \iota_L : \Sigma_L \mathcal C \rightarrow \Sigma  \mathcal C$ satisfies:
\begin{itemize}
\item $\langle (\iota^{-1} \iota_L) \Psi, \ii \nu \cdot (\iota^{-1} \iota_L) \Phi \rangle = \langle \Psi, \Phi \rangle_L$ for all $\Psi, \Phi \in \Sigma_L \mathcal C$.
\item $\nabla^\mathcal C_X (\iota^{-1} \iota_L) \Psi = (\iota^{-1} \iota_L) \nabla^L_X \Psi$ for all $X \in T \mathcal N$ and $\Psi \in \Gamma(\Sigma_L \mathcal C)$.
\item $X \cdot (\iota^{-1} \iota_L) \Psi = (\iota^{-1} \iota_L) (X \cdot_L \Psi)$ for all $X \in T \mathcal N$
\item $\ii \nu \cdot (\iota^{-1} \iota_L) = (\iota^{-1} \iota_L) \nu \cdot_L$.
\end{itemize}
\end{cor}

Under the identification of Corollary~\ref{identification1}, equation~\eqref{intermediaire1} reads
\begin{equation} \label{intermediaire2}
\omega \phi = \sum\limits_{j = 1}^n \nu \cdot e_j \cdot \nabla^\mathcal C_{e_j} \phi + \ii m \, \nu \cdot \phi = (- \mathcal D^\mathcal N + \ii m \, \nu \cdot) \phi.
\end{equation}
This is an eigenvalue equation, and it is now natural to look at the spectrum of the operator defined by the right-hand side. We just need to add a boundary condition to define a generalized MIT Bag operator. Since the physical condition imposed in \cite{Joh} is that the flux $\langle \phi, \nn \cdot_L \phi \rangle_L$ of the quantum field vanishes at the boundary, we consider the MIT Bag boundary condition $\ii \nn \cdot \phi = \phi$. One has
\begin{align*}
- \langle \phi, \phi \rangle_L = \langle \phi, - \ii \nn \cdot_L \phi \rangle_L = \langle \ii \nn \cdot_L \phi, \phi \rangle_L = \langle \phi, \phi \rangle_L,
\end{align*}
and we conclude that $\langle \phi, - \ii \nn \cdot_L \phi \rangle_L = 0$, so the condition of the physical model is verified. We can now define another generalization of the MIT Bag Dirac operator by
\begin{equation} \label{operatorAhat}
\Hat A_m := \mathcal D^\mathcal N + \ii m \, \nu \cdot, \quad \dom (\Hat A_m) = \left\lbrace \Psi \in \Gamma_c (\Sigma \mathcal C_{\vert \mathcal K}), \ii \nn \cdot \Psi = \Psi \right\rbrace.
\end{equation}
The change of sign for the mass in \eqref{operatorAhat} compared to \eqref{intermediaire2} comes from the fact that we consider a model where $m \rightarrow - \infty$ (see \cite[section 1.3.3]{ALTR16} for more explanations).

We have now two candidates operators for the generalization of the MIT Bag Dirac operator. However, one can remark that the difference between $\Tilde A_m$ and $\Hat A_m$ is only a matter of how the Clifford product is defined, and the two operators are unitary equivalent.

\begin{prop} \label{equivalenceofA}
The operators $\Tilde A_m$ and $\Hat A_m$ are unitary equivalent. Moreover, the unitary operator involved is parallel with respect to $\nabla^\mathcal C$.
\end{prop}
\begin{proof}
We define a new Clifford representation on the vector bundle $\Sigma \mathcal C$ by setting $X * \Psi := \nu \cdot X \cdot \Psi$ and $\nu * \Psi := \nu \cdot \Psi$ for $X \in T \mathcal N$ and $\Psi \in \Sigma \mathcal C$. This new product still satisfies the Clifford conditions in each fiber, and when $n$ is even the complex volume form $\omega_{n+1}^\CC$ acts as
\begin{align*}
\omega_{n+1}^\CC * \Psi &= i^{\lfloor \frac{n+2}{2} \rfloor} e_1 * \ldots * e_n * \nu * \Psi \\
&=  i^{\lfloor \frac{n+2}{2} \rfloor} (\nu \cdot e_1) \cdot \ldots (\nu \cdot e_n) \cdot \nu \cdot \Psi = \omega_{n+1}^\CC \cdot \Psi,
\end{align*}
where $(e_1, \ldots, e_n)$ is a direct orthonormal basis of $T \mathcal N$. It follows by the general theory of Clifford representations that there is an unitary isomorphism $U : \Sigma \mathcal C \rightarrow \Sigma \mathcal C$ such that $X \cdot U \Psi = U (X * \Psi)$ for all $X \in T \mathcal C$ and $\Psi \in \Sigma \mathcal C$.

Actually, one can give such an isomorphism explicitly. If $n$ is even, we use the decomposition $\Sigma \mathcal N = \Sigma^+ \mathcal N \oplus \Sigma^- \mathcal N$ (see \cite[Proposition 1.32]{BH3M}) and the pointwise identification $\Sigma \mathcal C_{\vert (x,t)} \cong \Sigma \mathcal N_{\vert x}$ for all $(x,t) \in \mathcal C$ given by Proposition~\ref{hypersurface}. Under this identification, one has
\[
\nu \cdot (\Psi^+, \Psi^-) = (- \ii \Psi^+, \ii \Psi^-), \: X \cdot (\Psi^+, \Psi^-) = \ii (- X \cdot \Psi^-, X \cdot \Psi^+) \textrm{ for all $X \in T \mathcal N$},
\]
and we deduce that $U$ can be defined by
\[
U (\Psi^+, \Psi^-) := (\Psi^+, - \ii \Psi^-).
\]
Indeed, one has for any $X \in T \mathcal N$
\begin{align*}
U (X * (\Psi^+, \Psi^-)) &= U (\nu \cdot X \cdot (\Psi^+, \Psi^-)) = U (\ii \nu \cdot (- X \cdot \Psi^-, X \cdot \Psi^+)) \\
&= - U (X \cdot \Psi^-, X \cdot \Psi^+) = (- X \cdot \Psi^-, \ii X \cdot \Psi^+)
\end{align*}
and
\begin{align*}
X \cdot U (\Psi^+, \Psi^-) = X \cdot (\Psi^+, - \ii \Psi^-) = (- X \cdot \Psi^-, \ii X \cdot \Psi^+),
\end{align*}
thus $U (X * (\Psi^+, \Psi^-)) = X \cdot U (\Psi^+, \Psi^-)$. In addition, $U$ obviously commutes with $\nu$.

In the case where $n$ is odd, one has the pointwise identification $\Sigma \mathcal C_{\vert(x,t)} \cong \Sigma \mathcal N_{\vert x} \oplus \Sigma \mathcal N_{\vert x}$ for all $(x,t) \in \mathcal C$ and under this identification, 
\[
\nu \cdot (\Psi_1, \Psi_ 2) = (- \ii \Psi_2, - \ii \Psi_1), \: X \cdot (\Psi_1, \Psi_2) = \ii (X \cdot \Psi_2, - X \cdot \Psi_1) \textrm{ for all $X \in T \mathcal N$},
\]
It follows that $U$ can be defined by
\[
U(\Psi_1, \Psi_2) := \frac{1}{\sqrt{2}} (\Psi_1 + \ii \Psi_2, \ii \Psi_1 + \Psi_2).
\]
Indeed, for all $X \in T \mathcal N$ one has
\begin{align*}
U (X * (\Psi_1, \Psi_2)) &= \ii U (\nu \cdot (X \cdot \Psi_2, - X \cdot \Psi_1)) = U (- X \cdot \Psi_1, X \cdot \Psi_2) \\
&= \frac{1}{\sqrt{2}} ( X \cdot (-\Psi_1 + \ii \Psi_2), X \cdot (- \ii \Psi_1 + \Psi_2))
\end{align*}
and
\begin{align*}
X \cdot U (\Psi_1, \Psi_2) &= \frac{1}{\sqrt{2}} X \cdot (\Psi_1 + \ii \Psi_2, \ii \Psi_1 + \Psi_2) \\
&= \frac{1}{\sqrt{2}} (X \cdot (-\Psi_1 + \ii \Psi_2), X \cdot (- \ii \Psi_1 + \Psi_2)),
\end{align*}
thus $X \cdot U (\Psi_1, \Psi_2) = U (X * (\Psi_1, \Psi_2))$. Again, $\nu$ commutes with $U$.

In both cases, $U$ is parallel with respect to $\nabla^\mathcal C$ and we remark that $U  (\dom(\Tilde A_m)) = \dom(\Hat A_m)$. We deduce from these considerations that
\begin{equation}
U^* \Hat A_m U \Psi = \Tilde A_m \Psi \qquad \textrm{for all $\Psi \in \dom(\Tilde A_m)$},
\end{equation}
which is the statement we wanted to prove.
\end{proof}

\begin{rem}
The key point to get Proposition~\ref{equivalenceofA} is of course that $H_\mathcal N = 0$. This is only under this condition that the isomorphism $U$ is parallel with respect to $\nabla^\mathcal C$.
Thus, it is equivalent to study one operator or the other, but we wanted to insist on the physical meaning of $\Hat A_m$.
\end{rem}

\subsection{The two-masses Dirac operator} \label{opB} We introduce now an operator that can be interpreted as a Dirac operator on $\mathcal N$ with two masses in the two separated regions $\mathcal K$ and $\mathcal K^c$. The interest of this operator, as we will show later, is that when the mass in $\mathcal K^c$ grows to infinity, its spectrum converges to the spectrum of the MIT Bag Dirac operator.

Let $m, M \in \RR$. We define the operator $\Tilde B_{m,M}$ by
\begin{equation} \label{tildeB}
\Tilde B_{m,M} := \nu \cdot \mathcal D^\mathcal N + \ii (m \mathbf 1_{\mathcal K} + M \mathbf 1_{\mathcal K^c}) \nu \cdot, \; \dom(\Tilde B_{m,M}) := \Gamma_c (\Sigma \mathcal C_{\vert \mathcal N}),
\end{equation}
and since the Clifford multiplication by $\nu$ is an endomorphism of $\Gamma_c (\Sigma \mathcal C_{\vert \mathcal N})$, the range of this operator is included in $\Gamma_c (\Sigma \mathcal C_{\vert \mathcal N})$.

Until the end of this subsection, we differentiate the Dirac operator on complete manifolds and their closure.

The operator $\Tilde B_{m,M}$ is symmetric because $\nu$ anti-commutes with $\mathcal D^\mathcal N$ \cite[Proposition 1]{HMZ} and by Corollary~\ref{ippcor} below. Since the manifold $\mathcal N$ is complete by assumption, the intrinsic Dirac operator on $\mathcal N$ is essentially self-adjoint in $L^2(\Sigma \mathcal C_{\vert \mathcal N})$ \cite[Proposition 1.3.5]{Gin}.  Moreover, \eqref{diracext} gives that $\mathcal D^\mathcal N$ is unitary equivalent to $\slashed D^\mathcal N$ if $n$ is even and $\slashed D^\mathcal N \oplus - \slashed D^\mathcal N$ if $n$ is odd, and the isomorphism $\iota$ sends $\Gamma_c (\Sigma \mathcal C_{\vert \mathcal N})$ into $\Gamma_c (\Sigma \mathcal N)$. Thus, $\mathcal D^\mathcal N$ is essentially self-adjoint, and it is easy to see that its closure still anti-commutes with $\nu$. Using the fact that the Clifford multiplication  by $\nu$ is an unitary isomorphism in $L^2(\Sigma \mathcal C_{\vert \mathcal N})$ we have
\[
(\nu \cdot \overline{\mathcal D^\mathcal N})^* = - \overline{\mathcal D^\mathcal N} \nu \cdot = \nu \cdot \overline{\mathcal D^\mathcal N}, \quad \textrm{and} \quad \overline{\nu \cdot \mathcal D^\mathcal N} = \nu \cdot \overline{\mathcal D^\mathcal N},
\]
so $\overline{\nu \cdot \mathcal D^\mathcal N}$ is self-adjoint.

We conclude that $\Tilde B_{m,M}$ is essentially self-adjoint because the potential is a bounded self-adjoint operator. We define the self-adjoint operator $B_{m,M}$ as the closure of $\Tilde B_{m,M}$.

\section{Sesquilinear forms for the operators with mass} \label{SectionForms}

An important tool for the asymptotic analysis will be the sesquilinear forms associated to the square of the operators. We begin this section by recalling some important facts for integration by part with the Dirac operator. After that, we compute the sesquilinear forms for the operator $A_m^2$ and $B_{m,M}^2$ and we show that $A_m$ is self-adjoint. We end this section with the study of a model operator which will appears naturally in the asymptotic analysis, and we prove that it is unitary equivalent to the square of the Dirac operator on $\partial \mathcal K$.

\subsection{Integration by parts with the Dirac operator} We first give a proof of the well-known result:

\begin{lemma} \label{ipp}
Let $\Psi, \Phi \in \Gamma_c(\Sigma \mathcal N)$. Then, one has the pointwise equality
\[
\langle \slashed D^\mathcal N \Psi, \Phi \rangle =  - \diver V + \langle \Psi,  \slashed D^\mathcal N \Phi \rangle
\]
where $V$ is the complex vector field on $\mathcal N$ defined by 
\[
g(V,X) := \left\langle \Psi, X \cdot \Phi \right\rangle, \; \forall X \in T \mathcal N.
\]
\end{lemma}
\begin{proof}
Let $\Psi, \Phi \in \Gamma_c(\Sigma \mathcal C_{\vert \mathcal N})$, $x \in \mathcal N$ and let $(e_1,\ldots,e_n)$ be a normal coordinate system at $x$ for $\nabla^\mathcal N$, i.e. $\nabla^\mathcal N_{e_i} e_j (x) =0$ for all $i,j \in \lbrace 1,\ldots,n \rbrace$. One has at $x$,
\[
\left\langle \slashed D^\mathcal N \Psi, \Phi \right\rangle = \langle \sum\limits_{j = 1}^n e_j \cdot \nabla^\mathcal N_{e_j} \Psi, \Phi \rangle.
\]
On the other hand, for all $j \in \lbrace 1,\ldots,n \rbrace$,
\begin{align*}
\left\langle e_j \cdot \nabla^\mathcal N_{e_j} \Psi, \Phi \right\rangle &= -\left\langle \nabla^\mathcal N_{e_j} \Psi, e_j \cdot \Phi \right\rangle \\
&= -e_j \left\langle \Psi, e_j \cdot \Phi \right\rangle + \left\langle \Psi, \nabla^\mathcal N_{e_j} (e_j \cdot \Phi) \right\rangle.
\end{align*}
Thus, $\langle \slashed D^\mathcal N \Psi, \Phi \rangle = -\sum\limits_{j = 1}^n e_j \left\langle \Psi, e_j \cdot \Phi \right\rangle + \langle \Psi, \slashed D^\mathcal N \Psi \rangle$.
We recognize in the first term of this last sum the divergence of a complex vector field. To see this, we introduce $V \in \Gamma (T \mathcal N)$ as in the statement of the lemma.
Then, we have at the point $x$
\begin{align*}
\diver V &= \sum\limits_{j=1}^n g(\nabla^\mathcal N_{e_j} V, e_j) = \sum\limits_{j=1}^n e_j \, g(V, e_j) - g(V, \nabla^\mathcal N_{e_j} e_j) \\
&= \sum\limits_{j=1}^n e_j \, g(V, e_j) = \sum\limits_{j=1}^n e_j \left\langle \Psi, e_j \cdot \Psi \right\rangle. \qedhere
\end{align*}
\end{proof}

A direct corollary is an integral version of Lemma~\ref{ipp}.

\begin{cor} \label{ippcor}
One has
\[
\langle \slashed D^\mathcal N \Psi, \Phi \rangle_{L^2(\mathcal K)} = \langle \Psi, \slashed D^\mathcal N \Phi \rangle_{L^2(\mathcal K)} - \int_{\partial \mathcal K} \langle \Psi, \nn \cdot \Phi \rangle v_{\partial \mathcal K}
\]
for all $\Psi, \Phi \in H^1(\Sigma \mathcal K)$,
and
\[
\langle \mathcal D^\mathcal N \Psi, \Phi \rangle_{L^2(\mathcal K)} = \langle \Psi, \mathcal D^\mathcal N \Phi \rangle_{L^2(\mathcal K)} - \int_{\partial \mathcal K} \langle \Psi, \nn \cdot \nu \cdot \Phi \rangle v_{\partial \mathcal K}
\]
for all $\Psi, \Phi \in H^1(\Sigma \mathcal C_{\vert \mathcal K})$.
\end{cor}
\begin{proof}
The first identity is  proved by integrating the formula obtained in Lemma~\ref{ipp} for $\Psi, \Phi \in \Gamma_c(\Sigma \mathcal C_{\vert \mathcal K})$ and using the divergence theorem. We conclude by density.
For the second one, we use the definition of the extrinsic Dirac operator given by \eqref{diracext} together with the first equation.
\end{proof}

Finally, we obtain an integration by part formula for the Dirac operator with a mass defined in the precedent section.

\begin{cor} \label{ippA}
For any $\Psi, \Phi \in H^1( \Sigma \mathcal C_{\vert \mathcal K})$, one has
\[
\left\langle \nu \cdot (\mathcal D^\mathcal N + \ii m) \Psi, \Phi \right\rangle_{L^2(\mathcal K)} = \left\langle \Psi, \nu \cdot (\mathcal D^\mathcal N + \ii m) \Phi \right\rangle_{L^2(\mathcal K)} + \int_{\partial \mathcal K} \left\langle \Psi, \nn \cdot \Phi \right\rangle v_{\partial \mathcal K}.
\]
\end{cor}
\begin{proof}
Let $\Psi, \Phi \in H^1(\Sigma \mathcal C_{\vert \mathcal K})$, using Corollary~\ref{ippcor} one has
\begin{align*}
\left\langle \nu \cdot (\mathcal D^\mathcal N +\ii m) \Psi, \Phi \right\rangle_{L^2(\mathcal K)} =& -\left\langle (\mathcal D^\mathcal N +\ii m) \Psi, \nu \cdot \Phi \right\rangle_{L^2(\mathcal K)} \\
=& -\left\langle \Psi, (\mathcal D^\mathcal N - \ii m) (\nu \cdot \Phi) \right\rangle_{L^2(\mathcal K)} \\
&- \int_{\partial \mathcal K} \left\langle \Psi, \nn \cdot \nu \cdot \nu \cdot \Phi \right\rangle v_{\partial \mathcal K} \\
=& \left\langle \Psi, \nu \cdot (\mathcal D^\mathcal N +\ii m) \Phi \right\rangle_{L^2(\mathcal K)} + \int_{\partial \mathcal K} \langle \Psi, \nn \cdot \Phi \rangle v_{\partial \mathcal K}. \qedhere
\end{align*}
\end{proof}

\subsection{Sesquilinear form for $\Tilde A_m^2$ and essential self-adjointeness} In this section we show that the operator $\Tilde A_m$ is actually essentially self-adjoint, and the domain of its closure is an extension of $\dom(\Tilde A_m)$ to the space $H^1(\Sigma \mathcal C_{\vert \mathcal K})$. The proof of this fact is done in two steps. First, we compute the sesquilinear form of $\Tilde A_m^2$ to get the domain of the closure and secondly, we show the essential self-adjointeness following the analysis of \cite{GN}.

With Corollary~\ref{ippA}, we see that $\Tilde A_m$ is symmetric since for any $\Psi, \Phi \in \dom(\Tilde A_m)$ one has
\[
\langle \Psi, \nn \cdot \Phi \rangle = \langle \Psi, \ii \nu \cdot \Phi \rangle = \langle \ii \nu \cdot \Psi, \Phi \rangle = \langle \nn \cdot \Psi, \Phi \rangle = - \langle \Psi, \nn \cdot \Phi \rangle = 0.
\]

\begin{prop} \label{quadA}
For all $\Psi \in \mathcal \dom(\Tilde A_m)$,
\begin{align*}
\Vert \Tilde A_m \Psi \Vert^2_{L^2(\mathcal K)} =&  \int_\mathcal K \left( \vert \nabla^\mathcal N (\iota \Psi) \vert^2 + \frac{\Scal^\mathcal N}{4} \vert \Psi \vert^2 \right) v_\mathcal N \\
&+ m^2 \Vert \Psi \Vert^2_{L^2(\mathcal K)} + \int_{\partial \mathcal K} \left(m - \frac{H}{2} \right) \vert \Psi \vert^2 v_{\partial \mathcal K}. \qedhere
\end{align*}
Moreover, the graph norm of $\Tilde A_m$ and the $H^1$-norm on are equivalent on $\dom (\Tilde A_m)$.
\end{prop}

\begin{proof}
We recall that $\dom(\Tilde A_m)$ was defined in \eqref{tildeA}. Let $\Psi \in \dom(\Tilde{A}_m)$. With Corollary~\ref{ippcor} one has
\begin{align*}
\Vert \Tilde A_m \Psi \Vert^2_{L^2(\mathcal K)} =& \left\langle (\mathcal D^\mathcal N + \ii m) \Psi, (\mathcal D^\mathcal N + \ii m) \Psi \right\rangle_{L^2(\mathcal K)} \\
=& \Vert \mathcal D^\mathcal N \Psi \Vert^2_{L^2(\mathcal K)} + m^2 \Vert \Psi \Vert^2_{L^2(\mathcal K)} + m \left\langle \mathcal D^\mathcal N \Psi, \ii \Psi \right\rangle_{L^2(\mathcal K)} \\
&+ m \left\langle \ii \Psi, \mathcal D^\mathcal N \Psi \right\rangle_{L^2(\mathcal K)} \\
=& \Vert \mathcal D^\mathcal N \Psi \Vert^2_{L^2(\mathcal K)} + m^2 \Vert \Psi \Vert^2_{L^2(\mathcal K)} - m \int_{\partial \mathcal K} \left\langle \Psi, \ii \nn \cdot \nu \cdot \Psi \right\rangle v_{\partial \mathcal K} \\
=& \Vert \mathcal D^\mathcal N \Psi \Vert^2_{L^2(\mathcal K)} + m^2 \Vert \Psi \Vert^2_{L^2(\mathcal K)} + m \int_{\partial \mathcal K} \vert \Psi \vert^2 v_{\partial \mathcal K},
\end{align*}
where we used the property $\Psi = \ii \nu \cdot \nn \cdot \Psi$ on $\partial \mathcal K$.

We consider the operator ${\Tilde {\mathcal D}}^{\partial \mathcal K} := \mathcal D^{\partial \mathcal K}$ if $n$ is even and ${\Tilde {\mathcal D}}^{\partial \mathcal K} := \mathcal D^{\partial \mathcal K} \oplus \mathcal D^{\partial \mathcal K}$ if $n$ is odd. From \cite[Formula (13)]{HMZ} we have for all $\Phi \in \Gamma(\Sigma \mathcal K)$
\begin{equation*}
\begin{aligned}
\int_\mathcal K \vert \slashed D^\mathcal N \Phi \vert^2 v_\mathcal N =& \int_\mathcal K \left( \vert \nabla^\mathcal N \Phi \vert^2 + \frac{\Scal^\mathcal N}{4} \vert \Phi \vert^2 \right) v_\mathcal N \\ &+ \int_{\partial \mathcal K} \left( -\frac{H}{2} \vert \Phi \vert^2 - \left\langle \mathcal D^{\partial \mathcal K} \Phi, \Phi \right\rangle \right) v_{\partial \mathcal K}.
\end{aligned}
\end{equation*}

Using this equation together with the definition of the extrinsic Dirac operator \eqref{diracext}, one has
\begin{equation} \label{13}
\begin{aligned}
\int_\mathcal K \vert \mathcal D^\mathcal N \Psi \vert^2 v_\mathcal N =& \int_\mathcal K \left( \vert \nabla^\mathcal N (\iota \Psi) \vert^2 + \frac{\Scal^\mathcal N}{4} \vert \Psi \vert^2 \right) v_\mathcal N \\ &+ \int_{\partial \mathcal K} \left( -\frac{H}{2} \vert \Psi \vert^2 + \left\langle {\Tilde {\mathcal D}}^{\partial \mathcal K} (\iota \Psi), \iota \Psi \right\rangle \right) v_{\partial \mathcal K}.
\end{aligned}
\end{equation}
On the other hand, as ${\Tilde {\mathcal D}}^{\partial \mathcal K}$ anti-commute with the Clifford multiplication by $\nn$ \cite[Proposition 1]{HMZ},
\begin{align*}
\left\langle {\Tilde {\mathcal D}}^{\partial \mathcal K} (\iota \Psi), \iota \Psi \right\rangle &= \left\langle {\Tilde {\mathcal D}}^{\partial \mathcal K} (\iota (-\ii \nn \cdot \nu \cdot \Psi)), \iota \Psi \right\rangle = \left\langle -\ii {\Tilde {\mathcal D}}^{\partial \mathcal K} \nn \cdot (\iota \Psi), \iota \Psi \right\rangle \\
&= \left\langle \ii \nn \cdot {\Tilde {\mathcal D}}^{\partial \mathcal K} (\iota \Psi), \iota \Psi \right\rangle = \left\langle {\Tilde {\mathcal D}}^{\partial \mathcal K} (\iota \Psi), \ii \nn \cdot (\iota \Psi) \right\rangle \\
&= \left\langle {\Tilde {\mathcal D}}^{\partial \mathcal K} (\iota \Psi), -\iota (\ii \nu \cdot \nn \cdot\Psi) \right\rangle = -\left\langle {\Tilde {\mathcal D}}^{\partial \mathcal K} (\iota \Psi), \iota \Psi \right\rangle
\end{align*}
and we deduce that $\left\langle {\Tilde {\mathcal D}}^{\partial \mathcal K} (\iota \Psi), \iota \Psi \right\rangle = 0$.

Finally, using this equation together with \eqref{13}, we get
\begin{align*}
\Vert \Tilde A_m \Psi \Vert^2_{L^2(\mathcal K)} =&  \int_\mathcal K \left( \vert \nabla^\mathcal N (\iota \Psi) \vert^2 + \frac{\Scal^\mathcal N}{4} \vert \Psi \vert^2 \right) v_\mathcal N \\
&+ m^2 \Vert \Psi \Vert^2_{L^2(\mathcal K)} + \int_{\partial \mathcal K} \left(m - \frac{H}{2} \right) \vert \Psi \vert^2 v_{\partial \mathcal K}.
\end{align*}
It remains to prove the equivalence of the norms. As $\mathcal K$ is a compact manifold with boundary, Theorem~\ref{trace} applies and there is $C_1 > 0$ such that for all $\Psi \in \dom(\Tilde{A}_m)$,
\begin{align*}
\Vert \Psi \Vert^2_{L^2(\mathcal K)} + \Vert \Tilde A_m\Psi \Vert^2_{L^2(\mathcal K)} =& \Vert \iota \Psi \Vert^2_{L^2(\mathcal K)} + \int_\mathcal K \left( \vert \nabla^\mathcal N (\iota \Psi) \vert^2 + \frac{\Scal^\mathcal N}{4} \vert \iota \Psi \vert^2 \right) v_\mathcal N \\
&+ m^2 \Vert \iota \Psi \Vert^2_{L^2(\mathcal K)} + \int_{\partial \mathcal K} \left(m - \frac{H}{2} \right) \vert \iota \Psi \vert^2 v_{\partial \mathcal K} \\
\leq&  C_1 \Vert \iota \Psi \Vert^2_{L^2(\mathcal K)} + \Vert \nabla^\mathcal N (\iota \Psi) \Vert^2_{L^2 (\mathcal K)} + C_1 \Vert \iota \Psi \Vert^2_{H^1(\mathcal K)} \\
\leq& 2 (C_1 + 1) \Vert \iota \Psi \Vert^2_{H^1(\mathcal K)}.
\end{align*}
Moreover, using Theorem \ref{trace} with $\varepsilon$ small enough, there exists a constant $C_2 > 0$ such that
\[
\Vert \Psi \Vert^2_{L^2(\mathcal K)} + \Vert \Tilde A_m\Psi \Vert^2_{L^2(\mathcal K)} \geq  C_2 \Vert \iota \Psi \Vert^2_{H^1(\mathcal K)}.
\]
Thus, the graph norm is equivalent to the $H^1(\iota (\Sigma \mathcal C_{\vert \mathcal K}))$ norm, which is equivalent to the $H^1(\Sigma \mathcal C_{\vert \mathcal K})$ norm thanks to Corollary~\ref{cor2}.
\end{proof}

Now, we show that $A_m$ is self-adjoint. For this purpose, it is sufficient to prove that $\nu \cdot \mathcal D^\mathcal N$ is essentially self-adjoint on $\dom(\Tilde A_m)$ because the potential is a bounded operator. In addition, one has
\begin{equation}
\iota^{-1} (\nu \cdot \mathcal D^\mathcal N) \iota = - \ii \omega^\CC_n \cdot \slashed D^\mathcal N \quad \textrm{if $n$ is even,}
\end{equation}
and
\begin{equation}
\iota^{-1} (\nu \cdot \mathcal D^\mathcal N) \iota = - \ii \left( \begin{matrix} 0 & \mathrm{Id} \\ \mathrm{Id} & 0 \end{matrix} \right) (\slashed D^\mathcal N \oplus - \slashed D^\mathcal N) \quad \textrm{if $n$ is odd.}
\end{equation}
Having these considerations in mind, we define
\begin{equation}
A := \slashed D^\mathcal N \textrm{ if $n$ is even, } A := \slashed D^\mathcal N \oplus - \slashed D^\mathcal N \textrm{ is $n$ is odd,}
\end{equation}
and
\begin{equation}
T := - \ii \omega^\CC_n \cdot \textrm{ if $n$ is even, } T := - \ii \left( \begin{matrix} 0 & \mathrm{Id} \\ \mathrm{Id} & 0 \end{matrix} \right) \textrm{ if $n$ is odd.}
\end{equation}
We remark that $T$ is an unitary skew-Hermitian operator which anti-commutes with $A$.

We define the operators
\begin{equation} \label{Ppm}
P_\pm  := \frac{1 \pm \ii \nn \cdot}{2} \textrm{ on $\iota (\Sigma \mathcal C_{\vert \mathcal K})$, and } \mathcal P_\pm := \frac{1 \pm \ii \nu \cdot \nn \cdot}{2} \textrm{ on $\Sigma \mathcal C_{\vert \mathcal K}$}.
\end{equation}

Let $A_\pm$ be the restriction of $A$ to the domain $\lbrace \Psi \in \Gamma_c(\Sigma \mathcal C_{\vert \mathcal K}), \, P_\pm \Psi = 0 \rbrace$. Then, whatever is the parity of $n$, $\nu \cdot \mathcal D^\mathcal N$ with domain $\dom(\Tilde A_m)$ is unitary equivalent to $T A_+$.

\begin{lemma} \label{Pbounded}
For any $s \in \RR$, $P_\pm$ and $\mathcal P_\pm$ define bounded operators from $H^s$ to itself.
\end{lemma}
\begin{proof}
The proof is straightforward, see \cite[Lemma 5.1 (ii)]{GN}.
\end{proof}

\begin{theorem} \label{Aselfadjoint}
The operator $A_m$ is self-adjoint, and the expression of Proposition~\ref{quadA} is true for any $\Psi \in \dom(A_m) = \left\lbrace \Psi \in H^1(\Sigma \mathcal C_{\vert \mathcal K}), \mathcal P_- \Psi = 0 \right\rbrace$.
\end{theorem}
\begin{proof}
We first prove that $E := \left\lbrace \Psi \in \Gamma_c(\Sigma \mathcal C_{\vert \mathcal K}), \mathcal P_- \Psi = 0 \right\rbrace$ is dense in $F := \left\lbrace \Psi \in H^1(\Sigma \mathcal C_{\vert \mathcal K}), \mathcal P_- \Psi = 0 \right\rbrace$ for the $H^1$ norm. Let $\Psi \in F$. There exists $(\Psi_j)_{j \in \NN}$ a sequence of $\Gamma_c(\Sigma \mathcal C_{\vert \mathcal K})$ converging to $\Psi$ in the $H^1$ norm. Let $\Phi_j := \Psi_j - \epsilon_\mathcal K \mathcal P_- \gamma_\mathcal K \Psi_j$, where we recall that $\epsilon_\mathcal K$ is the extension operator defined in Theorem~\ref{trace}. One has $\mathcal P_- \gamma_\mathcal K \Phi_j = 0$ and with Theorem~\ref{trace} and Lemma~\ref{Pbounded} it comes
\begin{align*}
\Vert \Phi_j - \Psi \Vert_{H^1(\mathcal K)} &= \Vert \Psi_j - \epsilon_\mathcal K \mathcal P_- \gamma_\mathcal K \Psi_j - \Psi \Vert_{H^1(\mathcal K)} \\
&\leq \Vert \Psi_j - \Psi \Vert_{H^1(\mathcal K)} + \Vert \epsilon_\mathcal K \mathcal P_- \gamma_\mathcal K \Psi_j \Vert_{H^1(\mathcal K)} \\
&\leq \Vert \Psi_j - \Psi \Vert_{H^1(\mathcal K)} + C_1 \Vert \mathcal P_- \gamma_\mathcal K \Psi_j - \mathcal P_- \gamma_\mathcal K \Psi \Vert_{H^\frac{1}{2} (\mathcal K)} \\
&\leq C_2 \Vert \Psi_j - \Psi \Vert_{H^1(\mathcal K)} \underset{j \rightarrow +\infty}{\longrightarrow} 0
\end{align*}
with $C_1, C_2 > 0$.

Then, $E$ is dense in $F$, and as the graph norm of $\Tilde A_m$ and the $H^1$ norm are equivalent on E by Proposition~\ref{quadA}. We conclude that $F \subset \dom(A_m)$. By density, the expression of Proposition~\ref{quadA} is true for any $\Psi \in F$, and the graph norm and the $H^1$ norm are still equivalent on $F$. But F is closed for the $H^1$ norm, so we deduce that $F = \dom(A_m)$, and using Corollary~\ref{cor2}, we have $\dom(\overline{A_+}) = \left\lbrace \Psi \in H^1(\iota \Sigma \mathcal C_{\vert \mathcal K}), P_+ \Psi = 0 \right\rbrace$. This means that $\overline{A_+}$ is exactly one or two copies of the operator $D_+$ (up to a sign) studied in \cite[Lemma 5.1]{GN}.

By the same method, we can show that $\dom(\overline{A_-}) = \left\lbrace \Psi \in H^1(\iota \Sigma \mathcal C_{\vert \mathcal K}), P_- \Psi = 0 \right\rbrace$ and $\overline{A_-}$ is one or two copies of the operator $D_-$ (up to a sign) studied in \cite[Lemma 5.1]{GN}.

Finally, \cite[ Lemma 5.1 ($v$)]{GN} gives us $(\overline{A_\pm})^*=\overline{A_\mp}$, and we deduce that
\[
(T \overline{A_+})^* = - (\overline{A_+})^* T = - \overline{A_-} T = T \overline{A_+}.
\]
Consequently, $T {A_-}$ is self-adjoint, and so is $A_m$ by unitary equivalence.
\end{proof}

\subsection{Sesquilinear form for $B_{m,M}^2$} As for the operator $A_m$, we compute the sesquilinear form of the operator $B_{m,M}^2$ defined in section~\ref{opB}. As a consequence of the Schr\"odinger-Lichnerowicz formula, we can first compute the square of the extrinsic Dirac operator acting on smooth section with compact support in $\mathcal N$.

\begin{lemma}
Let $\Psi \in \Gamma_c(\Sigma \mathcal C_{\vert \mathcal N})$. Then
\begin{align*}
\Vert \nu \cdot \left(\mathcal{D}^\mathcal N + \ii m\right) \Psi \Vert^2_{L^2(\mathcal N)} =& \int_\mathcal N \left[ \vert \nabla^\mathcal N (\iota \Psi) \vert^2 + \frac{\Scal^\mathcal N}{4} \vert \Psi \vert^2 + m^2 \vert \Psi \vert^2 \right] v_\mathcal N.
\end{align*}
\end{lemma}

\begin{proof}
Let $\Psi \in \Gamma_c(\Sigma \mathcal C_{\vert \mathcal N})$. One has

\begin{align*}
\Vert \nu \cdot \left(\mathcal{D}^\mathcal N + \ii m\right) \Psi \Vert^2_{L^2(\mathcal N)} =& \left\langle \nu \cdot \left(\mathcal{D}^\mathcal N + \ii m\right) \Psi, \nu \cdot \left(\mathcal{D}^\mathcal N + im\right) \Psi \right\rangle_{L^2(\mathcal N)} \\
=& \left\langle\left(\mathcal{D}^\mathcal N + im\right)\Psi, \left(\mathcal{D}^\mathcal N + im\right) \Psi \right\rangle_{L^2(\mathcal N)} \\
=& \left\langle \mathcal{D}^\mathcal N \Psi, \mathcal{D}^\mathcal N \Psi \right\rangle_{L^2(\mathcal N)} + m^2 \left\langle \Psi, \Psi \right\rangle_{L^2(\mathcal N)} \\
&+ m \left[ \left\langle \mathcal{D}^\mathcal N \Psi, \ii \Psi \right\rangle_{L^2(\mathcal N)} + \left\langle \ii \Psi,  \mathcal{D}^\mathcal N \Psi \right\rangle_{L^2(\mathcal N)} \right].
\end{align*}
Using Lemma~\ref{ipp}, one has at any point $x \in \mathcal N$,
\begin{align*}
\left\langle  \mathcal{D}^\mathcal N \Psi, \ii \Psi \right\rangle + \left\langle \ii \Psi,  \mathcal{D}^\mathcal N \Psi \right\rangle &= - \diver \, V.
\end{align*}
By the divergence theorem, the Schr\"odinger-Lichnerowicz formula (Proposition~\ref{S-Lformula}) and Equation~\ref{diracext}, one can integrate over $\mathcal N$ to obtain
\begin{align*}
\Vert \nu \cdot \left(\mathcal{D}^\mathcal N + \ii m\right) \Psi \Vert^2_{L^2(\mathcal N)} &= \left\langle \mathcal{D}^\mathcal N \Psi, \mathcal{D}^\mathcal N \Psi \right\rangle_{L^2(\mathcal N)} + m^2 \left\langle \Psi, \Psi \right\rangle_{L^2(\mathcal N)} \\
&= \int_\mathcal N \left[ \vert \nabla^\mathcal N (\iota \Psi) \vert^2 + \frac{\Scal^\mathcal N}{4} \vert \Psi \vert^2 + m^2 \vert \Psi \vert^2 \right] v_\mathcal N. \qedhere
\end{align*}
\end{proof}

We can now compute the quadratic form for the operator by integration over $\mathcal N$, and it comes out that the domain of $B_{m,M}$ is a subspace of the Sobolev space $H^1$.

\begin{prop} \label{quadB}
One has $\dom(B_{m,M}) \subset H^1(\Sigma \mathcal C_{\vert \mathcal N})$ and for $\Psi \in \dom(B_{m,M})$,
\begin{align*}
\Vert B_{m,M} \Psi \Vert^2_{L^2(\mathcal N)} =&  \int_\mathcal N \left[ \vert \nabla^\mathcal N (\iota\Psi)  \vert^2 + \frac{\Scal^\mathcal N}{4} \vert \Psi \vert^2 \right] v_\mathcal N + m^2 \Vert \Psi \Vert^2_{L^2(\mathcal K)} \\
&+ M^2 \Vert \Psi \Vert^2_{L^2(\mathcal K^c)} + (M-m) \int_{\partial \mathcal K} \left( \vert \mathcal P_- \Psi \vert^2 -  \vert \mathcal P_+ \Psi \vert^2 \right) v_{\partial \mathcal K}
\end{align*}
where we recall that $\mathcal P_\pm$ were defined in \eqref{Ppm}.
\end{prop}

\begin{proof}
Let $\Psi \in \Gamma_c(\Sigma \mathcal C_{\vert \mathcal N})$. One has
\begin{align*}
\Vert B_{m,M} \Psi \Vert^2_{L^2(\mathcal N)} =& \Vert \nu \cdot (\mathcal{D}^\mathcal N + \ii M) \Psi + \ii (m-M) \mathbf{1}_\mathcal K \nu \cdot\Psi \Vert^2_{L^2(\mathcal N)} \\
=& \Vert (\mathcal{D}^\mathcal N + \ii M) \Psi \Vert^2_{L^2(\mathcal N)} + (m-M)^2 \Vert \Psi \Vert^2_{L^2(\mathcal K)} \\
&+ (m-M) 2 \Re \left\langle (\mathcal{D}^\mathcal N + \ii M) \Psi, \ii \mathbf 1_\mathcal K \Psi \right\rangle_{L^2(\mathcal N)}
\end{align*}
With Lemma~\ref{ipp}
\begin{multline*}
2 \Re \langle (\mathcal{D}^\mathcal N + \ii M) \Psi, \ii \Psi\rangle_{L^2(\mathcal K)} = -\int_{\partial \mathcal K} \left\langle \Psi, \ii \nn \cdot \nu \cdot \Psi \right\rangle v_{\partial \mathcal K} + 2M \left\langle \Psi, \Psi \right\rangle_{L^2(\mathcal K)}.
\end{multline*}
Thus, we have
\begin{multline}
\Vert B_{m,M} \Vert^2_{L^2(\mathcal N)} =  \Vert (\mathcal{D}^\mathcal N + \ii M) \Psi \Vert^2_{L^2(\mathcal N)} + (m-M)^2 \Vert \Psi \Vert^2_{L^2(\mathcal K)} \\
 + (M-m) \int_{\partial \mathcal K} \left\langle \Psi, \ii \nn \cdot\nu\cdot\Psi \right\rangle v_{\partial \mathcal K} + 2M(m-M) \Vert \Psi \Vert^2_{L^2(\mathcal K)} \\
= \Vert (\mathcal{D}^\mathcal N + \ii M) \Psi \Vert^2_{L^2(\mathcal N)} + (m^2 - M^2) \Vert \Psi \Vert^2_{L^2(\mathcal K)} \\
+ (M-m) \int_{\partial \mathcal K} \left\langle \Psi, \ii \nn \cdot\nu\cdot\Psi \right\rangle v_{\partial \mathcal K} \\
= \int_\mathcal N \left[ \vert \nabla^\mathcal N (\iota\Psi) \vert^2 + \frac{\Scal^\mathcal N}{4} \vert \Psi \vert^2 + M^2 \vert \Psi \vert^2 \right] v_\mathcal N + (m^2 - M^2) \Vert \Psi \Vert^2_{L^2(\mathcal K)} \\
 + (M-m) \int_{\partial \mathcal K} \left\langle \Psi, \ii \nn \cdot \nu \cdot \Psi \right\rangle v_{\partial \mathcal K} \\
= \int_\mathcal N \left[ \vert \nabla^\mathcal N (\iota\Psi) \vert^2 + \frac{\Scal^\mathcal N}{4} \vert \Psi \vert^2 \vert \Psi \vert^2 \right] v_\mathcal N + m^2 \Vert \Psi \Vert^2_{L^2(\mathcal K)} + M^2 \Vert \Psi \Vert^2_{L^2(\mathcal K^c)} \\
+ (M-m) \int_{\partial \mathcal K} \left\langle \Psi, \ii \nn \cdot\nu\cdot\Psi \right\rangle v_{\partial \mathcal K}
\end{multline}
and
\[
\left\langle \Psi, \ii \nn \cdot\nu\cdot\Psi \right\rangle = \left\langle \Psi, - \ii \nu\cdot \nn \cdot\Psi \right\rangle \\
= \left\langle \Psi, \mathcal P_-\Psi \right\rangle - \left\langle \Psi, \mathcal P_+\Psi \right\rangle \\
= \vert \mathcal P_-\Psi \vert^2 - \vert \mathcal P_+\Psi \vert^2.
\]
It follows from Theorem~\ref{trace} that there is a constant $C > 0$ such that for all $\Psi \in \Gamma_c(\Sigma \mathcal C_{\vert \mathcal N})$,
\[
\left\Vert B_{m,M} \Psi \right\Vert^2_{L^2(\mathcal N)} \geq C \left( \Vert \nabla^\mathcal N ( \iota \Psi ) \Vert^2_{L^2(\mathcal N)} - \Vert \Psi \Vert^2_{L^2(\mathcal N)} \right).
\]
Then, the graph norm of $\Tilde{B}_{m,M}$ is higher than the $H^1(\Sigma \mathcal C_{\vert \mathcal N})$-norm up to a constant. Then $\dom(B_{m,M}) \subset H^1(\Sigma \mathcal C_{\vert \mathcal N})$, and one can conclude by density.
\end{proof}

\subsection{The limit operator} In this section, we introduce the effective operator $L$ which will appear naturally as the limit operator for $A_m$ when $m \rightarrow - \infty$. We define it as the operator acting on the Hilbert space
\begin{equation} \label{Hspacedef}
\mathbf H := \left\lbrace \Psi \in L^2(\Sigma \mathcal C_{\vert \partial \mathcal K}), \Psi = \ii \nu\cdot \nn \cdot \Psi \right\rbrace
\end{equation}
associated to the quadratic form
\begin{align} \label{quadL}
\ell [\Psi, \Psi] =& \int_{\partial \mathcal K} \left\lbrack \vert \overline \nabla^\mathcal N \iota \Psi \vert^2 + \frac{1}{4} \left(\Scal^{\partial \mathcal K} - \textrm{Tr}(W^2) \right) \vert \Psi \vert^2 \right\rbrack v_{\partial \mathcal K}, \\
\mathcal Q (\ell) :=& \left\lbrace \Psi \in H^1(\Sigma \mathcal C_{\vert \partial \mathcal K}), \Psi = \ii \nu \cdot \nn \cdot\Psi \right\rbrace. \notag
\end{align}

By the compactness of $\mathcal K$, it comes that the form \eqref{quadL} is closed and lower semibounded, so the operator $L$ is well-defined.

The operator $L$ is actually unitary equivalent to the square of the Dirac operator on $\partial \mathcal K$. This fact can be established using the link between the spinor bundles of the spaces $\partial \mathcal K \subset \mathcal N \subset \mathcal C$ and then the proof depends on the dimension $n$ of $\mathcal N$.

\begin{rem}
Using Gauss-Codazzi equations (see \cite[Proposition 4.1]{BGM}, for example), one has
\[
\textrm{Tr}(W^2) = H^2 + \Scal^\mathcal N - \Scal^{\partial \mathcal K} - 2 \Ric^\mathcal N(\nn,\nn).
\]
Thus, the operator we are considering here is a generalization of the operator $L$ defined in \cite[section 2.2]{MOP} and we generalize the result of \cite[Lemma 2.4]{MOP}.
\end{rem}

\begin{lemma} \label{Uequi}
The operator $L$ is unitary equivalent to $(\slashed D^{\partial \mathcal K})^2$.
\end{lemma}

\begin{proof}
We consider separatly the case of $n$ even and $n$ odd.

\textit{Case $n$ odd:} One can represent any $\Psi \in \mathbf H$ as $\Psi =: (\Psi^+, \Psi^-) \in L^2(\Sigma^+ \mathcal C_{\vert \partial \mathcal K}) \times L^2(\Sigma^- \mathcal C_{\vert \partial \mathcal K})$, and then 
\[
\Psi = \ii \nu \cdot \nn \cdot \Psi \Leftrightarrow \iota\Psi = \ii \iota(\nu \cdot \nn \cdot \Psi) \Leftrightarrow \iota\Psi = -\ii \nn \cdot \iota\Psi.
\]
Thus, the isomorphism $\iota$ induces the isomorphisms $\iota^\pm : \Sigma^\pm \mathcal C \rightarrow \Sigma \mathcal N$, and one has
\begin{align*}
\left(
\begin{array}{c}
\iota^+ \Psi^+ \\
\iota^- \Psi^-
\end{array}
\right) =
\left(
\begin{array}{c}
-\ii \nn\cdot \iota^+ \Psi^+ \\ 
\ii \nn \cdot \iota^- \Psi^-
\end{array}
\right).
\end{align*}
We introduce the (pointwise) unitary operator $U: L^2(\Sigma \mathcal N_{\vert \partial \mathcal K}) \rightarrow \mathbf H$, which sends $H^1(\Sigma \mathcal N_{\vert \partial \mathcal K})$ into $\mathcal Q (\ell)$, and is defined by
\begin{align*}
U\Psi = \frac{1}{2} \iota^{-1}
\left(
\begin{array}{c}
(1 - \ii \nn) \cdot \Psi \\
(1 + \ii \nn) \cdot \Psi
\end{array}
\right).
\end{align*}
We compute now $\vert \overline \nabla^\mathcal N \iota (U \Psi)\vert^2$ for $\Psi \in H^1(\Sigma \mathcal N_{\vert \partial \mathcal K})$. Let $(e_1,\ldots, e_{n-1})$ be a pointwise orthonormal frame of $T (\partial \mathcal K)$. The vector fields $(e_j)_{1 \le j \le n-1}$ are naturally identified to elements of $T \mathcal N$. Using the Schrödinger-Lichnerowicz formula and Proposition~\ref{hypersurface}, (3) one has
\begin{align*}
\vert \overline \nabla^N \iota (U \Psi)\vert^2 &= \frac{1}{4} \left( \vert \overline \nabla^N ((1+\ii \nn \cdot) \Psi) \vert^2 + \vert \overline \nabla^N ((1-\ii \nn \cdot) \Psi) \vert^2 \right) \\
&= \frac{1}{2} \sum\limits_{k=1}^{n-1} \left( \vert \nabla^\mathcal N_{e_k} \Psi \vert^2 + \vert (\nabla^\mathcal N_{e_k} \nn) \cdot \Psi + \nn \cdot \nabla^\mathcal N_{e_k} \Psi \vert^2 \right) \\
&= \sum\limits_{k=1}^{n-1} \vert \nabla^\mathcal N_{e_k} \Psi + \frac{1}{2} \nn \cdot W e_k \cdot \Psi \vert^2 + \frac{1}{4} \sum\limits_{k=1}^{n-1} \vert W e_k\cdot \Psi \vert^2 \\
&= \vert \mu^{-1} \nabla^{\partial \mathcal K} \mu \Psi \vert^2 + \frac{1}{4} \textrm{Tr}(W^2) \vert \Psi \vert^2 \\
&= \vert \mathcal D^{\partial \mathcal K} \Psi \vert^2 + \frac{1}{4} \left( -\Scal^{\partial \mathcal K} + \textrm{Tr}(W^2) \right) \vert \Psi \vert^2.
\end{align*}
Thus,
\begin{align*}
\ell [U\Psi, U\Psi] = \int_{\partial \mathcal K} \vert \mathcal D^{\partial \mathcal K} \Psi \vert^2 v_{\partial \mathcal K} = \int_{\partial \mathcal K} \vert \slashed D^{\partial \mathcal K} \mu \Psi \vert^2 v_{\partial \mathcal K}.
\end{align*}
\textit{Case $n$ even :} The isomorphism $\mu$ induces the isomorphisms $\mu^\pm : \Sigma^\pm \mathcal N \rightarrow \Sigma \mathcal K$. According to Proposition~\ref{hypersurface}, as $n-1$ is odd, for all $f \in \Gamma(\Sigma \mathcal N_{\vert \partial \mathcal K})$ one has
\begin{align*}
\mu (\ii \nn \cdot f) =
\left(
\begin{array}{cc}
0 & \mathrm{Id} \\
\mathrm{Id} & 0
\end{array}
\right)
\left(
\begin{array}{c}
\mu^+ f^+ \\
\mu^- f^-
\end{array}
\right).
\end{align*}
Then, for $\Psi \in \mathbf H$ one has
\begin{align*}
&\ii\nu \cdot \nn \cdot \Psi = \Psi \Leftrightarrow -\iota(\ii \nn \cdot \nu \cdot \Psi) = \iota\Psi \Leftrightarrow -\mu(\ii \nn \cdot \iota\Psi) = \mu\iota\Psi \\
&\Leftrightarrow -\left( \begin{array}{cc} 0 & \mathrm{Id} \\ \mathrm{Id} & 0 \end{array} \right) \left( \begin{array}{c} \mu^+ (\iota\Psi)^+ \\ \mu^- (\iota\Psi)^- \end{array} \right) = \mu\iota\Psi \Leftrightarrow (\iota \Psi)^- = -(\mu^-)^{-1}\mu^+(\iota\Psi)^+.
\end{align*}
Thus, the unitary operator
\begin{align*}
\begin{array}{ccc}
U: L^2(\Sigma (\partial \mathcal K)) & \longrightarrow & \mathbf H \\
\Psi & \longmapsto & \frac{1}{\sqrt{2}} \iota^{-1} \mu^{-1} \left( \begin{matrix} -\Psi \\ \Psi \end{matrix} \right)
\end{array}
\end{align*}
sends $H^1(\Sigma (\partial \mathcal K))$ into $ \mathcal Q(\ell)$.
Now we compute $\vert \overline \nabla^N \iota (U \Psi) \vert^2$ for $\Psi \in H^1(\Sigma (\partial \mathcal K))$. Let $(e_1,\ldots,e_{n-1})$ be a pointwise orthonormal frame of $T (\partial \mathcal K)$. One has, using Proposition~\ref{hypersurface}, (3)
\begin{align*}
\vert \overline \nabla^N \iota (U \Psi) \vert^2 &= \vert \mu \overline \nabla^N \iota (U \Psi) \vert^2 \\
&= \frac{1}{2} \left\vert \mu \overline \nabla^N \mu^{-1} \left( \begin{matrix} - \Psi \\ \Psi \end{matrix} \right) \right\vert^2 \\
&= \sum\limits_{k=1}^{n-1} \frac{1}{2} \left\vert \left( \nabla^{\partial \mathcal K}_{e_k} + \frac{1}{2} W e_k \right) \left( \begin{matrix} - \Psi \\ \Psi \end{matrix} \right) \right\vert^2 \\
&= \frac{1}{2} \sum\limits_{k=1}^{n-1} \left( \left\vert \left( \nabla^{\partial \mathcal K}_{e_k} + \frac{1}{2} W e_k \right) \Psi \right\vert^2 + \left\vert \left( \nabla^{\partial \mathcal K}_{e_k} - \frac{1}{2} W e_k \right) \Psi \right\vert^2 \right) \\
&= \sum\limits_{k=1}^{n-1} \left( \vert \nabla^{\partial \mathcal K}_{e_k} \Psi \vert^2 + \frac{1}{4} \vert  W e_k \vert^2 \vert \Psi \vert^2 \right) \\
&= \vert \slashed D^{\partial \mathcal K} \Psi \vert^2 + \frac{1}{4} \left(-\Scal^{\partial \mathcal K} + \textrm{Tr}(W^2) \right) \vert \Psi \vert^2
\end{align*}
Then
\begin{align*}
\ell [U\Psi, U\Psi] = \int_{\partial \mathcal K} \vert \slashed D^{\partial \mathcal K} \Psi \vert^2 v_{\partial \mathcal K}
\end{align*}
which concludes the proof.
\end{proof}

\section{Operators in tubular coordinates} \label{SectionTubular}

When the masses $m$ and $M$ become large, one can localize the eigenvalue problem in a neighbourhood of $\partial \mathcal K$ since the potential in the square of the operators is large outside of this region. For this reason, it is useful to express the operators in tubular coordinates around $\partial \mathcal K$. Thus, we identify a collar near the boundary of $\mathcal K$ with the cylinder $\partial \mathcal K \times (-\delta, \delta)$ and we look at the operator we obtain after this identification. However, the aim of this procedure is to simplify the expression, so we would like to change the induce metric on the cylinder into the product metric. This last step cannot be done without a way to compare the spinor bundles involved, and in particular the way we modify the covariant derivative.

\subsection{Tubular coordinates} For $\delta>0$ we define the tubular neighbourhood of $\partial \mathcal K$ by
\begin{equation}
\nn_{\delta} (\partial \mathcal K) := \lbrace x \in \mathcal N, \dist(x, \partial \mathcal K) < \delta \rbrace.
\end{equation}
Since $\partial \mathcal K$ is compact, $\nn_{\delta} (\partial \mathcal K)$ can be identify with the product $\partial \mathcal K \times (-\delta, \delta)$ through the Riemannian exponential map when $\delta$ is small. To precise this, we define
\begin{equation}
\Pi_\delta := \partial \mathcal K \times (-\delta, \delta), \Pi^+_\delta := \partial \mathcal K \times (0, \delta), \, \Pi^-_\delta := \partial \mathcal K \times (-\delta, 0), \, \Pi^t := \partial \mathcal K \times \lbrace t \rbrace,
\end{equation}
and it is standard that there exists $\delta_0 > 0$ such that the map
\begin{align}
\begin{array}{cccc}
\Pi_{\delta_0} & \longrightarrow & \nn_{\delta_0} (\partial \mathcal K) \\
(x,t) & \longmapsto & \exp^\mathcal N_x(t \nn(x))
\end{array}
\end{align}
is a diffeomorphism on its image.

For every $\delta < \delta_0$, $\Pi_{\delta}$ inherits an orientation via the previous identification. Moreover, one has $T(\Pi_{\delta}) \cong T (\partial \mathcal K) \times T\RR$ and we denote by $\ddt$ the vector field $(0,1) \in T (\partial \mathcal K) \times T\RR$.

From \cite{BGM}, we recall the definition of a generalized cylinder:

\begin{definition}
A generalized cylinder is a manifold of the form $\mathcal Z := I \times \mathcal M$ where $I \subset \RR$ is an interval, $\mathcal M$ is a differentiable manifold and $\mathcal Z$ admits a Riemannian metric $g_{\mathcal Z} = dt^2 + g_t$ where $(g_t)_{t \in I}$ is a smooth one parameter family of Riemannian metrics of $\mathcal M$.
\end{definition}

We identify any vector field $X$ on the hypersurface $\partial \mathcal K$ with the vector field on $T\Pi_{\delta_0}$ also denoted by $X$ and defined by $X_{(y,t)} := X_y$ for all $(y,t) \in \Pi_{\delta_0}$. Note that in this case $\lbrack \ddt, X \rbrack = 0$.

We have two natural metrics on $\Pi_{\delta_0}$. First, the metric $g$ of $\mathcal N$ via the previous identification, and secondly, the Riemannian product metric $h := g_{\vert \partial \mathcal K} + \dd t^2$. Furthermore, $\Sigma \Pi_{\delta_0}$ is the spinor bundle of $\mathcal N$ restricted to $\Pi_{\delta_0}$.

With these notations, we have the useful property:

\begin{lemma} \label{cylg}
The Riemannian manifold $(\Pi_{\delta_0}, g)$ is a generalized cylinder.
\end{lemma}

\begin{proof}
It is sufficient to prove that $g = g_t + dt^2$ with $(g_t)_t$ a family of metrics on $\partial \mathcal K$. This is equivalent to show that the vector field $\ddt$ is normal to $\Pi^t$ for all $t \in (-\delta_0, \delta_0)$. Let $(x,t) \in \Pi_{\delta_0}$ and $X \in T (\partial \mathcal K)$, identified with a vector field on $\Pi_{\delta_0}$ as before. One has
\begin{align*}
\frac{d}{dt} g \left(X, \ddt\right) &= g \left(\nabla^\mathcal N_\ddt X, \ddt\right) + g \left(X, \nabla^\mathcal N_\ddt \ddt\right) \\
&= \overbrace{g \left(\nabla^\mathcal N_{X} \ddt, \ddt \right)}^{= 0} + g \left(\left[\ddt, X\right], \ddt \right) \\
&= g \left(\left[\ddt, X\right], \ddt \right) = 0.
\end{align*}
It comes that $g \left(X, \ddt\right)$ is constant along the curves $s \mapsto (\cdot,s)$ and $g \left(X, \ddt\right)_{(x,0)} = 0$.
Thus, $g \left(X, \ddt\right)_{(x,t)} = 0$, which concludes the proof.
\end{proof}

With Proposition~\ref{cylg}, we deduce that there exist a family of metrics $(g_t)_t$ on $\partial \mathcal K$ such that $g = g_t + \dd t^2$. One can observe that $h = g_0 + \dd t^2$ with these notations.

We define for any $(s, t) \in (-\delta_0, \delta_0)$ the map $\Gamma_{s}^{t}$ which acts as the parallel transport from $s$ to $t$ along the curves $r \mapsto (\cdot,r)$ with respect to the connection $\nabla^\mathcal N$.

We recall that $v_\mathcal N$ is the volume form on $\Pi_{\delta_0}$ compatible with the metric $g$. Let  $v_h := v_{\partial \mathcal K} \wedge \dd t$ be the volume form compatible with $h$ on $\Pi_{\delta_0}$.

The bilinear form $g$ is identified with an endomorphism of $T \Pi_{\delta_0}$ via the metric $h$. Let $(x,t) \in \Pi_{\delta_0}$. For any direct orthonormal frame $f$ of $T_{(x,t)} \Pi_{\delta_0}$ endowed with the metric $h$ we define
\begin{equation} \label{phi}
\phi(x,t) := \sqrt{{\det}_f g}.
\end{equation}
One can show that this does not depend on the choice of the basis, and the volume forms with respect to the different metrics are related by
\begin{equation}
v_\mathcal N = \phi v_h.
\end{equation}
Our aim in this section is to express all the objects on $(\Pi_{\delta_0}, g)$ in terms of the structures over $(\Pi_{\delta_0}, h)$. The function $\phi$ defined above allows to link the integration over these two Riemannian manifolds, and in particular we have a relation between the $L^2$ spaces. Then, we define the isomorphism
\begin{equation}
\begin{array}{cccc}
\Theta : & L^2(\Sigma \Pi_{\delta_0}, v_\mathcal N) & \longrightarrow & L^2(\Sigma \Pi_{\delta_0}, v_h) \\
& \Psi & \longmapsto & \sqrt{\phi} \Psi
\end{array}.
\end{equation}
We remark that $\Theta$ is unitary from $L^2(\Sigma \Pi_{\delta_0}, v_\mathcal N)$ onto $L^2(\Sigma \Pi_{\delta_0}, v_h)$.

\subsection{Estimates in the generalized cylinder} We now fix $\delta < \frac{\delta_0}{2}$. In order to compare the structures over the hypersurfaces $\Pi^t$ for $t \in (-\delta, \delta)$, we first show that the norm of a vector field defined on $\Pi^t$ and extended by parallel transport with respect to $\nabla^\mathcal N$ does not varied too much when $\delta$ is small.

\begin{lemma} \label{comparaison}
We endow $\Pi_{\delta}$ with the metric $g$. There exists $C > 0$ depending only on $\delta_0$ such that for all $t,t' \in (-\delta, \delta)$ and $X \in \Gamma(T \Pi^t)$, for all $x \in \partial \mathcal K$, one has the estimate
\[
\vert X_{(x, t')} - \Gamma_t^{t'} (X_{(x,t)}) \vert_g \leq C \vert t-t' \vert \vert X_{(x,t)} \vert_g,
\]
where $X$ is extended to $T \Pi_\delta$ constantly as before.
\end{lemma}
\begin{proof}
First, we remark that $C_1 := \underset{(y,s) \in \Pi_{\delta_0 / 2}}{\sup} \; \underset{Z \in T_{(y,s)} \setminus \{0 \}}{\sup} \frac{ \vert g(W_{\Pi^s}Z,Z) \vert}{g(Z,Z)}$ is finite by compactness.
Let $t \in (-\delta, \delta)$ and $X \in \Gamma(T \Pi^t)$. We define the vector field $Y \in \Gamma(T \Pi_\delta)$ by $Y_{(y,s)} := \Gamma_t^s (X_{(y,t)})$ for any $(y,s) \in \Pi_{\delta}$.

One has for all $t' \in ( -\delta, \delta)$,
\begin{align*}
\left\vert \ddt g(X,X) \right\vert_{\vert (\cdot, t')} &= \left\vert 2 g \left( \nabla^\mathcal N_{\ddt} X,X \right) \right\vert_{\vert (\cdot, t')} \leq 2 C_1 g(X,X)_{(\cdot, t')}.
\end{align*}

By integration, we obtain the inequality $g(X,X)_{(\cdot, t')} \leq g (X,X)_{(\cdot, t)} \exp(2C_1 \vert t'-t \vert )$, and with $C_2 := \exp (2 \delta_0 C_1)$ one has $g(X,X)_{(\cdot, t')} \le  C_2 g (X,X)_{(\cdot, t)}$.

Now, one has
\begin{align*}
\left\vert \ddt g(X-Y,X-Y) \right\vert_{(\cdot, t')} &= \left\vert 2 g(\nabla^\mathcal N_{\ddt} X,X-Y) \right\vert_{(\cdot, t')} \\
&= \left\vert 2 g(W_{\Pi^{t'}} X,X-Y) \right\vert_{(\cdot, t')} \\
&\leq 2 C_1 \vert X_{(\cdot, t')} \vert_g \vert (X-Y)_{(\cdot, t')} \vert_g \\
&\leq 2C_1 C_2 \vert X_{(\cdot,t)} \vert_g \vert (X-Y)_{(\cdot, t')} \vert_g.
\end{align*}

We need the following technical lemma to conclude.
\begin{lemma} \label{tech1}
Let $I$ be an interval of $\RR$ containing $0$ and let $f : I \rightarrow \RR$ be a differentiable non-negative function. Assume there is $C > 0$  such that $\vert f' \vert \le C \sqrt{f}$. Then, one has $\vert \sqrt f(x) - \sqrt f(0) \vert \le \frac{C}{2} \vert x \vert$ for all $x \in I$.  
\end{lemma}

Using Lemma~\ref{tech1} we arrive at
\[
g(X-Y,X-Y)_{(\cdot, t')} \leq C_1 C_2 \vert X_{(\cdot)} \vert_g^2 (t'-t)^2
\]
and we have the result by taking the square root in this inequality.
\end{proof}

\begin{proof}[Proof of Lemma~\ref{tech1}]
Let $\varepsilon > 0$. One has $\vert f' \vert \le C \sqrt{f + \varepsilon}$, which gives $\left \vert \frac{\dd \sqrt{f + \varepsilon}}{\dd x} \right\vert \le \frac{C}{2}$.  By integration, we obtain that for all $x \in I$, $\vert \sqrt{f(x) + \varepsilon} - \sqrt{f(0) + \varepsilon} \vert \le \frac{C}{2} \vert x \vert$. Letting $\varepsilon$ go to zero, one gets the result.
\end{proof}

We are now able to compare the norms of the covariant derivatives on the different hypersurfaces of $\Pi_\delta$. For this purpose, we recall that $\overline \nabla^\mathcal N \Psi$ is defined as the restriction of $\nabla^\mathcal N \Psi$ to $T^* \partial \mathcal K \otimes \Sigma \Pi_\delta$.

\begin{lemma} \label{partranscomp}
There exists $C > 0$ only depending on $\delta_0$ such that for any $t \in (- \delta, \delta)$ and $\Psi \in \Gamma \left(\Sigma {\Pi_{\delta}} \right)$,
\begin{align*}
(1 - C \delta) \left\vert \overline \nabla^\mathcal N \Gamma_t^0 \Psi (\cdot,t) \right\vert^2 - C \delta \vert \Psi (\cdot,t) \vert^2 \leq \left\vert \overline \nabla^\mathcal N \Psi (\cdot,t) \right\vert^2 \\
\leq (1 + C \delta) \left\vert \overline \nabla^\mathcal N \Gamma_t^0 \Psi (\cdot,t) \right\vert^2 + C \delta \vert \Psi \vert^2 (\cdot,t).
\end{align*}
\end{lemma}
\begin{proof}
Let $\Psi \in \Gamma(\Sigma \Pi_\delta)$. Let $(x,t) \in \Pi_\delta$ and $X \in T(\partial \mathcal K)$ such that $\vert X_{(x,t)} \vert_{g_t} = 1$, and we extend it constantly. The Riemannian curvature of $(\Pi_\delta, g)$ is bounded, so for any $s \in (-\delta, \delta)$ one can find $C_1 > 0$ such that
\begin{align*}
\left \vert \frac{\partial}{\partial s} \vert (\nabla^\mathcal N_X \Gamma_t^s \Psi)(x,s) \vert^2 \right\vert =& 2 \left\vert \Re \left\langle (\nabla^\mathcal N_{\ddt} \nabla^\mathcal N_X \Gamma_t^s \Psi) (x,s),  (\nabla^\mathcal N_X \Gamma_t^s \Psi) (x,s)\right\rangle \right\vert \\
=& \left\vert \Re \left\langle R^\mathcal N \left( \ddt, X \right) \cdot (\Gamma_t^s \Psi) (x,s),  (\nabla^\mathcal N_X \Gamma_s^t \Psi) (x,s)\right\rangle \right\vert \\
\le& C_1 \vert X_{(x,s)} \vert_g \vert \Psi(x,t) \vert  \vert (\nabla^\mathcal N_X \Gamma_s^t \Psi) (x,s) \vert,
\end{align*}
and with Lemma~\ref{comparaison}, one can find $C > 0$ independent of $X$ such that
\[
\vert X_{(x,s)} \vert_g \leq 1 + C \vert t - s \vert \le 1 + C \delta_0.
\]
Thus,
\[
\left\vert \frac{\partial}{\partial s} \vert (\nabla^\mathcal N_X \Gamma_t^s \Psi)(x,s) \vert^2 \right\vert \le C_1 (1+C \delta_0) \vert \Psi(x,t) \vert  \vert (\nabla^\mathcal N_X \Gamma_t^s \Psi) (x,s) \vert.
\]
Using Lemma~\ref{tech1}, we obtain
\[
\left\vert \vert (\nabla^\mathcal N_X \Gamma_t^0 \Psi)(x,0) \vert - \vert \nabla^\mathcal N_X \Psi(x,t) \vert \right\vert \le C_1 (1+ C \delta_0) \vert t \vert \vert \Psi(x,t) \vert.
\]
On the other hand,
\begin{align*}
\vert (\nabla^\mathcal N_X \Gamma_t^0 \Psi)(x,0) - (\nabla^\mathcal N_{\Gamma_t^0 X} \Gamma_t^0 \Psi)(x,0) \vert \le& \vert X_{(x,0)} - \Gamma_t^0 (X_{(x,t)}) \vert_g \vert (\overline \nabla^\mathcal N \Gamma_t^0 \Psi)(x,0) \vert \\
\le& C \vert t \vert \vert (\overline \nabla^\mathcal N \Gamma_t^0 \Psi)(x,0) \vert.
\end{align*}
Thus, combining the previous estimates, one can find $C_2 > 0$ such that
\[
\left\vert \vert (\nabla^\mathcal N_{\Gamma_t^0 X} \Gamma_t^0 \Psi)(x,0) \vert - \vert \nabla^\mathcal N_X \Psi(x,t) \vert \right\vert \le C_2 \vert t \vert \left( \vert \Psi(x,t) \vert + \vert (\overline \nabla^\mathcal N \Gamma_t^0 \Psi)(x,0) \vert \right).
\]
Now, let $(e_1,\ldots, e_n)$ be an othonormal frame at the point $(x,t)$. One obtains
\begin{align*}
\left\vert \vert (\overline \nabla^\mathcal N \Gamma_t^0 \Psi)(x,0) \vert - \vert \overline \nabla^\mathcal N \Psi(x,t) \vert \right\vert \le& \sum\limits_{k = 1}^n \left\vert \vert (\nabla^\mathcal N_{\Gamma_t^0 e_k} \Gamma_t^0 \Psi)(x,0) \vert - \vert \nabla^\mathcal N_{e_k} \Psi(x,t) \vert \right\vert \\
\le& \sum\limits_{k = 1}^n C_2 \vert t \vert \left( \vert \Psi(x,t) \vert + \vert (\overline \nabla^\mathcal N \Gamma_t^0 \Psi)(x,0) \vert \right) \\
\le& n C_2 \delta \left( \vert \Psi(x,t) \vert + \vert (\overline \nabla^\mathcal N \Gamma_t^0 \Psi)(x,0) \vert \right).
\end{align*}
The result then comes from the following lemma:
\begin{lemma} \label{tech2}
For all $C >  0$ and $\delta < \delta_0/2$, there is $C' > 0$ depending only on $\delta_0$ and $C$ such that for all $a,b,d > 0$ verifying $\vert a - b \vert \le C \delta (b + d)$, one has $\vert a^2 - b^2 \vert \le C' \delta (b^2 + d^2)$.
\end{lemma}
\end{proof}

\begin{proof}[Proof of Lemma~\ref{tech2}]
One has
\begin{align*}
\vert a^2 - b^2 \vert =& \vert (a - b + b)^2 - b^2 \vert = \vert (a - b)^2 + 2 (a-b) b \vert \le \vert a - b \vert^2 + \vert 2 (a-b) b \vert \\
\le& C^2 \delta^2 (b + d)^2 + 2 C \delta (b+ d) b \le C^2 \delta^2 (b + d)^2 + C \delta (b + d)^2 + C \delta b^2 \\
\le& (2 C^2 \delta^2 + C \delta) (b^2 + d^2) + C \delta b^2 \le (2 C^2 \delta_0 + 2 C) \delta (b^2 + d^2),
\end{align*}
which is equivalent to the statement of the lemma.
\end{proof}

\subsection{Bracketing for the quadratic form of $A_m^2$} We end this section by finding a lower and an upper bound for the quadratic form of $A_m^2$ expressed in the tubular coordinates.

\begin{lemma} \label{estimations}
There exists $c > 0$ depending only on $\delta_0$ such that the following estimates hold:

\begin{minipage}{0.4\textwidth}
\begin{equation}
\Vert \phi - 1 \Vert_{L^\infty(\Pi_\delta)} \leq c \delta \label{lmeq1}
\end{equation}
\end{minipage}
\begin{minipage}{0.4\textwidth}
\begin{equation}
\Vert \overline \nabla^\mathcal N \phi \Vert_{L^\infty(\Pi_\delta)}^2 \leq c \delta^2 \label{lmeq2}
\end{equation}
\end{minipage}

\begin{minipage}{0.4\textwidth}
\begin{equation}
\left\Vert \frac{(\partial_t \phi) (\cdot, \delta)}{2 \phi (\cdot, \delta)} \right\Vert_{L^\infty(\partial \mathcal K)} \leq c \label{lmeq3}
\end{equation}
\end{minipage}
\begin{minipage}{0.4\textwidth}
\begin{equation}
\partial_t \phi (\cdot, 0) = -\frac{H}{2} \label{lmeq4}
\end{equation}
\end{minipage}

\begin{equation}
\left\vert \frac{\partial_ t^2 \phi}{2 \phi}(x,t) - \frac{(\partial_ t \phi)^2}{4 \phi^2} (x,t) - \frac{1}{4}(\Scal^{\partial \mathcal K} (x) - \mathrm{Tr} (W^2) (x) - \Scal^\mathcal N(x,t)) \right\vert \leq c \delta, \label{lmeq5}
\end{equation}
for all $(x,t) \in \Pi_\delta$.

\end{lemma}
\begin{proof}
To show \eqref{lmeq1}, \eqref{lmeq2} and \eqref{lmeq3}, we just remark that $\phi$ is a smooth function on the closure of $\Pi_\delta$ which is compact, so it is bounded on $\Pi_\delta$ as well as all its derivatives.

Thanks to Lemma~\ref{cylg} we can use \cite[formula (4.1)]{BGM}, so \eqref{lmeq4} follows from:
\[
\frac{\partial_t \phi (\cdot, 0)}{2 \phi(\cdot, 0)} = \frac{\partial_t \sqrt{\det_f g} (\cdot, 0)}{2} = \frac{\mathrm{Tr} (\partial_t g) (\cdot, 0)}{4\sqrt{\det_f g}(\cdot,0)} = - \frac{2 \textrm{Tr} (W)}{4} = -\frac{H}{2}.
\]
Finally, we prove (\ref{lmeq5}). Let $(x,t) \in \Pi_\delta$ and let $f$ be a direct orthonormal frame of $(\Pi_\delta, h)$ at $(x,t)$. One has, using lemma~\ref{cylg} and \cite[equation (4.8)]{BGM},
\begin{align*}
\frac{\partial_ t^2 \phi}{2 \phi}(x,t) - \frac{(\partial_ t \phi)^2}{4 \phi^2} (x,t) =& \frac{\partial_ t^2 \det_f g}{4 \det_f g} (x,t) - \frac{3 (\partial_t \det_f g)^2}{16 (\det_f g)^2} (x,t) \\
=& \left( \frac{\partial_ t^2 \det_f g}{4} - \frac{3 (\partial_t \det_f g)^2}{16} \right)(x,0) + \mathcal O (t) \\
=& \left( \frac{H^2}{4} - \textrm{Tr} (W^2) + \frac{\textrm{Tr} (\ddot g_t\vert_{t=0})}{4} \right) (x) + \mathcal O (t) \\
=& \frac{1}{4} (\Scal^{\partial \mathcal K} (x) - \textrm{Tr} (W^2) (x) - \Scal^\mathcal N (x,t)) + \mathcal O (t),
\end{align*}
which gives the result.
\end{proof}

For $\alpha \in \RR$, $\delta \in ( 0, \delta_0/2 )$ and $\Psi \in H^1 \left( \Sigma \overline{\Pi^\pm_\delta} \right))$ we define
\begin{equation}
J_\pm(\Psi) := \int_ {\Pi^\pm_\delta} \left\lbrack \vert \nabla^\mathcal N \Psi \vert^2 + \frac{\Scal^\mathcal N}{4} \vert \Psi \vert^2 \right\rbrack v_\mathcal N + \int_{\partial \mathcal K} \left(\alpha \pm \frac{H}{2}\right) \vert \Psi \vert^2 v_{\partial \mathcal K}.
\end{equation}

\begin{prop} \label{quadbound}
There is a constant $c > 0$ depending only on $\delta_0$ such that for all $\alpha \in \RR$ and $\delta \in (0,\delta_0/2)$, the following inequalities hold:
\begin{enumerate}
\item for every $\Psi \in H^1 \left( \Sigma\overline{\Pi^\pm_\delta} \right)$, one has
\begin{multline}
J_\pm(\Psi) \geq \int_{\Pi^\pm_\delta} \left[ (1 - c \delta) \left\vert (\overline \nabla^\mathcal N \Gamma_ t^0 \Theta \Psi) (x,0) \right\vert^2 + \vert \nabla^\mathcal N_\ddt \Theta \Psi \vert^2 \right] v_h(x,t) \\
+ \int_{\Pi_\delta^\pm} \left[ \left(\frac{\Scal^{\partial \mathcal K} - \mathrm{Tr} (W^2)}{4} - c\delta\right) \vert \Theta \Psi \vert^2 \right] v_h \\
+ \int_{\partial \mathcal K} \left[\alpha \vert (\Theta \Psi) ( \cdot, 0) \vert^2 - c \vert (\Theta \Psi) (\cdot , \delta) \vert^2 \right] v_{\partial \mathcal K}.
\end{multline}
\item if moreover $\Psi = 0$ on the outer boundary $\Pi^{\pm \delta}$, one has
\begin{multline}
J_\pm(\Psi) \leq \int_{\Pi^\pm_\delta} \left[ (1 + c \delta) \left\vert (\overline \nabla^\mathcal N \Gamma_ t^0 \Theta \Psi) (x,0) \right\vert^2 + \vert \nabla^\mathcal N_\ddt \Theta \Psi \vert^2 \right] v_h(x,t) \\
+ \int_{\Pi_\delta^\pm} \left[ \left(\frac{\Scal^{\partial \mathcal K} - \mathrm{Tr} (W^2)}{4} + c\delta\right) \vert \Theta \Psi \vert^2 \right] v_h + \alpha \int_{\partial \mathcal K} \vert (\Theta \Psi) ( \cdot , 0) \vert^2 v_{\partial \mathcal K}
\end{multline}
\end{enumerate}
\end{prop}
\begin{proof}
It is sufficient to prove the result for $\Psi \in \Gamma_c \left( \Sigma  \overline{\Pi_\delta^\pm} \right)$ and to conclude by density. One has
\[
J_\pm(\Psi) = \int_ {\Pi^\pm_\delta} \left[ \vert \nabla^\mathcal N \phi^{-\frac{1}{2}} \Theta \Psi \vert^2 + \frac{\Scal^\mathcal N}{4} \vert \phi^{-\frac{1}{2}} \Theta \Psi \vert^2 \right] \phi v_h + \int_{\partial \mathcal K} \left(\alpha \pm \frac{H}{2}\right) \vert \Psi \vert^2 v_{\partial \mathcal K}.
\]
We remark that $\phi = 1$ on $\partial \mathcal K$ and Lemma~\ref{partranscomp} gives a constant $C > 0$ such that
\begin{multline*}
\int_ {\Pi^\pm_\delta} \left\lbrack \left\vert \nabla^\mathcal N_{\ddt} \phi^{-\frac{1}{2}} \Theta \Psi \right\vert^2 + (1 - C \delta) \left\vert \overline \nabla^\mathcal N \Gamma_t^0 \phi^{-\frac{1}{2}} \Theta \Psi \right\vert^2 (\cdot ,0) - C \delta \vert \phi^{-\frac{1}{2}} \Theta \Psi \vert^2 \right\rbrack \phi v_h \\
+ \int_{\Pi^\pm_\delta} \frac{\Scal^\Pi}{4} \vert \Theta \Psi \vert^2 \phi v_h + \int_{\partial \mathcal K} \left(\alpha \pm \frac{H}{2}\right) \vert \Theta \Psi \vert^2 v_{\partial \mathcal K} \leq J_\pm(\Psi) \\
\leq \int_ {\Pi^\pm_\delta} \left\lbrack \left\vert \nabla^\mathcal N_{\ddt} \phi^{-\frac{1}{2}} \Theta \Psi \right\vert^2 + (1 + C \delta) \left\vert \overline \nabla^\mathcal N \Gamma_t^0 \phi^{-\frac{1}{2}} \Theta \Psi \right\vert^2 (\cdot,0) + C \delta \vert \phi^{-\frac{1}{2}} \Theta \Psi \vert^2 \right\rbrack \phi v_h \\
+ \int_{\Pi^\pm_\delta} \frac{\Scal^\mathcal N}{4} \vert \phi^{-\frac{1}{2}} \Theta \Psi \vert^2 \phi v_h + \int_{\partial \mathcal K} \left(\alpha \pm \frac{H}{2}\right) \vert \Theta \Psi \vert^2 v_{\partial \mathcal K}.
\end{multline*}
Moreover, for all $(x,t) \in \Pi_\delta$ and $X \in T_x \partial \mathcal K$,
\begin{multline*}
\left\vert \overline \nabla^\mathcal N_X \Gamma_t^0 (\phi^{-\frac{1}{2}} \Theta \Psi) \right\vert^2 (x,0) \phi(x,t) \\
= \left\vert \overline \nabla^\mathcal N_X \Gamma_t^0 \Theta \Psi - \frac{1}{2 \phi(x,t)}X(\phi)(x,t) \Gamma_t^0 \Theta \Psi\right\vert^2 (x,0) \\
= \left\vert \overline \nabla^\mathcal N_X \Gamma_t^0 \Theta \Psi \right\vert^2 (x,0) + \left\vert \frac{1}{2 \phi(x,t)}X(\phi)(x,t) \Gamma_t^0 \Theta \Psi \right\vert^2 (x,0) \\
- \frac{1}{\phi(x,t)} \Re \left\langle \overline \nabla^\mathcal N_X \Gamma_t^0 \Theta \Psi, X(\phi)(x,t) \Gamma_t^0 \Theta \Psi \right\rangle (x,0)
\end{multline*}
and 
\[
\left\vert \Re \left\langle \overline \nabla^\mathcal N \Gamma_t^0 \Theta \Psi, X(\phi)(x,t) \Gamma_t^0 \Theta \Psi \right\rangle (x,0) \right\vert \leq \delta \vert \overline \nabla^\mathcal N \Gamma_t^0 \Theta \Psi \vert^2(x,0) + \vert \Theta \Psi \vert^2 \vert X(\phi) \vert^2 (x,t)/ \delta.
\]
Then, with the inequality (\ref{lmeq2}), there is $C' > 0$ such that
\begin{multline*}
(1 - C' \delta) \left\vert \overline \nabla^\mathcal N \Gamma_t^0 \Theta \Psi \right\vert^2 (x,0) -  C' \delta \left\vert \Theta \Psi \right\vert^2 (x,t) \\
\leq (1 \pm C \delta) \left\vert \overline \nabla^\mathcal N \Gamma_t^0 \phi^{-\frac{1}{2}} \Theta \Psi \right\vert^2 (x,0) \phi(x,t) \\
\leq (1 + C' \delta) \left\vert \overline \nabla^\mathcal N \Gamma_t^0 \Theta \Psi \right\vert^2 (x,0) +  C' \delta \left\vert \Theta \Psi \right\vert^2 (x,t).
\end{multline*}
It remains to compute
\begin{align*}
\phi \vert \nabla^\mathcal N_{\ddt} \phi^{-\frac{1}{2}} \Theta \Psi \vert^2 &= \left\vert \nabla^\mathcal N_{\ddt} \Theta \Psi - \frac{1}{2 \phi} \partial_t \phi (\Theta \Psi) \right\vert^2 \\
&= \left\vert \nabla^\mathcal N_{\ddt} \Theta \Psi \right\vert^2 + \frac{(\partial_t \phi)^2}{4 \phi^2} \left\vert \Theta \Psi \right\vert^2 - \frac{\partial_t \phi}{\phi} \Re \left\langle \nabla^\mathcal N_{\ddt} \Theta \Psi, \Theta \Psi \right\rangle \\
&= \left\vert \nabla^\mathcal N_{\ddt} \Theta \Psi \right\vert^2 + \frac{(\partial_t \phi)^2}{4 \phi^2} \left\vert \Theta \Psi \right\vert^2 - \frac{\partial_t \phi}{2 \phi} \partial_t \left\vert \Theta \Psi \right\vert^2.
\end{align*}
Thus, integrating by parts, one has
\begin{multline*}
\int_{\Pi^\pm_\delta} \vert \nabla^\mathcal N_{\ddt} \phi^{-\frac{1}{2}} \Theta \Psi \vert^2 \phi v_h = \int_{\Pi^\pm_\delta} \left\lbrack \left\vert \nabla^\mathcal N_{\ddt} \Theta \Psi \right\vert^2 + \frac{(\partial_t \phi)^2}{4 \phi^2} \left\vert \Theta \Psi \right\vert^2 - \frac{\partial_t \phi}{2 \phi} \partial_t \left\vert \Theta \Psi \right\vert^2 \right\rbrack v_h \\
= \int_{\Pi^\pm_\delta} \left\lbrack \left\vert \nabla^\mathcal N_{\ddt} \Theta \Psi \right\vert^2 + \frac{(\partial_t \phi)^2}{4 \phi^2} \left\vert \Theta \Psi \right\vert^2 + \left(  \frac{\partial_t^2 \phi}{2 \phi} - \frac{(\partial_t \phi)^2}{2 \phi^2} \right) \left\vert \Theta \Psi \right\vert^2 \right\rbrack v_h \\
\mp \int_{\Pi^{\pm \delta}} \frac{\partial_t \phi}{2 \phi} \vert \Theta \Psi \vert^2 v_{\Pi^\pm\delta} \pm \int_{\Pi^0} \frac{\partial_t \phi}{2 \phi} \vert \Theta \Psi \vert^2 v_{\partial \mathcal K} \\
= \int_{\Pi^\pm_\delta} \left\lbrack \left\vert \nabla^\mathcal N_{\ddt} \Theta \Psi \right\vert^2 + \left(  \frac{\partial_t^2 \phi}{2 \phi} - \frac{(\partial_t \phi)^2}{4 \phi^2} \right) \left\vert \Theta \Psi \right\vert^2 \right\rbrack v_h \\
\mp \int_{\Pi^{\pm \delta}} \frac{\partial_t \phi}{2 \phi} \vert \Theta \Psi \vert^2 v_{\partial \mathcal K} \mp \int_{\Pi^0} \frac{H}{2} \vert \Theta \Psi \vert^2 v_{\partial \mathcal K}
\end{multline*}
where we used (\ref{lmeq4}). Thus, we have
\begin{multline*}
J_\pm (\Psi) \leq \int_ {\Pi^\pm_\delta} \left\lbrack (1 + C \delta) \left\vert (\overline \nabla^\mathcal N \Gamma_ t^0 \Theta \Psi) (x,0) \right\vert^2 + \vert \nabla^\mathcal N_\ddt \Theta \Psi \vert^2 \right. \\
\left. + \left( \frac{\partial_t^2 \phi}{2 \phi} - \frac{(\partial_t \phi)^2}{4 \phi^2} + \frac{\Scal^\mathcal N}{4} + C \delta \right) \vert \Theta \Psi \vert^2 (x,t) \right\rbrack v_h(x,t) \\
+ \alpha \int_{\Pi^0} \vert \Theta \Psi \vert^2 v_{\partial \mathcal K} \; \textrm{ if $\Psi = 0$ on $\Pi^{\pm \delta}$}
\end{multline*}
\begin{multline*}
J_\pm (\Psi) \geq \int_ {\Pi^\pm_\delta} \left\lbrack (1 - C \delta) \left\vert (\overline \nabla^\mathcal N \Gamma_ t^0 \Theta \Psi) (x,0) \right\vert^2 + \vert \nabla^\mathcal N_\ddt \Theta \Psi \vert^2 \right. \\
\left. + \left(  \frac{\partial_t^2 \phi}{2 \phi} - \frac{(\partial_t \phi)^2}{4 \phi^2} + \frac{\Scal^\mathcal N}{4} - C \delta \right) \vert \Theta \Psi \vert^2 (x,t) \right\rbrack v_h \\
+ \alpha \int_{\Pi^0} \vert \Theta \Psi \vert^2 v_{\partial \mathcal K} \mp \int_{\Pi^{\pm \delta}} \frac{\partial_t \phi}{2 \phi} \vert \Theta \Psi \vert^2 v_{\partial \mathcal K}.
\end{multline*}
These estimates, together with \eqref{lmeq3} and \eqref{lmeq5} give the result.
\end{proof}

\section{Analysis of the one-dimensional operators} \label{OneDim}

Our proofs for the main results will use some separation of variables in the generalized cylinder $\Pi_\delta$. For this reason, we will need to analyse various one-dimensional operators. We define them in this section and we state the propositions we need on the behaviour of their eigenvalues in some asymptotic regimes.

We recall the two results from \cite[Section 3]{MOP}:

\begin{lemma} \label{Sanalysis}
Let $\varepsilon > 0$. Let $\alpha > 0$ and let $S$ be the self-adjoint operator on $L^2(0, \delta)$ associated to the quadratic form
\begin{equation*}
s[f,f] = \int_0^\varepsilon \vert f' \vert^2 dt - \alpha \vert f(0) \vert^2, \, \mathcal Q(s) = \left\lbrace f \in H^1(0,\varepsilon), \, f(\varepsilon) = 0 \right\rbrace.
\end{equation*}
Then, when $\alpha \rightarrow + \infty$, one has $E_1(S) = - \alpha^2 + \mathcal O(e^{-\varepsilon \alpha})$, and the associated $L^2-normalized$ eigenfunction $f$ satisfies $\vert f(0) \vert^2 = 2 \alpha + \mathcal O (1)$.
\end{lemma}

\begin{lemma} \label{S'analysis}
Let $\varepsilon > 0$. Let $\alpha, \beta > 0$ and let $S'$ be the self-adjoint operator on $L^2(0, \varepsilon)$ associated to the quadratic form
\begin{equation*}
s'[f,f] = \int_0^\varepsilon \vert f' \vert^2 dt + m\vert f(0) \vert^2 - \beta \vert f(\varepsilon) \vert^2, \, \mathcal Q(S') = H^1(0,\varepsilon).
\end{equation*}
Then, when $\alpha \rightarrow + \infty$, one has $E_1(S') = - \alpha^2 + \mathcal O(e^{-\varepsilon \alpha})$, and there exist $b^\pm > 0$ and $b > 0$ such that
\[
b^- j^2 - b \le E_j(S') \le b^+ j^2 \textrm{ for all $j \ge 2$ and $\alpha > 0$}.
\]
\end{lemma}

A third one-dimensional operator will be of interest for the proof of Theorem~\ref{main3}. It can be interpreted as the laplacian on an interval $(- \delta, \delta)$ with a potential consiting of two masses on the two sides of the origin and a $\delta$-interaction at $0$. For this last operator, we state the result in the very specific case of our framework, so we recall that $m, M \in \RR$ and $\delta \in (0, \delta_0/2)$.

For $\beta > 0$, let $X$ be the operator associated to the quadratic form
\begin{multline} \label{Xdefinition}
x [f,f] = \int_{-\delta}^\delta \vert f' \vert^2 \dd t - \beta ( \vert f(\delta) \vert^2 + \vert f(- \delta) \vert^2 ) \\
+ \int_{-\delta}^0 M^2 \vert f \vert^2 \dd t + \int_0^\delta m^2 \vert f \vert^2 \dd t - ( M - m ) \vert f(0) \vert^2,
\\ \mathcal Q (x) = H^1(-\delta, \delta).
\end{multline}

\begin{lemma} \label{lemma1}
For $\delta > 0$ and $\beta > 0$ fixed, one has $E_1(X) = \mathcal O (e^{-\frac{\min(\vert m \vert, M)}{2} \delta})$ when $\min (- m, M) \rightarrow + \infty$. Moreover, for all $j \ge 2$, one can find $C_1, C_2 > 0$ such that
\begin{equation*}
\min (m^2, M^2) + C_1 j^2 - C_2 \le E_j(X).
\end{equation*}
\end{lemma}
\begin{proof}
One can see that the operator $X$ acts as $f \mapsto - f'' + (M^2 \mathbf 1_{(-\delta,0)} + m^2 \mathbf 1_{(0,\delta)}) f$ on the functions $f \in H^1(-\delta, \delta) \cap (H^2(-\delta, 0)\cup H^2(0, \delta))$ satisfying $f'(\delta) - \beta f(\delta) = f'(-\delta) + \beta f(-\delta) = 0$ and $f'(0^+) - f(0^-) + (\vert m \vert + M) f(0) = 0$. We search for a negative eigenvalue for $X$ in the form $-k^2$ with $k > 0$. The associated eigenfunction must be of the form
\begin{equation}
f(t) = \left\lbrace
\begin{array}{cc}
a_1 e^{-k_1 t} + b_1 e^{k_1 t} & \textrm{if $t \in (-\delta, 0)$} \\
a_2 e^{k_2 t} + b_2 e^{-k_2 t} & \textrm{if $t \in (0, \delta)$}
\end{array}
\right.
\end{equation}
where $k_1 := \sqrt{M^2 + k^2}$ and $k_2 := \sqrt{m^2 + k^2}$.

We can rewrite the equations satisfied by $f$ as
\begin{align*}
0 &= a_2 (k_2 - \beta) e^{k_2 \delta} - b_2 (k_2 + \beta) e^{-k_2 \delta} \\
0 &= a_1 (k_1 - \beta) e^{k_1 \delta} - b_1 (k_2 + \beta) e^{-k_2 \delta} \\
a_1 + b_1 &= a_2 + b_2 \\
0 &= a_2 k_2 - b_2 k_2 + a_1 k_1 - b_1 k_1 + (\vert m \vert + M) (a_1 + b_1) \label{13}.
\end{align*}
The two first equations give $b_2 = \frac{k_2 - \beta}{k_2 + \beta} e^{2k_2 \delta} a_2$ and $b_1 = \frac{k_1 - \beta}{k_1 + \beta} e^{2k_1 \delta} a_1$. Thus, with the equation of continuity we have
\[
a_1 \left( 1 + \frac{k_1 - \beta}{k_1 + \beta} e^{2k_1 \delta} \right) = a_2 \left( 1 + \frac{k_2 - \beta}{k_2 + \beta} e^{2k_2 \delta} \right).
\]
We conclude that
\begin{equation*}
a_2 = a_1 \left( 1 + \frac{k_2 - \beta}{k_2 + \beta} e^{2k_2 \delta} \right)^{-1} \left( 1 + \frac{k_1 - \beta}{k_1 + \beta} e^{2k_1 \delta} \right)
\end{equation*}
because for $\min (\vert m \vert, M)$ large enough, one has that the different term are not zero.

We arrive at
\begin{multline*}
\vert m \vert + M = k_2 \left( \frac{k_2 - \beta}{k_2 + \beta} e^{2k_2 \delta} - 1\right) \left( 1 + \frac{k_2 - \beta}{k_2 + \beta} e^{2k_2 \delta} \right)^{-1} \\
+ k_1 \left( \frac{k_1 - \beta}{k_1 + \beta} e^{2k_1 \delta} - 1 \right) \left( 1 + \frac{k_1 - \beta}{k_1 + \beta} e^{2k_1 \delta} \right)^{-1}.
\end{multline*}
Let $F(x) := x \left( \frac{x - \beta}{x + \beta} e^{2x \delta} - 1\right) \left( 1 + \frac{x - \beta}{x + \beta} e^{2x \delta} \right)^{-1}$ defined on $(\min (\vert m \vert, M), + \infty)$. The previous equation reads $\vert m \vert + M = F(k_1) + F(k_2)$, and when $k = 0$ the right-hand side is $F(\vert m \vert) + F(M) < \vert m \vert + M$. Since $F(k_1) + F(k_2) \rightarrow + \infty$ when $k \rightarrow + \infty$ and $F$ is strictly increasing there exists an unique $k \in (0, + \infty)$ such that $\vert m \vert + M = F(k_1) + F(k_2)$.

Now, one has
\begin{align*}
F(x) = x (1 + \mathcal O (e^{-2x \delta}) ) = x + \mathcal O (e^{-3x \delta/2}).
\end{align*}
Thus, for $\zeta := \min (\vert m \vert, M)$ large enough there holds
\[
k_2 + k_1- 2 e^{-\zeta \delta} \le \vert m \vert + M \le k_2 + k_1 + 2 e^{-\zeta \delta}
\]
and
\[
0 \le \sqrt{m^2 + k^2} - \vert m \vert + \sqrt{M^2 + k^2} - M \le 2 e^{-\zeta \delta}.
\]
Then, $\sqrt{\zeta^2 + k^2} - \zeta \le 2 e^{-\zeta \delta}$ and we arrive at
\begin{equation*}
k^2 = \mathcal O( e^{-\zeta \delta / 2} ).
\end{equation*}
To conclude, we consider the operator $X'$ defined by the same form as $X$ but with the form domain $\lbrace f \in H^1(-\delta, \delta), f(0) = 0 \rbrace$. From the Min-Max principle, one has $E_{j-1}(X') \le E_j(X_\alpha) \le E_j(X')$ for all $j \ge 2$ because $X$ is a rank-one perturbation of $X'$. But $X' \cong (S_D + m^2) \oplus (S_D + M^2)$ where $S_D$ is the operator acting in $L^2(0,\delta)$ as $f \mapsto -f''$ for $f \in H^2(0, \delta)$ with $f(0) = f'(\delta) - \beta f(\delta) = 0$. We conclude by remarking that $E_j(S_D) \sim \pi^2 j^2 / \delta^2$ when $j \rightarrow + \infty$, so $E_j(X') \ge \min (m^2, M^2) - C_2 + C_1 j^2$ for suitable $C_1, C_2 > 0$.
\end{proof}

\section{Asymptotics analysis for the operator $A_m$} \label{Aform}

In this section, we prove Theorem~\ref{main1} following the analysis of \cite[Section 4]{MOP}. The proof is made by localizing the problem near the boundary of $\mathcal K$ and using the analysis done in the previous section to find a lower and an upper bound for the limit of the eigenvalues. These bounds coincide and are equal to the eigenvalues of the model operator $L$ introduced in \eqref{quadL}. We begin by showing a Dirichlet-Neumann bracketing for the operator $A_m$.

Let $\delta \in (0, \delta_0/2)$. We introduce several new operators. Let $Z^+_m$, $Z^-_m$, $Z'_m$ be the operators defined by their quadratic forms $z^+_m$, $z^-_m$, $z'_m$ which admit the same expression as the quadratic form of $A_m^2$ given in Proposition~\ref{quadA} with 
\begin{equation}
\dom(z^+_m) = \left\lbrace \Psi \in H^1(\Sigma \mathcal C_{\vert \overline{\Pi^-_\delta}}), \Psi = \ii \nu \cdot \nn \cdot \Psi \textrm{ on } \partial \mathcal K \textrm{ and } \Psi = 0 \textrm{ on } \Pi^{-\delta} \right\rbrace,
\end{equation}
\begin{equation}
\dom(z^-_m) = \left\lbrace \Psi \in H^1(\Sigma \mathcal C_{\vert \overline{\Pi^-_\delta}}), \Psi = \ii \nu \cdot \nn \cdot \Psi \textrm{ on } \Pi^0 \right\rbrace,
\end{equation}
\begin{equation}
\dom(z'_m) = H^1 \left(\Sigma \mathcal C_{\vert \mathcal K \setminus (\Pi^-_\delta \cup \Pi^0)} \right).
\end{equation}

We define the maps $J_1: \dom(A_m) \rightarrow \dom(z_m^-) \oplus \dom(z'_m)$, $\Psi \mapsto (\Psi_{\vert \overline{\Pi^-_\delta}}, \Psi_{\vert \mathcal K \setminus (\Pi^-_\delta \cup \Pi^0)})$ and $J_2: \dom(z^+_m) \rightarrow \dom(A_m)$ which is the extension by zero. For $\Psi_1 \in \dom(A_m)$ one has
\[
(z^-_m \oplus z'_m) \left \lbrack J_1(\Psi_1), J_1(\Psi_1) \right\rbrack \leq \left\langle A_m \Psi_1, A_m \Psi_1 \right\rangle_{L^2(\mathcal K)},
\]
and for $\Psi_2 \in \dom(z^+_m)$,
\[
\left\langle A_m J_ 2(\Psi_2), A_m J_ 2(\Psi_2) \right\rangle_{L^2(\mathcal K)} \leq z^+_m \left\lbrack \Psi_2, \Psi_2 \right\rbrack.
\]

Then, the Min-Max principle gives
\begin{equation}
E_j\left(Z^-_m \oplus Z'_m \right) \leq E_j\left(A_m^2\right) \leq E_j\left(Z^+_m \right).
\end{equation}
We remark that $Z'_m \geq m^2$ and then, for any $j \in \mathbb \mathcal N$ such that $E_j \left( Z^+_m \right) < m^2$, one has
\begin{equation}
E_j\left( Z^-_m \right) \leq E_j\left( A_m^2\right) \leq E_j\left( Z^+_m \right).
\end{equation}

We introduce the notation $\SS^-_\delta := \iota (\Sigma \mathcal C_{\vert \overline{\Pi^-_\delta}})$. Let $c > 0$ be the constant given by Proposition~\ref{quadbound}.
We consider the two quadratic forms in $L^2(\SS^-_\delta, v_h)$ given by
\begin{multline}
y^+_m [\Psi, \Psi] := \int_{\Pi^-_\delta} \left[ (1 + c \delta) \vert \overline \nabla^\mathcal N \Gamma_t^0 \Psi \vert^2 + \vert \nabla^\mathcal N_\ddt \Psi \vert^2 \right] v_h \\
+ \int_{\Pi^-_\delta} \left[ \left(m^2 + \frac{\Scal^{\partial \mathcal K} - \textrm{Tr} (W^2)}{4} + c\delta\right) \vert \Psi \vert^2 \right] v_h + m \int_{\partial \mathcal K} \vert \Psi (\cdot, 0) \vert^2 v_{\partial \mathcal K} \\
\mathcal Q(y^+_m) := \left\lbrace \Psi \in H^1 \left( \SS^-_\delta \right), \, \mathcal P_- \iota^{-1} (\Psi(\cdot,0)) = 0 \textrm{ and } \Psi(\cdot,\delta) = 0 \right\rbrace,
\end{multline}
and
\begin{multline}
y^-_m [\Psi,\Psi] := \int_{\Pi^-_\delta} \left[ (1- c \delta) \vert \overline \nabla^\mathcal N \Gamma_t^0 \Psi \vert^2 + \vert \nabla^\mathcal N_\ddt \Psi \vert^2 \right] v_h \\
+ \int_{\Pi^-_\delta} \left[ \left(m^2 + \frac{\Scal^{\partial \mathcal K} - \textrm{Tr} (W^2)}{4} - c\delta\right) \vert \Psi \vert^2 \right] v_h \\
+ \int_{\partial \mathcal K} \left[m \vert \Psi (\cdot, 0) \vert^2 - c \vert \Psi (\cdot, \delta) \vert^2 \right] v_{\partial \mathcal K} \\
\mathcal Q(y^-_m) := \left\lbrace \Psi \in H^1 \left( \SS^-_\delta \right), \, \mathcal P_- \iota^{-1} \Psi(\cdot,0) = 0 \right\rbrace.
\end{multline}
Remarking that $\mathcal Q(y^\pm_m) = \Theta \iota (\dom(z^\pm_m))$, and that $\Theta \iota$ is unitary from $L^2(\Sigma \mathcal C_{\vert \Pi^-_\delta}, v_\mathcal N)$ onto $L^2 \left( \SS^-_\delta , v_h \right)$, Proposition~\ref{quadbound} and the Min-Max principle give
\begin{equation}
\Lambda_j\left(y^-_m \right) \leq E_j\left( A_m^2 \right) \leq \Lambda_j\left(y^+_m \right) \textrm{ for any $j \in \mathbb \mathcal N$ such that $\Lambda_j(y^+_m) < m^2$}.
\end{equation}

\subsection{Upper bound} The upper bound is found by taking good test functions in the Min-Max principle. The first observation is that the quadratic form $y^+_m$ admits a separation of variable. It can be seen as the tensor product of a sesquilinear form on $\partial \mathcal K$ and a one-dimensional sesquilinear form $S$. The behaviour of its first eigenvalue allows to find the bound we are searching for.

Let $S$ be the self-adjoint operator on $L^2(0, \delta)$ associated to the quadratic form
\begin{equation} \label{opS}
s[f,f] = \int_0^\delta \vert f' \vert^2 dt + m\vert f(0) \vert^2, \, \mathcal Q(s) = \left\lbrace f \in H^1(0,\delta), \, f(\delta) = 0 \right\rbrace,
\end{equation}
and let $f$ be a normalized eigenfunction for the first eigenvalue of $S$. According to Lemma~\ref{Sanalysis}, when $-m$ is large, there is $b > 0$ such that $S[f,f] + m^2 \leq b \exp(- \delta \vert m \vert )$. 

For $a > 0$, we introduce the quadratic form
\begin{multline} \label{definitionLa}
\ell_a[\Psi, \Psi] = \int_{\partial \mathcal K} \left\lbrack (1 + ca) \vert \overline \nabla^N \iota \Psi \vert^2 + \left(\frac{\Scal^{\partial \mathcal K} - \textrm{Tr} (W^2)}{4} + ca \right) \vert \Psi \vert^2 \right\rbrack v_{\partial \mathcal K}, \\
\mathcal Q(\ell_a) = \mathcal Q (\ell),
\end{multline}
where $\ell$ was defined in \eqref{quadL}.
The sesquilinear form $\ell_a$ is lower semibounded and closed. We denote by $L_a$ the associated self-adjoint operator.

Let $\xi_1,\ldots, \xi_j$ be linearly independant eigenspinors for the first $j$ eigenvalues of $L_\delta$. We define the set
\begin{equation}
V := \left\lbrace \Psi \in L^2 \left( \SS^-_\delta \right), \Psi(x,t) = f(t) \Gamma_0^t (\iota\xi(x)), \, \xi \in \Span(\xi_1,\ldots,\xi_j) \right\rbrace.
\end{equation}
With all these notations, for $\Psi(x,t) := f(t) \Gamma_0^t (\iota\xi(x)) \in V$ and $-m$ large enough, one has, using Leibniz's rule
\begin{multline*}
y^+_m[\Psi, \Psi] = \int_{\Pi^-_\delta} \left[\vert \nabla^\mathcal N_\ddt \Psi \vert^2 + (1+ c \delta) \vert \overline \nabla^\mathcal N \Gamma_0^t \Psi \vert^2 \right] v_h \\
+ \int_{\Pi^-_\delta} \left[ \left(m^2 + \frac{\Scal^{\partial \mathcal K} - \textrm{Tr} (W^2)}{4} + c\delta\right) \vert \Psi \vert^2 \right] v_h + m \int_{\partial \mathcal K} \vert \Psi (., 0) \vert^2 v_{\partial \mathcal K} \\
= \int_{\Pi^-_\delta} \left[\vert \ddt f \vert^2 \vert \xi \vert^2 + (1+ c \delta) \vert \overline \nabla^N \iota \xi \vert^2 \vert f \vert^2 \right] v_h \\
+ \int_{\Pi^-_\delta} \left[ \left(m^2 + \frac{\Scal^{\partial \mathcal K} - \textrm{Tr} (W^2)}{4} + c\delta\right) \vert \Psi \vert^2 \right] v_h + m \int_{\partial \mathcal K} \vert \Psi (\cdot, 0) \vert^2 v_{\partial \mathcal K} \\
= \ell_\delta[\xi,\xi] \Vert f \Vert^2_{L^2(0,\delta)} +\left(S[f,f] + m^2 \right) \Vert \xi \Vert^2_{L^2(\partial \mathcal K)} \\
\leq \ell_\delta[\xi,\xi] + b \exp(- \delta \vert m \vert) \Vert \xi \Vert^2_{L^2(\partial \mathcal K)} \\
\leq \left( E_j(L_\delta) + b \exp(- \delta \vert m \vert) \right) \Vert \xi \Vert^2_{L^2(\partial \mathcal K)}.
\end{multline*}
Thus, $\Lambda_j(y^+_m) \leq E_j(L_\delta) + b \exp(- \delta \vert m \vert)$. We remark that $\underset{\delta \rightarrow 0}{\textrm{lim }} E_j(L_\delta) = E_j(L)$ so we have the bound
\begin{equation} \label{upA}
\underset{m \rightarrow -\infty}{\textrm{lim sup }} E_j(A_m^2) \leq E_j(L).
\end{equation}

\subsection{Lower bound} The strategy to obtain the lower bound is to relax the constraint in the domain of $y^-_m$ in order to operate a separation of variable. Thus, we arrive at a good configuration to apply the monotone convergence theorem. This analysis will occupy the end of this section.

Let $S'$ be the self-adjoint operator on $L^2(0, \delta)$ associated to the quadratic form
\begin{equation}
S'[f,f] = \int_0^\delta \vert f' \vert^2 dt + m\vert f(0) \vert^2 - c \vert f(\delta) \vert^2, \, \mathcal Q(S') = H^1(0,\delta),
\end{equation}
and let $(f_k)_{k\in \NN}$ be a sequence of normalized eigenfunctions for the eigenvalues  $E_k(S')$. According to Lemma~\ref{S'analysis}, there exist $b^\pm > 0$, $b > 0$ and $b_0 > 0$ such that $E_1(S') \geq -m^2 - b e^{-\delta \vert m \vert}$ when $m \rightarrow -\infty$ and $b^- k^2 - b_0 \leq  E_k(S') \leq b^+ k^2$ for all $k \geq 2$.

If $c > 0$ is the constant given by Proposition~\ref{quadbound}, we define the quadratic form $y_m$ by the same formule as $y^-_m$, but with the domain $Q(y_m) = H^1 \left( \SS^-_\delta \right)$.

We also define for $a \in \RR$ the sesquilinear form
\begin{multline}
\ell'_a[\Psi, \Psi] = \int_{\partial \mathcal K} \left\lbrack (1 + ca) \vert \overline \nabla^N \iota \Psi \vert^2 \left(\frac{\Scal^{\partial \mathcal K} - \textrm{Tr} (W^2)}{4} + ca \right) \vert \Psi \vert^2 \right\rbrack v_{\partial \mathcal K}, \\
\mathcal Q(\ell'_a) = H^1(\Sigma \mathcal C_{\vert \partial \mathcal K}).
\end{multline}
This form is closed and lower semibounded. We note $L'_a$ the associated self-adjoint operator.

We state the following density result, which allows us to express $Y_m$ as the sum of tensorial product of operators.

\begin{lemma} \label{tens}
Let
\[
F := \left\lbrace \Psi, \, \exists (f, \Psi_0) \in L^2(0, \delta) \times L^2\left(\Sigma \mathcal C_{\vert \partial \mathcal K} \right), \, \Psi(x,t) = f(-t) \Gamma_0^t \, (\iota \Psi_0 (x)) \right\rbrace.
\]
Then, $\Span (F)$ is dense in $L^2\left( \Sigma \Pi_\delta^- \right)$.
Thus, one has a natural isomorphism $L^2\left( \SS^-_\delta, v_h \right) \cong L^2(0, \delta) \otimes L^2\left(\Sigma \mathcal C_{\vert \partial \mathcal K} \right)$.
\end{lemma}
\begin{proof}
Let $E := (-\delta, 0) \times \RR$ viewed as a vector bundle over $(-\delta, 0)$, and $P := E \otimes \Sigma \mathcal C_{\vert \partial \mathcal K}$. The statement of the lemma is then equivalent to the density of $\Span(F')$ in $L^2(P, v_h)$ where
\[
F' := \left\lbrace \Psi, \, \exists (f, \Psi_0) \in L^2(- \delta, 0) \times L^2\left(\Sigma \mathcal C_{\vert \partial \mathcal K} \right), \, \Psi(x,t) = f(t) \Psi_0 (x) \right\rbrace,
\]
and this last result is standard.
\end{proof}

We denote by $Y_m$ the self-adjoint operator associated to $y_m$, and using the identification of Lemma~\ref{tens}, one can write
\[
Y_m = (S' + m^2) \otimes 1 + 1 \otimes L'_{-\delta}.
\]
Now, we define the unitary transform
\begin{align*}
\mathcal U: L^2 \left( \SS^-_\delta \right) \longrightarrow \ell^2(\NN) \otimes L^2\left(\Sigma \mathcal C_{\vert \partial \mathcal K} \right) \\
\mathcal U \Psi = (\Psi_k), \; \Psi_k = \int_0^\delta f_k(t) \, \iota^{-1} \Gamma_t^0 (\Psi(\cdot, t)) \dd t.
\end{align*}
By the spectral theorem, $\Hat Y_m := \mathcal U Y_m \mathcal U^*$ is given by its quadratic form denoted by $\Hat y_m$:
\begin{align*}
\Hat y_m [(\Psi_k), (\Psi_k)] = \sum\limits_{k \in \NN} \left(\ell'_{-\delta} [\Psi_k, \Psi_k] + (E_k(S') +m^2) \Vert \Psi_k \Vert^2_{L^2(\partial \mathcal K)} \right),
\end{align*}
and the form domain is the subset of $\ell^2(\NN) \otimes L^2\left(\Sigma \mathcal C_{\vert \partial \mathcal K} \right)$ for which the right-hand side converges. Thus,
\begin{multline}
\mathcal Q (\Hat y_m) = \left\lbrace (\Psi_k) \in \ell^2(\NN) \otimes L^2 \left( \Sigma \mathcal C_{\vert \partial \mathcal K} \right), \, \Psi_k \in H^1 \left( \Sigma \mathcal C_{\vert \partial \mathcal K} \right) \right. \\
\left. \textrm{ and } \sum\limits \left( \Vert \Psi_k \Vert^2_{H^1(\partial \mathcal K)} + k^2 \Vert \Psi_k \Vert^2_{L^2(\partial \mathcal K)} \right) < \infty \right\rbrace.
\end{multline}
Setting $\Hat Y^-_m := \mathcal U Y^-_m \mathcal U^*$, the sesquilinear form for $\Hat Y^-_m$ is the same as for $\Hat Y_m$ with the domain
\begin{equation}
\mathcal Q(\Hat y^-_m) = \left\lbrace \Hat \Psi = (\Psi_k) \in \mathcal Q(\Hat y_m): \mathcal P_- \mathcal U^* \Hat \Psi (\cdot, 0) = 0 \right\rbrace.
\end{equation}
Then, if we define
\begin{multline}
w_m [\Hat \Psi, \Hat \Psi] := \ell'_{-\delta} [\Psi_1, \Psi_1] - b \exp(-\delta \vert m \vert) \Vert \Psi_1 \Vert^2_{L^2(\partial \mathcal K)}\\
+ \sum\limits_{k \geq 2} \ell'_{-\delta} [\Psi_k, \Psi_k] + (b^- k^2 - b_0 + m^2) \Vert \Psi_k \Vert^2_{L^2(\partial \mathcal K)}, \\
\mathcal Q (w_m) := \mathcal Q(\Hat y^-_m),
\end{multline}
we have $\Hat y^-_m \geq w_m$. The form $w_m$ is lower semibounded and closed. Let $W_m$ be the associated self-adjoint operator, and by Theorem~\ref{RellKond}, this operator has compact resolvent. For all $j \in \NN$, one has
\[
E_j(A_m^2) \geq \Lambda_j(y^-_m) = \Lambda_j(\Hat y^-_m) \geq E_j(W_m).
\]
We can now apply the monotone convergence theorem to the non-decreasing family of self-adjoint operators $(W_m)$. The form domain of the limit operator will be:
\begin{equation}
\mathcal Q_\infty := \left\lbrace \Hat \Psi = (\Psi_k) \in \underset{m < 0}{\bigcap} \mathcal Q(W_m), \, \underset{m < 0}{\sup} W_m[\Hat \Psi, \Hat \Psi] < \infty \right\rbrace.
\end{equation}
One has $\Hat \Psi := (\Psi_k) \in \mathcal Q_\infty$ iff $\Psi_k = 0$ for all $k \geq 2$ and $0 = \mathcal P_- \mathcal U^*  \Hat \Psi (\cdot, 0) = f_1(0) \mathcal P_- \Psi_1$. If we denote $e_1 := (1,0,0,\ldots) \in \ell^2(\NN)$ this gives
\[
\mathcal Q_\infty = \left\lbrace \Hat \Psi = e_1 \otimes \Psi_1 : \Psi_1 \in \mathcal Q(\ell) \right\rbrace.
\]
Thus, for any $\Hat \Psi \in \mathcal Q_\infty$ one has
\[
\underset{m \rightarrow - \infty}{\lim} W_m[\Hat \Psi, \Hat \Psi] = L_{-\delta}[\Psi_1, \Psi_1].
\]
We define the Hilbert space $\mathbf H_\infty := e_1  \otimes \mathbf H$ and the sesquilinear form
\begin{equation}
w_\infty [e_1 \otimes \Psi_1, e_1 \otimes \Psi_1] = L_{-\delta}[\Psi_1, \Psi_1], \, \mathcal Q(w_\infty) = \mathbf H_\infty.
\end{equation}
Let $W_\infty$ be the associated self-adjoint operator. By Corollary~\ref{monotonecor} (monotone convergence), one has $\underset{m \rightarrow - \infty}{\lim} E_j(W_m) = E_j(W_\infty) = E_j(L_{-\delta})$ for all $j \in \NN$. Letting $\delta$ go to zero we obtain
\begin{equation} \label{downA}
\underset{m \rightarrow \infty}{\textrm{lim inf }} E_j(A_m^2) \geq E_j (L).
\end{equation}
The estimates (\ref{upA}) and (\ref{downA}) together with Lemma~\ref{Uequi} give
\begin{equation}
\underset{m \rightarrow \infty}{\textrm{lim}} E_j \left( A_m^2 \right) = E_j \left( (\slashed D^{\partial \mathcal K})^2 \right).
\end{equation}

\begin{rem} \label{asym1}
With the help of the sesquilinear form, we can investigate another asymptotic regime. Let $\Psi \in \dom(A_m)$ and assume $m > 0$. Proposition~\ref{Aselfadjoint} gives that for $m$ large enough, $\Vert A_m \Psi \Vert^2_{L^2(\mathcal N)} \ge m^2 \Vert \Psi \Vert^2_{L^2(\mathcal N)}$. Hence, when $m \rightarrow + \infty$, one has $E_j(A_m) \rightarrow + \infty$ for all $j \in \NN$ by the Min-Max principle.
\end{rem}

\section{The operator $B_{m,M}^2$ in the limit of large $M$} \label{BforM}

We now prove Theorem~\ref{main2} following the lines of \cite[Section5]{MOP}. Again, this is done by finding a lower and an upper bound for the limit of the eigenvalues of $B_{m,M}^2$. The proof relies on the localization of the problem in a neighbourhood of $\mathcal K$ and the construction of an appropriated extension for the spinors in $\mathcal K$. For the lower bound, we make another use of the monotone convergence theorem to observe that the projection $\mathcal P_+$ on the boundary of $\mathcal K$ must vanish in the asymptotic regime. 

We begin with some preliminary estimates and the definition of the extension operator. We define $\SS^+_\delta := \iota (\Sigma \mathcal C_{\vert \overline{\Pi^+_\delta}})$.

\begin{lemma} \label{semi-boundr}
Let $r'_\alpha$ be the sesquilinear form given by
\[
r'_\alpha[\Psi, \Psi] := \int_{\mathcal K^c \setminus \Pi_\delta^+} \left(\vert  \nabla^\mathcal N \iota \Psi \vert^2 + \frac {\Scal^\mathcal N}{4} \vert \Psi \vert^2 \right) v_\mathcal N
\]
with $\mathcal Q (r'_\alpha) = \lbrace \Psi_{\vert \mathcal K^c \setminus \Pi_\delta^+}, \, \Psi \in \dom(B_{m,M}) \rbrace$.
Then, $r'_\alpha$ is lower semibounded.
\end{lemma}
\begin{proof}
Let $\Psi \in \mathcal Q (r'_\alpha)$. Let $\chi_1, \chi_2$ be two non-negative real smooth functions on $\mathcal N$ such that $\chi_1^2 + \chi^2 = 1$, $\chi_1$ is supported in $\mathcal K \cup \Pi^+_{\frac{3 \delta}{2}}$ and $\chi_2$ is supported in $\mathcal N \setminus (\mathcal K \cup \Pi^+_{\frac{5 \delta}{4}})$. 

An easy computation gives
\[
r'_\alpha[\Psi, \Psi] = r'_\alpha[\chi_1 \Psi, \chi_ 1 \Psi] + r'_\alpha[\chi_2 \Psi, \chi_2 \Psi] - \int_{\mathcal K^c \setminus \Pi_\delta^+} (\vert (\dd \chi_1) \iota \Psi \vert^2 + \vert (\dd \chi_2) \iota \Psi \vert^2) v_\mathcal N,
\]
and then there exists a constant $C_1 > 0$ such that
\[
r'_\alpha[\Psi, \Psi] \ge r'_\alpha[\chi_1 \Psi, \chi_1 \Psi] + r'_\alpha[\chi_2 \Psi, \chi_2 \Psi] - C_1 \Vert \Psi \Vert^2_{L^2(\mathcal N)}.
\]
Now, the Schr\"odinger-Lichnerowicz formula gives
\[
r'_\alpha[\chi_2 \Psi, \chi_2 \Psi] = \Vert \mathcal D^\mathcal N \chi_k \Psi \Vert^2_{L^2(\mathcal N)} \ge 0.
\]
Moreover, there exists $C_2 > 0$ such that
\[
r'_\alpha[\chi_1 \Psi, \chi_1 \Psi] \ge - C_2 \Vert \chi_1 \Psi \Vert^2_{L^2(\mathcal N)}
\]
because $\chi_1$ has compact support.

All together, we have $r'_\alpha[\Psi, \Psi] \ge - C \Vert \Psi \Vert^2_{L^2(\mathcal N)}$ for a constant $C > 0$.
\end{proof}

\begin{lemma} \label{extension}
For $\Psi \in \lbrace \Phi_{\vert \mathcal K^c}, \, \Phi \in \dom (B_{m,M}) \rbrace$ and $\alpha > 0$ we define the sesquilinear form
\[
r_\alpha [ \Psi, \Psi ] = \int_{\mathcal K^c} \left( \left\vert \nabla^\mathcal N \iota \Psi \right \vert^2 + \frac{\Scal^\mathcal N}{4} \vert \Psi \vert^2 \right) v_\mathcal N + \int_{\partial \mathcal K} \left( \frac{H}{2} - \alpha \right) \vert \Psi \vert^2 v_{\partial \mathcal K}.
\]
Then, there exists $C > 0$ such that for $\alpha > 0$ large enough, one has a map $F_\alpha : H^1(\iota(\Sigma \mathcal C_{\vert \partial \mathcal K})) \rightarrow \dom (r_\alpha)$ with $F_\alpha \Psi = \Psi$ on $\partial \mathcal K$ and
\[r_\alpha[F_\alpha \Psi, F_\alpha \Psi] + \alpha^2 \Vert F_\alpha \Psi \Vert^2_{L^2(\mathcal K^c)} \leq \frac{c}{\alpha} \Vert \Psi \Vert^2_{H^1(\partial \mathcal K)}.
\]
Moreover there is a constant $C_0 > 0$, such that $\Lambda_1(r_\alpha) \geq -\alpha^2 - C_0$.
\end{lemma}

\begin{proof}
We recall that for $\alpha > 0$ we defined the operator $S$ associated to the sesquilinear form
\begin{equation*}
s[f,f] = \int_0^\delta \vert f' \vert^2 dt - \alpha \vert f(0) \vert^2, \, \mathcal Q(s) = \left\lbrace f \in H^1(0,\delta), \, f(\delta) = 0 \right\rbrace.
\end{equation*}
Let $f$ be the first eigenfunction of the operator $S$ normalized by $f(0) = 1$.

We define the map $F_\alpha$ by
\[
F_\alpha \Psi (x) := \left\lbrace
\begin{array}{cc}
(\Theta\iota)^{-1} v (x) & \textrm{if $x \in \Pi_\delta^+$} \\
0 & \textrm{if $x \in \mathcal K^c \setminus \Pi_\delta^+$}
\end{array}
\right.
\]
where $v := f \otimes \Psi$. From Lemma~\ref{Sanalysis} there exists $C > 0$ such that $\Vert f \Vert^2_{L^2(0, \delta)} \le \frac{C}{\alpha}$ and $\alpha^2 + E_1(S) \le C e^{-\delta \alpha}$. Then, using Proposition~\ref{quadbound}, there is $a > 0$ such that
\begin{multline*}
r_\alpha[F_\alpha \Psi, F_\alpha \Psi] + \alpha^2 \Vert F_\alpha \Psi \Vert^2_{L^2(\mathcal K^c)} = J_\alpha(F_\alpha \Psi) + \alpha^2 \Vert f_\alpha \Vert^2_{L^2(\mathcal K^c)} \\
\leq \int_{\Pi_\delta^+} \left( a \vert \overline \nabla^\mathcal N \Gamma_t^0 v \vert^2 + \vert \nabla^\mathcal N_{\ddt} v \vert^2 + (\alpha^2 + a) \vert v \vert^2 \right) v_h - \alpha \int_{\partial \mathcal K} \vert \Psi \vert^2 v_{\partial \mathcal K} \\
= \int_{\partial \mathcal K} \left[ a \vert \overline \nabla^\mathcal N \Psi \vert^2 + (E_1(S) + \alpha^2+ a) \vert \Psi \vert^2 \right] v_h \Vert f \Vert^2_{L^2(0, \delta)} \\
\leq \frac{C (C+a)}{\alpha} \Vert \Psi \Vert^2_{H^1(\partial \mathcal K)}.
\end{multline*}

For the second assertion, we introduce the sesquilinear forms
\[
r^0_\alpha[\Psi, \Psi] := \int_{\Pi_\delta^+} \left(\vert \nabla^\mathcal N \iota \Psi \vert^2 + \frac {\Scal^\mathcal N}{4} \vert \Psi \vert^2 \right) v_\mathcal N + \int_{\partial \mathcal K} \left( \frac{H}{2} - \alpha \right) \vert \Psi \vert^2 v_{\partial \mathcal K}
\]
with $\mathcal Q (r^0_\alpha) = \lbrace \Psi_{\vert \Pi_\delta^+}, \, \Psi \in \dom (B_{m,M}) \rbrace$ and
\[
r'_\alpha[\Psi, \Psi] := \int_{\mathcal K^c \setminus \Pi_\delta^+} \left(\vert  \nabla^\mathcal N \iota \Psi \vert^2 + \frac {\Scal^\mathcal N}{4} \vert \Psi \vert^2 \right) v_\mathcal N
\]
with $\mathcal Q (r'_\alpha) = \lbrace \Psi_{\vert \mathcal K^c \setminus \Pi_\delta^+}, \, \Psi \in \dom(B_{m,M}) \rbrace$. One has the inequality $\Lambda_1 (r_\alpha) \ge \min(\Lambda_1 (r'_\alpha), \Lambda_1(r^0_\alpha))$. Since $r'_\alpha$ is lower semibounded by Lemma~\ref{semi-boundr}, another use of Proposition~\ref{quadbound} gives that when $\alpha$ is large $\Lambda_1(r_\alpha^0) \geq \Lambda_1(q_\alpha)$ with
\begin{multline*}
q_\alpha[\Psi, \Psi] = \int_{\Pi^+_\delta} \left[\frac{1}{a} \vert \overline \nabla^\mathcal N \Gamma_t^0 \Psi \vert^2 + \vert \nabla^{h}_\ddt \Psi \vert^2 - a \vert \Psi \vert^2 \right] v_h \\
- \alpha \int_{\partial \mathcal K} \vert \Psi(\cdot, 0) \vert^2 v_{\partial \mathcal K} - a \int_{\partial \mathcal K} \vert \Psi(\cdot, \delta) \vert^2 v_{\partial \mathcal K}
\end{multline*}
where $a > 0$ and $\mathcal Q (q_\alpha) = H^1 \left( \SS_\delta^+ \right)$. Trivializing locally the vector bundle via parallel sections along the normal geodesics and using Fubini's theorem we deduce that $\Lambda_1 (r_\alpha) \ge \Lambda_1(q_\alpha) \geq \Lambda_1(S') - a \geq - \alpha^2 - C$ with $C > 0$ when $\alpha \rightarrow + \infty$.
\end{proof}

Using Proposition~\ref{quadB}, the sesquilinear form for $B_{m,M}^2$ can be written for any spinor $\Psi \in \dom (B_{m,M})$ and any $\varepsilon > 0$ by
\begin{multline} \label{quadBrep}
\Vert B_{m,M} \Psi \Vert^2_{L^2(\mathcal N)} =  \int_\mathcal K \left[ \vert \nabla^\mathcal N (\iota\Psi)  \vert^2 + \left( \frac{\Scal^\mathcal N}{4} + m^2 \right) \vert \Psi \vert^2 \right] v_\mathcal N \\
+\int_{\partial \mathcal K} \left( m-\varepsilon - \frac{H}{2} \right) \vert \Psi \vert^2 v_{\partial \mathcal K}
+ 2(M-m) \int_{\partial \mathcal K} \vert \mathcal P_- \Psi \vert^2 v_{\partial \mathcal K} \\
+ \int_{\mathcal K^c} \left[ \vert \nabla^\mathcal N (\iota\Psi)  \vert^2 + \left( \frac{\Scal^\mathcal N}{4} + M^2 \right) \vert \Psi \vert^2 \right] v_\mathcal N -\int_{\partial \mathcal K} \left( M - \varepsilon - \frac{H}{2} \right) \vert \Psi \vert^2 v_{\partial \mathcal K}
\end{multline}
where we recall that $\mathcal P_- = \frac{1 - i \nu \cdot \nn \cdot}{2}$.

\subsection{Upper bound} We are now able to find an upper bound for the limit of $E_j(B_{m,M}^2)$ when $M \rightarrow + \infty$ for $j \in \NN$. Let $\eta > 0$ and pick $(\Psi_1,\ldots,\Psi_j)$ in $\Gamma(\Sigma \mathcal C_{\vert \mathcal K})$, smooth spinors such that
\[
\inf_{\substack{\Psi \in \Span(\Psi_1, \ldots, \Psi_j)}} \frac{\langle A_m^2 \Psi, \Psi \rangle_{L^2(\mathcal K)}}{\Vert \Psi \Vert^2_{L^2(\mathcal K)}} \le E_j(A_m^2) + \eta.
\]
We define $a := \sup \left\lbrace \Vert \Psi \Vert^2_{H^1(\partial \mathcal K)}, \Psi \in \Span (\Psi_1,\ldots, \Psi_j), \Vert \Psi \Vert_{L^2(\mathcal K)} = 1 \right\rbrace$. Let $\Psi \in V := \Span(\Psi_1, \ldots, \Psi_j)$ and
\begin{align*}
\Tilde \Psi := \left\lbrace
\begin{array}{cc}
\Psi & \textrm{ in } \mathcal K \\
F_M (\Psi_{\vert \partial \mathcal K}) & \textrm{ in $\mathcal K^c$}.
\end{array}
\right.
\end{align*}
Lemma~\ref{extension} gives us a constant $C > 0$ such that
\begin{multline*}
\int_{\mathcal K^c} \left[ \vert \nabla^\mathcal N (\iota \Tilde \Psi)  \vert^2 + \left( \frac{\Scal^\mathcal N}{4} + M^2 \right) \vert \Tilde \Psi \vert^2 \right] v_\mathcal N -\int_{\partial \mathcal K} \left( M - \frac{H}{2} \right) \vert \Tilde \Psi \vert^2 v_{\partial \mathcal K} \\
= r_M[\Tilde \Psi, \Tilde \Psi] + M^2 \Vert \Tilde \Psi \Vert^2_{L^2(\mathcal K^c)} \leq \frac{C}{M} \Vert \Tilde \Psi \Vert^2_{H^1(\partial \mathcal K)} \leq \frac{Ca}{M} \Vert \Psi \Vert^2_{L^2(\mathcal K)},
\end{multline*}
and then, using the expression \eqref{quadBrep} with $\varepsilon = 0$,
\begin{align*}
\Vert B_{m,M} \Tilde \Psi \Vert^2_{L^2(\mathcal N)} &\leq A_m^2 [\Psi, \Psi] + \frac{Ca}{M} \Vert \Psi \Vert^2_{L^2(\mathcal K)} \leq \left( E_j(A_m^2) + \eta + \frac{Ca}{M} \right) \Vert \Psi \Vert^2_{L^2(\mathcal K)} \\
&\leq \left( E_j(A_m^2) + \eta + \frac{Ca}{M} \right) \Vert \Tilde \Psi \Vert^2_{L^2(\mathcal K)}
\end{align*}
and letting $\eta$ go to zero one has $\textrm{lim sup}_{M \rightarrow +\infty} E_j(B_{m,M}^2) \leq E_j(A_m^2)$.

\subsection{Lower bound} It remains to find a lower bound for the eigenvalues. In order to do so, we separate the representation \eqref{quadBrep} in two parts corresponding to $\mathcal K$ and $\mathcal K^c$ and we remark that the outer part becomes very large when $M$ goes to $+ \infty$ so the eigenvalues must converge to the eigenvalues of an operator in $\mathcal K$.

Let $j \in \NN$. One has
\[
E_j(B_{m,M}^2) \geq \min \left\lbrace \Lambda_j(k^c_{M,\varepsilon}), E_j(K_{m,M,\varepsilon}) \right\rbrace
\]
where $K_{m,M,\varepsilon}$ is the operator associated to the sesquilinear form
\begin{multline}
k_{m,M,\varepsilon}[\Psi,\Psi] := \int_\mathcal K \left( \vert \nabla^\mathcal N \iota \Psi \vert^2 + \left( m^2 + \frac{\Scal^\mathcal N}{4} \right) \vert \Psi \vert^2 \right) v_\mathcal N \\
+ \int_{\partial \mathcal K} (m- \varepsilon - \frac{H}{2}) \vert \Psi \vert^2 v_{\partial \mathcal K} + 2 (M-m) \int_{\partial \mathcal K} \vert P_- \Psi \vert^2 v_{\partial \mathcal K}
\end{multline}
and $k^c_{M,\varepsilon}$ is the sesquilinear form
\begin{multline}
k^c_{M,\varepsilon}[\Psi,\Psi] := \int_{\mathcal K^c} \left( \vert \nabla^\mathcal N \iota \Psi \vert^2 + \left( M^2 + \frac{\Scal^\mathcal N}{4} \right) \vert \Psi \vert^2 \right) v_N \\
- \int_{\partial \mathcal K} (M - \varepsilon - \frac{H}{2}) \vert \Psi \vert^2 v_{\partial \mathcal K}
\end{multline}
where the respective domains are the restrictions of $\dom(B_{m,M})$ to $\mathcal K$ and $\mathcal K^c$.

One has $k^c_{M,\varepsilon} = r_{M-\varepsilon} + M^2$, where $r_{M-\varepsilon}$ was defined in Lemma~\ref{extension}), and the same lemma gives
\begin{align*}
\Lambda_1(k^c_{M,\varepsilon}) = \Lambda_1(r_{M-\varepsilon} + M^2) &\geq -(M-\varepsilon)^2 - C_0 + M^2 = 2 \varepsilon M - \varepsilon^2 - C_0 \\
&= \varepsilon M + (\varepsilon M - \varepsilon^2 - C_0) \geq \varepsilon M \textrm{ when $M \rightarrow +\infty$}.
\end{align*}

It comes out that $E_j(B_{m,M}^2) = E_j(K_{m,M,\varepsilon})$ when $M \rightarrow +\infty$. But $k_{M,m,\varepsilon}$ increases when $M$ does, and
\[
k_{M,m,\varepsilon}[\Psi, \Psi] \underset{M \rightarrow +\infty}{\longrightarrow} \left\langle A_m \Psi, A_m \Psi \right\rangle_{L^2(\mathcal K)} - \varepsilon \Vert \Psi \Vert_{L^2(\partial \mathcal K)}.
\]
Furthermore,
\[
\left\lbrace \Psi \in \bigcap\limits_{M > 0} \dom (k_{m,M,\varepsilon}), \, \lim_{M\rightarrow + \infty} k_{m,M,\varepsilon}[\Psi, \Psi] < \infty \right\rbrace = \dom(A_m),
\]
thus, by monotone convergence (Corollary~\ref{monotonecor}) and letting $\varepsilon$ go to zero, one has $\lim \inf_{M \rightarrow + \infty} E_j(B_{m,M}^2) = E_j(A_m^2)$. With the upper bound, one has $\lim_{M \rightarrow +\infty} E_j(B_{m,M}^2) = E_j(A_m^2)$.

\section{The operator $B_{m,M}$ for large masses} \label{BformM}

In this section, we give a proof of Theorem~\ref{main3}, so we are investigating the asymptotic regime $m \rightarrow - \infty$ and $M \rightarrow + \infty$. The method we use is very similar to the one of section~\ref{BforM}. The difference lies in the proof of the lower bound, where we do not make the analyse on the opeator outside and inside $\mathcal K$, but we rather cut the space in three pieces: the tubular neighbourhood of $\partial \mathcal K$, and the remaining regions lying inside and outside of the compact $\mathcal K$. By Dirichlet-Neumann bracketing, it is then sufficient to study the operator restricted to the tubular neighbourhood to conclude.

\subsection{Lower bound} In this section, we write $\SS_\delta := \iota \left( \Sigma \mathcal C_{\vert \overline{\Pi_\delta}} \right)$. We recall that for $\alpha \in \RR$ we defined the operator
\begin{equation}
S_\alpha [f,f] = \int_0^\delta \vert f' \vert^2 \dd t - \alpha \vert f(0) \vert^2, \; \mathcal Q (S_\alpha) = \left\lbrace f \in H^1(0, \delta), f(\delta) = 0 \right\rbrace,
\end{equation}
and denoting by $f_\alpha$ the $L^2$-normalized eigenfunction associated to $E_1(S_\alpha)$, one has $\vert f_\alpha (0) \vert^2 = 2 \alpha + \mathcal O (1)$ and $E_1(S_\alpha) = \alpha^2 + \mathcal O( e^{-\alpha \delta})$ when $\alpha \rightarrow + \infty$ (see Lemma~\ref{Sanalysis}).

The operator $L_a$ was defined by the quadratic form \eqref{definitionLa}.

Let $j \in \NN$ and $\Psi_1, \ldots , \Psi_j$ be $j$ eigenspinors for the first $j$ eigenvalues of $L_\delta$. For $\Psi \in V := \Span (\Psi_1, \ldots, \Psi_j)$, we define the extension operator $\mathcal E : H^1(\SS_\delta) \rightarrow H^1(\Sigma \mathcal C_{\vert \mathcal N})$ by
\begin{equation}
\mathcal E \Psi :=
\left\lbrace
\begin{array}{cc}
\frac{\vert f_{-m}(0) \vert}{\vert f_{M}(0) \vert} (\Theta \iota)^{-1} (\Psi \otimes f_M) & \textrm{in $\Pi_\delta^+$} \\
(\Theta \iota)^{-1} (\Psi \otimes f_{-m}) & \textrm{in $\Pi_\delta^-$} \\
0 & \textrm{in $\mathcal N \setminus \Pi_\delta$}
\end{array}
\right. .
\end{equation}
One easily sees that $\Vert \mathcal E \Psi \Vert^2_{L^2(\mathcal N)} = \left( 1 + \left(\frac{f_{-m}(0)}{f_{M}(0)} \right)^2 \right) \Vert \Psi \Vert^2_{L^2(\partial \mathcal K)}$, and the operator $\mathcal E$ is injective. We use the expression \eqref{quadBrep} and Proposition~\ref{quadbound} to compute:
\begin{multline*}
\Vert B_{m,M}^2 \mathcal E \Psi \Vert^2_{L^2(\mathcal N)} = \int_\mathcal K \left[ \vert \nabla^\mathcal N (\iota \mathcal E \Psi)  \vert^2 + \left( \frac{\Scal^\mathcal N}{4} + m^2 \right) \vert \mathcal E \Psi \vert^2 \right] v_\mathcal N \\
+\int_{\partial \mathcal K} \left( m-\varepsilon - \frac{H}{2} \right) \vert \mathcal E \Psi \vert^2 v_{\partial \mathcal K} + \int_{\mathcal K^c} \left[ \vert \nabla^\mathcal N (\iota \mathcal E \Psi)  \vert^2 + \left( \frac{\Scal^\mathcal N}{4} + M^2 \right) \vert \mathcal E \Psi \vert^2 \right] v_\mathcal N \\
- \int_{\partial \mathcal K} \left( M - \varepsilon - \frac{H}{2} \right) \vert \mathcal E \Psi \vert^2 v_{\partial \mathcal K} \\
\le \int_{\Pi_\delta^-} \left[ (1 + c \delta) \left\vert (\overline \nabla^\mathcal N \Gamma_ t^0 \Psi \otimes f_{-m}) (x,0) \right\vert^2 + \vert \nabla^\mathcal N_\ddt \Psi \otimes f_{-m} \vert^2 \right] v_h \\
+ \int_{\Pi_\delta^-} \left[ \left(\frac{\Scal^{\partial \mathcal K} - \mathrm{Tr} (W^2)}{4} + m^2 + c\delta\right) \vert \Psi \otimes f_{-m} \vert^2 \right] v_{\partial \mathcal K} \dd t \\
+ \int_{\Pi_\delta^+} \left[ (1 + c \delta) \left\vert (\overline \nabla^\mathcal N \Gamma_ t^0 \Psi \otimes f_M) (x,0) \right\vert^2 + \vert \nabla^\mathcal N_\ddt \Psi \otimes f_M \vert^2 \right] v_h \\
+ \int_{\Pi_\delta^+} \left[ \left(\frac{\Scal^{\partial \mathcal K} - \mathrm{Tr} (W^2)}{4} + M^2 + c\delta\right) \vert \Psi \otimes f_M \vert^2 \right] v_{\partial \mathcal K} \dd t \\
+ \int_{\partial \mathcal K} (- m \vert \Psi \otimes f_{-m} ( \cdot , 0) \vert^2 + M \vert (\Psi \otimes f_M) ( \cdot , 0) \vert^2) v_{\partial \mathcal K} \\
\le \left( 1 + \left(\frac{f_{-m}(0)}{f_{M}(0)} \right)^2 \right) \left[ \ell_\delta [\Psi,\Psi] + C \Vert \Psi \Vert^2_{L^2(\partial \mathcal K)} \left( e^{-M\delta} + e^{- \vert m \vert \delta} \right) \right]
\end{multline*}
where $C > 0$.

The Min-Max principle gives
\begin{align*}
E_j(B_{m,M}^2) &\le \underset{\Psi \in V}{\sup} \, \frac{B_{m,M}^2 [\mathcal E \Psi, \mathcal E \Psi]}{\Vert \mathcal E \Psi \Vert^2_{L^2(\mathcal N)}} \\
&\le \underset{v \in V}{\sup}  \left[ L_\delta [\Psi,\Psi] + C \Vert \Psi \Vert^2_{L^2(\partial \mathcal K)} \left( e^{-M\delta} + e^{- \vert m \vert \delta} \right) \right] \Vert \Psi \Vert^{-2}_{L^2(\partial \mathcal K)} \\
&\le E_j(L_\delta) + C \left( e^{-M\delta} + e^{- \vert m \vert \delta} \right).
\end{align*}

We now let $\min (- m, M) \rightarrow + \infty$, so we obtain
\[
\underset{\min (- m, M) \rightarrow + \infty}{\textrm{lim sup}} E_j(B_{m,M}^2) \le E_j(L_\delta)
\]
but $\delta$ can be taken arbitrary small, and one has the obvious limit $E_j(L_\delta) \underset{\delta \rightarrow 0}{\longrightarrow} E_j(L)$, so we arrive at
\begin{equation}
\underset{\min (- m, M) \rightarrow + \infty}{\textrm{lim sup}} E_j(B_{m,M}^2) \le E_j(L_\delta).
\end{equation}

\subsection{Lower bound} We consider the lower semibounded sesquilinear forms
\begin{multline}
k_{m,M} [\Psi,\Psi] = \int_{\mathcal N \setminus \Pi_\delta} \left[ \vert \nabla^\mathcal N (\iota\Psi)  \vert^2 + \left( \frac{\Scal^\mathcal N}{4} + m^2 \mathbf 1_{\mathcal K} + M^2 \mathbf 1_{\mathcal K^c} \right) \vert \Psi \vert^2 \right] v_\mathcal N \\
\mathcal Q(K_{m,M}) = \lbrace \Psi_{\mathcal N \setminus \Pi_\delta}, \, \Psi \in \dom(B_{m,M}) \rbrace
\end{multline}
and
\begin{multline}
k'_{m,M} [u,u] = \int_{\Pi_\delta^-} \left[ \vert \nabla^\mathcal N (\iota\Psi)  \vert^2 + \left( \frac{\Scal^\mathcal N}{4} + m^2 \right) \vert \Psi \vert^2 \right] v_\mathcal N \\
+\int_{\partial \mathcal K} \left( m-\varepsilon - \frac{H}{2} \right) \vert \Psi \vert^2 v_{\partial \mathcal K}
+ 2(M-m) \int_{\partial \mathcal K} \vert \mathcal P_- \Psi \vert^2 v_{\partial \mathcal K} \\
+ \int_{\Pi_\delta^+} \left[ \vert \nabla^\mathcal N (\iota\Psi)  \vert^2 + \left( \frac{\Scal^\mathcal N}{4} + M^2 \right) \vert \Psi \vert^2 \right] v_\mathcal N -\int_{\partial \mathcal K} \left( M - \varepsilon - \frac{H}{2} \right) \vert \Psi \vert^2 v_{\partial \mathcal K}, \\
\mathcal Q(K'_{m,M}) = H^1(\Sigma \mathcal C_{ \overline{\Pi_\delta}}).
\end{multline}
We denote by $K'_{m,M}$ the operator associated to $k'_{m,M}$.

Let $j \in \NN$. The Min-Max principle gives the lower estimate $E_j(B_{m,M}^2) \ge \min (E_j(K'_{m,M}), \Lambda_1(k_{m,M}))$, and by Lemma~\ref{semi-boundr} there is a constant $C > 0$ such that $\Lambda_1(k_{m,M}) \ge \min (m^2, M^2) - C$. This last quantity goes to $+\infty$ in the asymptotic regime under consideration, and we know thanks to the upper bound that $E_j(B_{m,M}^2) = \mathcal O(1)$. Thus, in the asymptotic regime one has $E_j(B_{m,M}^2) \ge E_j(K'_{m,M})$.

We now apply a transformation to look at the operator $K'_{m,M}$ written in tubular coordinates, and we consider the operator $P_{m,M}$ associated to the quadratic form
\begin{multline} \label{Pdef}
p_{m,M} [\Psi,\Psi] = \int_{\Pi_\delta} \left[ (1 - c \delta) \left\vert (\overline \nabla^\mathcal N \Gamma_ t^0 \Psi) (x,0) \right\vert^2 + \vert \nabla^\mathcal N_\ddt \Psi \vert^2 \right] v_h \\
+ \int_{\Pi_\delta} \left[ \left(\frac{\Scal^{\partial \mathcal K} - \mathrm{Tr} (W^2)}{4} + m^2 \mathbf 1_\mathcal K + M^2 \mathbf 1_{\mathcal K^c} - c \delta \right) \vert \Psi \vert^2 \right] v_{\partial \mathcal K} \dd t \\
+ (m - M) \int_{\partial \mathcal K} \vert \Psi ( \cdot , 0) \vert^2 v_{\partial \mathcal K} - c \int_{\partial \mathcal K} \vert \Psi (\cdot , \delta) \vert v_{\partial \mathcal K} + 2 (M - m) \int_{\partial \mathcal K} \vert \mathcal P_- \Psi \vert^2 v_\mathcal K, \\
\mathcal Q(p_{m,M}) = H^1(S_\delta),
\end{multline}
where $c > 0$ is chosen so that Lemma~\ref{S'analysis} is valid, so $E_j(K'_{m,M}) \ge E_j(P_{m,M})$.

For $a \in \RR$, let $L''_a$ be the operator given by the sesquilinear form $\ell''_a$ having the same expression as \eqref{definitionLa} but with the domain $H^1(\Sigma \mathcal C_{\vert \partial \mathcal K})$.

Let $P'_{m,M}$ be the operator defined by the same quadratic form as in \eqref{Pdef} but without the term involving the operator $\mathcal P_-$. We recall that the one-dimensional operator X was defined by \eqref{Xdefinition}, so one has
\[
P'_{m,M} = \ell''_{-\delta} \otimes 1 + 1 \otimes X.
\]
Let $(f_k)$ be a sequence of $L^2$-normalized eigenfunctions for the eigenvalues $E_k(X)$. We define the unitary transform
\begin{align*}
\mathcal U: L^2 \left( \SS_\delta \right) \longrightarrow \ell^2(\NN) \otimes L^2\left(\Sigma \mathcal C_{\vert \partial \mathcal K} \right) \\
\mathcal U \Psi = (\Psi_k), \; \Psi_k = \int_{-\delta}^\delta f_k(t) \, \iota^{-1} \Gamma_t^0 (\Psi(t, \cdot)) \dd t.
\end{align*}
Let $\Hat P'_{m,M} := \U P'_{m,M} \U^*$. This is a self-adjoint operator acting on $\ell^2(\NN) \otimes L^2\left(\Sigma \mathcal C_{\vert \partial \mathcal K} \right)$. One can write
\begin{multline}
\Hat P'_{m,M} [\Hat v, \Hat v] = \sum\limits_{k \in \NN} \left( \ell''_{-\delta} [\Psi_k, \Psi_k] + E_k(X) \Vert \Psi_k \Vert^2_{L^2(\Sigma)} \right), \\
\mathcal Q(\Hat P'_{m,M}) = \left\lbrace \Hat \Psi \in \ell^2(\NN) \otimes L^2 (\Sigma \mathcal C_{\vert \partial \mathcal K}), \Psi_k \in H^1(\Sigma \mathcal C_{\vert \partial \mathcal K}), \right. \\
\left. \sum\limits_{k \in \NN} \left( \Vert \Psi_k \Vert^2_{H^1(\partial \mathcal K)} + k^2 \Vert \Psi_k \Vert^2_{L^2(\partial \mathcal K)} \right) \right\rbrace.
\end{multline}

The operator $\Hat P_{m,M} = \U^* P_{m,M} \U$ has the same form domain as $\Hat P'_{m,M}$ and 
\[
\Hat P_{m,M} [\Hat \Psi, \Hat \Psi] = \sum\limits_{k \in \NN} \left( \ell''_{-\delta} [\Psi_k, \Psi_k] + E_k(X) \Vert \Psi_k \Vert^2_{L^2(\Sigma)} \right) + 2 (M + \vert m \vert) \int_\Sigma \vert \mathcal P_- \U^* \Hat \Psi \vert^2 \dd s.
\]
where the operator $X$ was defined in \eqref{Xdefinition}. We set
\begin{equation}
\zeta := \min (M, - m).
\end{equation}

Using Lemma~\ref{lemma1}, we consider the quadratic form $w_\zeta$ defined by
\begin{multline}
w_\zeta [\Hat \Psi, \Hat \Psi] = \ell''_{-\delta} [\Psi_1, \Psi_1] - C e^{- \zeta \delta/2} + 4 \zeta \int_\Sigma \vert \mathcal P_- \U^* \Hat \Psi \vert^2 \dd s \\
+ \sum\limits_{k \ge 2} \left( \ell''_{-\delta} [\Psi_k, \Psi_k] + (C_1 k^2 - C_2) \Vert \Psi_k \Vert^2_{L^2(\Sigma, \CC^N)} + \zeta^2 \Vert \Psi_k \Vert^2_{L^2(\Sigma)} \right), \\
\mathcal Q(w_\zeta) = \mathcal Q(\Hat P_{m,M}),
\end{multline}
and we claim that $\Hat P_{m,M} \ge w_\zeta$ for a suitable $C > 0$.
The  form $w_\zeta$ is lower semibounded and closed, and we define the associated self-adjoint operator $W_\zeta$ with compact resolvent. The previous discussion gives the lower estimate $E_j(B_{m,M}^2) \ge E_j(W_\zeta)$ in the asymptotic regime.

We now want to apply the monotone convergence theorem and we define
\begin{equation}
\mathcal Q_\infty = \left\lbrace \Hat \Psi \in \bigcap\limits_{\zeta > 0} \mathcal Q(W_\zeta) = \mathcal Q(w_\zeta), \; \underset{\zeta > 0}{\sup} \, w_\zeta [\Hat \Psi, \Hat \Psi] < + \infty \right\rbrace.
\end{equation}
One easily see that $\Hat \Psi$ is in $\mathcal Q_\infty$ iff $\Psi_k = 0$ for all $k \ge 2$ and $\mathcal P_- \U^* \Hat \Psi = 0$, which is equivalent to $\Hat \Psi = e_1 \otimes \Psi_1$ with $e_1 := (1, 0,0, \ldots)$ and $\mathcal P_- \Psi_1 = 0$.
It comes out  that $\mathcal Q_\infty = \left\lbrace e_1 \otimes \Psi_1 : \Psi_1 \in H^1(\Sigma, \CC^N) \cap \mathcal H \right\rbrace$.
Moreover, we have
\begin{equation}
\underset{\zeta \rightarrow \infty}{\lim} W_\zeta [e_1 \otimes \Psi_1, e_1 \otimes \Psi_1] = L_{-\delta} [\Psi_1,\Psi_1].
\end{equation}
Thus, if we define the operator $W_\infty [e_1 \otimes \Psi_1, e_1 \otimes \Psi_1] := L_{-\delta} [\Psi_1,\Psi_1]$ on $e_1 \otimes \mathcal H$, the monotone convergence theorem gives $\underset{\zeta \rightarrow \infty}{\lim} E_j(W_\zeta) = E_j(L_{-\delta})$. All together, we arrive at $\underset{\min (- m, M) \rightarrow +\infty}{\textrm{lim inf}} E_j(B_{m,M}^2) \ge E_j(L_{-\delta})$. We now let $\delta$ goes to zero and we obtain $\underset{\min (- m, M) \rightarrow +\infty}{\textrm{lim inf}} E_j(B_{m,M}^2) \ge E_j(L)$. The upper and the lower bound together give
\begin{equation}
\underset{\min (- m, M) \rightarrow +\infty}{\lim} E_j(B_{m,M}^2) = E_j(L) = E_j \left( (\slashed D^{\partial \mathcal K})^2 \right).
\end{equation}

\begin{rem} \label{asym2}
We can look at the asymptotic regime $M \to + \infty$ and $m \rightarrow + \infty$. Let $(m_k, M_k)_{k \in \NN}$ be a sequence of $\RR^2$ such that $m_k, M_k \underset{k \rightarrow + \infty}{\longrightarrow} + \infty$. In this case, we can use the inequality $E_1(B_{m,M}^2) \ge E_1(P_{m,M})$, and for any $\Psi \in \mathcal Q(p_{m,M})$ there exists a constant $C > 0$ such that
\begin{multline*}
p_{m,M} [\Psi,\Psi] \ge \int_{\Pi_\delta} \vert \nabla^\mathcal N_\ddt \Psi \vert^2 v_h + \int_{\Pi_\delta} \left[ m^2 \mathbf 1_{(0, \delta)} + M^2 \mathbf 1_{(- \delta, 0)} - C \right] \vert \Psi \vert^2 v_h \\
- C \int_{\partial \mathcal K} \vert \Psi (\cdot , \delta) \vert v_{\partial \mathcal K} - \vert M - m \vert \int_{\partial \mathcal K} \vert \Psi \vert^2 v_\mathcal K.
\end{multline*}
Without loss of generality, we can assume that there is a subsequence of $(M_k, m_k)$ still denoted by $(M_k, m_k)$ such that $M_k \ge m_k$ for all $k$. We have
\begin{align*}
p_{m_k,M_k} [\Psi,\Psi] &\ge m_k^2 \Vert \Psi \Vert^2_{L^2(\Pi_\delta^-)} + \Vert \Psi \Vert^2_{L^2(\Pi_\delta^+)} (M_k^2 + E_1(S_{M_k-m_k})) - C \Vert \Psi \Vert^2_{L^2(\Pi_\delta)},
\end{align*}
but when $k$ is large there is a constant $C_1$ such that
\begin{align*}
M_k^2 + E_1(S_{M_k-m_k}) &\ge M_k^2 - M_k^2 - m_k^2 + 2 M_k m_k - C_1 \\
&\ge 2 M_k m_k - m_k^2 - C_1 \ge m_k^2 - C_1.
\end{align*}
Thus, $E_1(B_{m_k,M_k}^2) \ge E_1(P_{m_k,M_k}) \ge m_k^2 - C - C_1 \underset{k \rightarrow + \infty}{\longrightarrow} + \infty$. This means that every sequence $E_1(B_{m_k, M_k}^2)$ admits a divergent subsequence, an we conclude that $E_1(B_{m, M}^2) \rightarrow + \infty$ in this regime.

By similar constructions, the result holds for $m,M\to -\infty$ as well.
\end{rem}

\renewcommand{\refname}


{\bf References}
\begin{thebibliography}{999}

\small

\addtolength{\itemsep}{3pt}

\bibitem{ALTR18}N.~Arrizabalaga, L.~Le~Treust, A. Mas, N.~Raymond:
\emph{The MIT Bag Model as an infinite mass limit}.
Preprint \url{http://arXiv.org/abs/1808.09746}.

\bibitem{ALTR16}
N.~Arrizabalaga, L.~Le~Treust, N.~Raymond:
{\it On the MIT Bag model in the non-relativistic limit.}
Comm. Math. Phys. {\bf 354} (2017) 641--669.

\bibitem{BGM} C.Bär, P.Gauduchon, A.Moroianu, \textit{Generalized cylinders in semi-Riemannian and spin geometry}, Mathematische Zeitschrift, Springer Science and Business Media LLC, 2005

\bibitem{BHSW}
J. Behrndt, S. Hassi, H. de Snoo, R. Wietsma:
\emph{Monotone convergence theorems for semibounded operators and forms with applications.}
Proc. Roy. Soc. Edinburgh Sect. A {\bf 140} (2010) 927--951.

\bibitem{BH3M} J-P.Bourguignon, O.Hijazi, J-L.Milhorat, A.Moroianu, S.Moroianu, \textit{A Spinorial Approach to Riemannian and Conformal Geometry}, European Mathematical Society, 2015.

\bibitem{Gin} N.Ginoux: \textit{The Dirac spectrum}, Lecture notes in Mathematics, vol. 1976, Springer-Verlag, Berlin, 2009.

\bibitem{Gri} P. Grisvard, \textit{Elliptic problems in nonsmooth domains}, volume 24 of Mono-graphs and Studies in Mathematics, Pitman (Advanced Publishing Program), Boston, MA, 1985.

\bibitem{GN} N. Grosse and R.Nakad (2018), \textit{Boundary value problems for noncompact boundaries of $\Spin^c$ manifolds and spectral estimates}, Proceedings of the London Mathematical Society, 109: 946-974.

\bibitem{GS} N. Grosse and C. Schneider (2013), \textit{Sobolev spaces on Riemannian manifolds with bounded geometry: general coordinates and traces}, Math. Nachr., 286: 1586-1613.

\bibitem{HMZ} O. Hijazi, S. Montiel, et X. Zhang, \textit{Dirac Operator on Embedded Hypersurfaces}, Mathematical Research Letters, 2001, 10.4310/MRL.2001.v8.n2.a8

\bibitem{Joh}  K. Johnson, \textit{The MIT Bag model}. Acta Phys. Pol., B(6):865–892, 1975.

\bibitem{Kat}
T. Kato,
{\it Perturbation theory for linear operators.}
2nd edition. Springer, 1980.

\bibitem{MOP} A.Moroianu, T.Ourmières-Bonafos, K.Pankrashkin, \textit{Dirac operators on hypersurfaces as large mass limits}, 2018.

\bibitem{TS2} T.Schick, \textit{Analysis and Geometry of Boundary-Manifolds of Bounded Geometry}. arXiv Mathematics e-prints, 1998.

\bibitem{Wei2}
J. Weidmann (1980),
{\it Monotone continuity of the spectral resolution and the eigenvalues.} Proceedings of the Royal Society of Edinburgh: Section A Mathematics, 85(1-2), 131-136. doi:10.1017/S0308210500011744

\end{thebibliography}
\end{document}